\pdfoutput=1
\documentclass{scrartcl}
\usepackage{paralist}
\usepackage[intlimits]{amsmath}
\usepackage{amsthm,amssymb,mathtools,bm,ifpdf,relsize,authblk}
\usepackage{float}
\usepackage{url}
\usepackage[abbrev]{amsrefs}
\usepackage[usenames,dvipsnames]{color}
\usepackage{longtable}
\usepackage{pifont}
\usepackage{longtable}
\usepackage[T1]{fontenc}
\usepackage{fix-cm}
\usepackage{microtype}
\usepackage{tocstyle}
\usepackage{footnote}
\usepackage{hyperref}

\allowdisplaybreaks[1]

\newtheorem{thm}{Theorem}[section]
\newtheorem{lemma}[thm]{Lemma}
\newtheorem{prop}[thm]{Proposition}
\newtheorem{cor}[thm]{Corollary}
\newtheorem{assumption}[thm]{Assumption}
\newtheorem*{prb*}{Uniformity Problem}
\theoremstyle{definition}
\newtheorem{ex}[thm]{Example}

\newtheorem{question}[thm]{Question}
\newtheorem{rem}[thm]{Remark}
\newtheorem*{rem*}{Remark}

\numberwithin{equation}{section}
\allowdisplaybreaks[1]

\renewcommand{\le}{\leqslant}
\renewcommand{\ge}{\geqslant}

\def\emptyset{\varnothing}

\def\emph{}
\DeclareTextFontCommand{\bfemph}{\bf}
\DeclareTextFontCommand{\itemph}{\it}
\def\emph{\bfemph}

\makeatletter
\def\blankfootnote{\xdef\@thefnmark{}\@footnotetext}

\newcommand*{\textlabel}[2]{%
  \edef\@currentlabel{#1}
  \phantomsection
  #1\label{#2}
}
\makeatother

\newcommand{\idx}[1]{\lvert#1\rvert}
\newcommand{\divides}[2]{\ensuremath{ {#1} \mid {#2} }}

\newcommand{\FF}{\mathbf{F}}

\newcommand{\Gl}{\mathfrak{gl}}
\newcommand{\Sl}{\mathfrak{sl}}

\newcommand{\GG}{\ensuremath{\mathbf{G}}}
\newcommand{\G}{\ensuremath{\mathsf{G}}}

\newcommand{\acts}{\ensuremath{\curvearrowright}}

\newcommand{\QQ}{\mathbf{Q}}
\newcommand{\NN}{\mathbf{N}}
\newcommand{\ZZ}{\mathbf{Z}}
\newcommand{\CC}{\mathbf{C}}

\newcommand{\RR}{\mathbf{R}}

\newcommand{\fg}{\ensuremath{\mathfrak g}}

\newcommand{\fm}{\ensuremath{\mathfrak m}}

\newcommand{\fK}{\ensuremath{\mathfrak K}}
\newcommand{\Places}{\ensuremath{\mathcal V}}
\newcommand{\XX}{\ensuremath{\bm X}}
\newcommand{\Zeta}{\ensuremath{\mathsf{Z}}}
\newcommand{\xx}{\ensuremath{\bm x}}

\newcommand{\fp}{\mathfrak{p}}
\newcommand{\fP}{\mathfrak{P}}
\newcommand{\fo}{\mathfrak{o}}
\newcommand{\fO}{\mathfrak{O}}

\newcommand{\cA}{\mathcal{A}}
\newcommand{\sA}{\mathsf{A}}
\newcommand{\sM}{\mathsf{M}}
\newcommand{\sU}{\mathsf{U}}

\newcommand{\sV}{\mathsf{V}}

\newcommand{\cL}{\mathcal{L}}
\newcommand{\sL}{\mathsf{L}}

\newcommand{\cC}{\mathcal{C}}
\newcommand{\cP}{\mathcal{P}}
\newcommand{\cZ}{\mathcal{Z}}

\DeclareMathOperator{\GL}{GL}
\DeclareMathOperator{\End}{End}
\DeclareMathOperator{\Uni}{U}
\DeclareMathOperator{\Mat}{M}
\DeclareMathOperator{\Spec}{Spec}
\DeclareMathOperator{\topo}{top}
\DeclareMathOperator{\adj}{adj}
\DeclareMathOperator{\ad}{ad}
\DeclareMathOperator{\wirr}{\widetilde{irr}}

\newcommand{\Torus}{\mathbf{T}}
\newcommand{\Euler}{\ensuremath{\chi}}
\newcommand{\Orth}{\RR_{\ge 0}}

\DeclareMathOperator{\dd}{d\!}

\newcommand{\normal}{\triangleleft}
\newcommand{\dtimes}{\ensuremath{\,\cdotp}}
\newcommand{\card}[1]{\lvert#1\rvert}

\DeclarePairedDelimiter{\abs}{\lvert}{\rvert}
\DeclarePairedDelimiter{\genfun}{\lvert}{\rvert}

\DeclareMathOperator{\rank}{rk}

\DeclareMathOperator{\gr}{gr}

\newcommand{\LattE}{\texttt{LattE}}
\newcommand{\llb}{\ensuremath{[\![ }}
\newcommand{\rrb}{\ensuremath{]\!] }}

\ifpdf
  \hypersetup{pdftitle={Computing local zeta functions of groups, algebras, and modules},
    pdfauthor={Tobias Rossmann}
  }
\fi
\title{Computing local zeta functions of groups, algebras, and modules}
\author{Tobias Rossmann}
\affil{\small Fakult\"at f\"ur Mathematik, Universit\"at Bielefeld, D-33501
  Bielefeld, Germany}
\date{}

\begin{document}

\maketitle
\thispagestyle{empty}

\vspace*{-4em}
\begin{abstract}
  \small
  We develop a practical method for computing local zeta
  functions of groups, algebras, and modules in fortunate cases.
  Using our method, we obtain a complete classification of generic local
  representation zeta functions associated with unipotent algebraic groups of
  dimension at most six. 
  We also determine the generic local subalgebra zeta functions associated with
  $\Gl_2(\QQ)$.
  Finally, we introduce and compute examples of graded subobject zeta functions.
\end{abstract}

\blankfootnote{\indent{\itshape 2010 Mathematics Subject Classification.}
  11M41, 20F69, 20G30, 20F18, 20C15.
  

  This work is supported by the DFG Priority Programme
  ``Algorithmic and Experimental Methods in Algebra, Geometry and Number
  Theory'' (SPP 1489).}

\section{Introduction}
\label{s:intro}

\paragraph{Zeta functions counting subobjects and representations.}
By considering associated Dirichlet series, various algebraic counting problems
give rise to a \itemph{global zeta function} $\Zeta(s)$ which admits a natural Euler
product factorisation $\Zeta(s) = \prod_p \Zeta_p(s)$ into
\itemph{local zeta functions} $\Zeta_p(s)$ indexed by rational primes~$p$. 
For example, $\Zeta(s)$ could be the Dirichlet series enumerating 
subgroups of finite index within a finitely generated nilpotent group and
$\Zeta_p(s)$ might enumerate those subgroups of $p$-power index only (see
\cite{GSS88});
in the special case of the infinite cyclic group, 
we then recover the classical Euler factorisation $\zeta(s) = \prod_p
1/(1-p^{-s})$ of the Riemann zeta function.

This article is concerned with three types of counting problems and 
associated zeta functions;
all of these problems arose from (and remain closely related to)
enumerative problems for nilpotent groups.
\begin{itemize}
\item
  (\cite{GSS88})
  Enumerate subalgebras of finite additive index of a possibly non-associative
  algebra, e.g.\ a Lie algebra (possibly taking into account an additive grading).
\item
  (\cite{Sol77})
  Enumerate submodules of finite additive index under the action of an integral
  matrix algebra.
\item
  (\cite{Vol10,HMRC15})
  Enumerate twist-isoclasses of finite-dimensional complex representations of a finitely generated
  nilpotent group.
\end{itemize}

\paragraph{Generic local zeta functions.}
Each of the preceding three counting problems provides us with a global zeta function
$\Zeta(s)$ (namely the associated Dirichlet series) and a factorisation
$\Zeta(s) = \prod_p \Zeta_p(s)$ as above. 
The goal of this article is to compute the 
\itemph{generic local zeta functions} $\Zeta_p(s)$ at least in favourable situations---that is, we
seek to simultaneously determine $\Zeta_p(s)$ for almost all $p$ using a single
finite computation.
To see why this is a sensible problem, we first recall some theory.

In the cases of interest to us, each $\Zeta_p(s)$ will be a rational
function in $p^{-s}$ over $\QQ$.
In particular, the task of computing one local zeta function $\Zeta_p(s)$ using
exact arithmetic is well-defined.
Regarding the behaviour of $\Zeta_p(s)$ under variation of~$p$,
in all three cases from above, sophisticated results from $p$-adic
integration imply the existence of schemes $\sV_1,\dotsc,\sV_r$ and rational
functions $W_1,\dotsc,W_r \in \QQ(X,Y)$ such that for almost all primes~$p$,
\begin{equation}
  \label{eq:intro_denef}
  \Zeta_p(s) = \sum_{i=1}^r \# \sV_i(\FF_p) \dtimes W_i(p,p^{-s});
\end{equation}
for more details, see Theorem~\ref{thm:denef_formulae} below.
While constructive proofs of \eqref{eq:intro_denef} are known, they are
generally impractical due to their reliance on resolution of singularities.

\paragraph{Previous work: computing topological zeta functions.}
In \cites{topzeta,topzeta2,unipotent}, the author developed practical methods
for computing so-called topological zeta functions associated with the above
counting problems;
these zeta functions are derived from generic local ones by means of a termwise
limit ``$p \to 1$'' applied to a formula \eqref{eq:intro_denef}.
Due to their reliance on non-degeneracy conditions for associated
families of polynomials, the author's methods for computing topological zeta
functions do not apply in all cases.
However, whenever they are applicable, as we will explain below, they come close
to producing an \itemph{explicit} formula~\eqref{eq:intro_denef}. 

\paragraph{Computing generic local zeta functions.}
In general, we understand the task of computing $\Zeta_p(s)$ for almost~$p$ to 
be the explicit construction of $\sV_i$ and $W_i$ as in \eqref{eq:intro_denef}.
While this seems to be the only adequate general notion of ``computing''
generic local zeta functions, we will often be more ambitious in practice.

\begin{prb*}
  Decide if there exists $W \in \QQ(X,Y)$ such that $\Zeta_p(s) = W(p,p^{-s})$ for almost all primes~$p$;
  in that case, we call $(\Zeta_p(s))_{p\text{ prime}}$ \emph{uniform}.
  Find $W$ if it exists.
\end{prb*}

The term ``uniformity'' is taken from \cite[\S 1.2.4]{dSW08}.
In practice, a weaker, non-con\-struc\-tive form of the Uniformity Problem which
merely asks for the existence of $W$ as above is often easier to solve.
For example, if $\Zeta_p(s)$ is the zeta function enumerating subgroups (or
normal subgroups)
of finite index in the free nilpotent pro-$p$ group of some fixed finite rank
(independent of~$p$) and class~$2$, then $(\Zeta_p(s))_{p\text{
    prime}}$ is shown to be uniform 
in \cite[Thm~2]{GSS88} even though no explicit construction of a rational function $W$ is given.

For many cases of interest, a rational function~$W$ as in the Uniformity Problem
exists, see e.g.\ most examples in~\cite{dSW08}.
However, no conceptual explanation as to why this is so seems to be known beyond
explicit computations.

Woodward~\cite{Woo05} used computer-assisted calculations to
solve the Uniformity Problem for a large number of subalgebra and ideal zeta
functions of nilpotent Lie algebras.
Unfortunately, few details on his computations are available, rendering them
rather difficult to reproduce.

\paragraph{Results.}
While explicit formulae \eqref{eq:intro_denef} have been obtained for specific
examples and even certain infinite families of these,
all known general constructions of $\sV_i$ and $W_i$
as in~\eqref{eq:intro_denef} are impractical.
In full generality, we thus regard the Uniformity Problem as too ambitious a task.
In the present article, we extend the author's work on explicit,
combinatorially defined formulae \eqref{eq:intro_denef} (see
\cite{topzeta,topzeta2,unipotent}) in order to provide practical solutions to
the Uniformity Problem in fortunate cases.
We will also consider computations of generic local zeta functions in cases where no $W$ as above exists.

As the following list illustrates, the method developed here can be used to
compute a substantial number of interesting new examples of generic local zeta functions:
\begin{itemize}
\item
  We completely determine the generic local representation zeta functions
  associated with unipotent algebraic groups of dimension at most six
  (\S\ref{app:reps}, Table~\ref{t:reps6}).
\item
  We compute the generic local subalgebra zeta functions associated with
  $\Gl_2(\QQ)$;
  this constitutes only the second instance (after $\mathfrak{sl}_2(\QQ)$) where
  such zeta functions associated with an insoluble Lie algebra have been
  computed (\S\ref{ss:gl2}, Theorem~\ref{thm:gl2}).
\item
  We compute the generic local submodule zeta functions for the natural action
  of the group of upper unitriangular integral $n\times n$-matrices 
  (or, equivalently, the nilpotent associative algebra of strictly upper
  triangular integral $n\times n$-matrices) for $n \le 5$ (\S\ref{ss:Un}, Theorem~\ref{thm:Un}). 
\item
  We compute the graded subalgebra and ideal zeta functions associated with
  $\QQ$-forms of each of the $26$ ``fundamental graded'' Lie algebras
  of dimension at most six over $\CC$ (\S\ref{app:graded}, Tables~\ref{t:grid}--\ref{t:grsub}).
\end{itemize}

\paragraph{Outline.}
In \S\ref{s:established}, we recall definitions of the subobject and
representation zeta functions of concern to us.
In~\S\ref{s:graded}, as a variation of established subalgebra and ideal zeta
functions, we discuss graded versions of these zeta functions.
In~\S\ref{s:explicit_formulae}, we consider formulae such as
\eqref{eq:intro_denef} both in theory and as provided by the author's previous
work.
Our work on the Uniformity Problem then proceeds in two steps.
First, in \S\ref{s:count}, we consider the symbolic determination of numbers 
such as the $\#\sV_i(\FF_p)$ in \eqref{eq:intro_denef} as a function of $p$.
Thereafter, in~\S\ref{s:ratfun}, we discuss the explicit computation of the
rational functions $W_i$ as provided by \cite{topzeta,topzeta2,unipotent};
a key role will be played by algorithms of
Barvinok et al.\ \cites{Bar94,BP99,BW03} surrounding generating functions of
rational polyhedra.
In~\S\ref{s:reduced}, we consider ``reduced representation zeta functions'' in
the spirit of Evseev's work \cite{Evs09};
while these functions turn out to be trivial, they provide us with a
simple necessary condition for the correctness of calculations.
Finally, examples of generic local zeta functions are the subject of
\S\S\ref{app:reps}--\ref{app:graded}.

\subsection*{Acknowledgement}
I would like to thank Christopher Voll for interesting discussions.

\subsection*{\textnormal{\textit{Notation}}}

The symbol ``$\subset$'' indicates not necessarily proper inclusion.
For the remainder of this article, let $k$ be a number field with ring of
integers $\fo$.
We write $\Places_k$ for the set of non-Archimedean places of $k$.
For $v\in \Places_k$, we denote by $k_v$ the $v$-adic completion of $k$ and
by $\fo_v$ the valuation ring of $k_v$.
We further let $\fp_v \in \Spec(\fo)$ denote the prime ideal corresponding to $v \in
\Places_k$
and write $q_v = \card{\fo/\fp_v}$.
Finally, we let $\abs{\dtimes}_v$ denote the absolute value on $k_v$
with $\abs{\pi}_v = q_v^{-1}$ for $\pi \in \fp_v^{\phantom 1}\!\setminus\fp_v^2$.

We let $\QQ_p$ and $\ZZ_p$ denote the field of $p$-adic numbers and ring of
$p$-adic integers, respectively.
By a $p$-adic field, we mean a finite extension of $\QQ_p$.
For a $p$-adic field $K$, let $\fO_K$ denote the valuation ring of $K$ and let
$\fP_K$ denote the maximal ideal of $\fO_K$.
We write $q_K = \card{\fO_K/\fP_K}$.

\section{Established zeta functions of groups, algebras, and modules}
\label{s:established}

\subsection{Subalgebra and ideal zeta functions}
\label{ss:subalgebras}

Following \cite{GSS88} (cf.~\cite[\S 2.1]{topzeta}),
for a commutative ring $R$ and a (possibly non-associative) $R$-algebra $\sA$,
we formally define the \emph{subalgebra zeta function} of $\sA$ to be 
\[
\zeta_{\sA}^\le (s) = \sum_{\sU} \idx{\sA:\sU}^{-s},
\]
where $\sU$ ranges over the $R$-subalgebras of $\sA$ such that the $R$-module
quotient $\sA/\sU$ has finite cardinality $\idx{\sA:\sU}$.
Additional hypotheses (which are satisfied in our applications
below) ensure that the number $a_n(\sA)$ of $R$-subalgebras of index $n$ of
$\sA$ is finite for
every $n \ge 1$ and, in addition, $a_n(\sA)$ grows at most polynomially as a function of $n$.
Under these assumptions, $\zeta_{\sA}^\le(s)$  defines an analytic function
in some complex right half-plane.

Now let $\cA$ be a finite-dimensional possibly non-associative $k$-algebra,
where $k$ is a number field as above.
Choose an $\fo$-form $\sA$ of $\cA$ whose underlying $\fo$-module is free.
For $v\in \Places_k$, let $\sA_v := \sA \otimes_{\fo} \fo_v$, regarded as an
$\fo_v$-algebra.
We then have an Euler product
$\zeta_{\sA}^\le(s) = \prod_{v\in \Places_k} \zeta_{\sA_v}^\le(s)$;
see \cite[Lem.~2.3]{topzeta}.
While the \emph{global zeta function} $\zeta_{\sA}^\le(s)$ is an analytic object, as
we will recall below, the \emph{local zeta functions} $\zeta_{\sA_v}^\le(s)$ are
algebro-geometric in nature.
Note that up to discarding finitely many elements, the
family $\bigl(\zeta_{\sA_v}^\le(s)\bigr)_{v\in \Places_k}$ of local zeta
functions only depends on $\cA$ and not on the $\fo$-form $\sA$.

If, instead of enumerating subalgebras, we consider ideals, we obtain
the global and local \emph{ideal zeta functions} $\zeta_{\sA}^\normal(s)$ and
$\zeta_{\sA_v}^\normal(s)$ of $\sA$, respectively;
these are also linked by an Euler product as above.

\subsection{Submodule zeta functions}
\label{ss:submodules}

Submodule zeta functions were introduced by Solomon~\cite{Sol77} in the
context of semisimple associative algebras.
In the following generality (based upon \cite[\S 2.1]{topzeta}), they 
also generalise ideal zeta functions of algebras.
For a commutative ring $R$, an $R$-module $\sV$, and a set $\mathsf\Omega
\subset \End_R(\sV)$, we formally define the \emph{submodule zeta function} of
$\mathsf\Omega$ acting on $\sV$ to be
\[
\zeta_{\mathsf\Omega \acts \sV}(s) = \sum_{\sU} \idx{\sV:\sU}^{-s},
\]
where $\sU$ ranges over the $\mathsf\Omega$-invariant $R$-submodules of $\sV$ with finite
$R$-module quotients~$\sV/\sU$.
The name ``submodule zeta function'' is justified by the observation that we are
free to replace $\mathsf\Omega$ by its enveloping unital associative algebra within $\End_R(\sV)$.

Let $V$ be a finite-dimensional vector space over $k$ and let $\Omega
\subset \End_k(V)$ be given.
Choose an $\fo$-form $\sV$ of $V$ which is free as an $\fo$-module.
Furthermore, choose a finite set $\mathsf \Omega \subset \End_{\fo}(\sV)$ which
generates the same unital subalgebra of $\End_k(V)$ as $\Omega$.
Writing $\sV_v = \sV \otimes_{\fo}\fo_v$,
we obtain an Euler product
$\zeta_{\mathsf\Omega \acts \sV}(s) = \prod_{v\in \Places_k} \zeta_{\mathsf \Omega \acts \sV_v}(s)$;
as in \S\ref{ss:subalgebras}, up to discarding finitely many factors, the collection of local zeta
functions on the right-hand side of this product only depends on $(\Omega, V)$
and not on the choice of $(\mathsf \Omega,\sV)$.

\subsection{Representation zeta functions associated with unipotent groups}
\label{ss:reps}

Following \cite{SV14,HMRC15},
for a topological group $G$, we let $\tilde r_n(G)$ denote the number of
continuous irreducible representations $G \to \GL_n(\CC)$ counted up to
equivalence and tensoring with continuous $1$-dimensional complex
representations.
We formally define the \emph{(twist) representation zeta function} of $G$ to be
\[
\zeta_G^{\wirr}(s) = \sum_{n=1}^\infty \tilde r_n(G) n^{-s}.
\]

Let $\GG$ be a unipotent algebraic group over $k$;
see \cite[Ch.~IV]{DG70} for background.
Let $\Uni_n$ denote the group scheme of upper unitriangular
$n\times n$-matrices.
We choose an embedding of $\GG$ into some $\Uni_n \otimes k$ and let $\G \le \Uni_n
\otimes \fo$ be the associated $\fo$-form of $\GG$
(viz.\ the scheme-theoretic closure of $\GG$ within $\Uni_n\otimes \fo$).
By \cite[Prop.\ 2.2]{SV14},
the Euler product $\zeta_{\G(\fo)}^{\wirr}(s) = \prod_{v\in \Places_k}
\zeta_{\G(\fo_v)}^{\wirr}(s)$ connects the representation zeta function of the
discrete group $\G(\fo)$ and those of the pro-$p_v$ groups
$\G(\fo_v)$,
where $p_v$ is the rational prime contained in~$\fp_v$.

\subsection{Motivation: zeta functions of nilpotent groups}
\label{ss:groups}

We briefly recall the original motivation for the study of subalgebra and ideal
zeta function from \cite{GSS88} and representation zeta functions in
\cite{Vol10,HMRC15} (cf.~\cite{SV14}).
For any topological group~$G$, the \emph{subgroup zeta function}
$\zeta_G^\le(s)$ (resp.\ the \emph{normal subgroup zeta function}
$\zeta_G^\normal(s)$) of $G$ is formally defined to be $\sum_H \idx{G:H}^{-s}$,
where $H$ ranges of the closed subgroups (resp.\  closed normal subgroups) of
$G$ of finite index.
Let $G$ be a discrete torsion-free finitely generated nilpotent group.
Then $\zeta_G^\le(s) = \prod_p \zeta_{\hat G_p}^\le(s)$, where $p$ ranges over
primes and $\hat G_p$ denotes the pro-$p$ completion of $G$. 
Moreover, the global and local zeta functions $\zeta_G^\le(s)$ and $\zeta_{\hat
  G_p}^\le(s)$ all converge in some complex right half-plane.
Analogous statements hold for the normal subgroup and representation zeta
functions of $G$.

Apart from finitely many exceptions, the local subobject and representation zeta
functions attached to $G$ are special cases of those in
\S\S\ref{ss:subalgebras}--\ref{ss:submodules}.
Recall that the Mal'cev correspondence attaches a finite-dimensional nilpotent
Lie $\QQ$-algebra,  $\cL$ say, to~$G$.
As explained in \cite{GSS88},
if $\sL$ is a $\ZZ$-form of $\cL$ which is
finitely generated as a $\ZZ$-module, then
$\zeta_{\hat G_p}^\le(s) = \zeta_{\sL \otimes \ZZ_p}^\le(s)$ and $\zeta_{\hat G_p}^\normal(s)
= \zeta_{\sL \otimes \ZZ_p}^\normal(s)$ for almost all~$p$.
Moreover, if $\GG$ is the unipotent algebraic group over $\QQ$ with Lie algebra $\cL$ and
if $\G$ is a $\ZZ$-form of $\GG$ arising from an embedding $\GG \le \Uni_n \otimes \QQ$,
then $\hat G_p = \G(\ZZ_p)$ for almost all primes~$p$ (see \cite{SV14}).

\section{Graded subalgebra and ideal zeta functions}
\label{s:graded}

In this section, we introduce variations of the subalgebra and ideal zeta functions from
\S\ref{ss:subalgebras} which take into account a given additive grading of the
algebra under consideration.

\subsection{Definitions}
\label{ss:graded}

Let $R$ be a commutative ring and let $\sA$ be a possibly non-associative
$R$-algebra.
Further suppose that we are given a direct sum decomposition
\begin{equation}
  \label{eq:decomp}
  \sA = \sA_1 \oplus \dotsb \oplus \sA_r
\end{equation}
of $R$-modules.
As usual, an $R$-submodule $\sU \le \sA$ is \emph{homogeneous} if it
decomposes as $\sU = \sU_1 \oplus \dotsb \oplus \sU_r$ for $R$-submodules $\sU_i
\le \sA_i$ for $i = 1,\dotsc,r$.
We formally define the \emph{graded subalgebra zeta function} of $\sA$ with
respect to the decomposition \eqref{eq:decomp} to be
\[
\zeta_{\sA}^{\gr\le}(s) = \sum_{\sU} \idx{\sA:\sU}^{-s},
\]
where $\sU$ ranges over the homogeneous $R$-subalgebras of $\sA$ such that
the $R$-module quotient $\sA/\sU$ is finite.
We also define the \emph{graded ideal zeta function} $\zeta_{\sA}^{\gr
  \normal}(s)$ in the evident way.
Note that we do not require \eqref{eq:decomp} to be compatible with the given
multiplication in $\sA$.
As in the non-graded context, given a finite-dimensional possibly
non-associative $k$-algebra $\cA$ together with a vector space decomposition
$\cA = \cA_1\oplus \dotsb \oplus \cA_r$,
we obtain associated global and local graded subalgebra and ideal zeta functions
generalising those from \S\ref{ss:subalgebras} by choosing appropriate $\fo$-forms.

\begin{ex}
  Let $\sA = \ZZ^{n_1} \oplus \dotsb \oplus \ZZ^{n_r}$ be regarded as an abelian
  Lie $\ZZ$-algebra for $n_1, \dotsc, n_r \ge 1$.
  It follows from the well-known non-graded case ($r = 1$; see \cite[Prop.~1.1]{GSS88})
  that $\zeta_{\sA}^{\gr\le}(s) = \prod_{i=1}^r \prod_{j=0}^{n_i-1} \zeta(s-j)$, 
  where $\zeta$ denotes the Riemann zeta function.
\end{ex}

\begin{rem}
  Let $R$, $\sV$, and $\Omega \subset \End_R(\sV)$ be as in
  \S\ref{ss:submodules}.
  Fix an $R$-module decomposition $\sV = \sV_1 \oplus \dotsb \oplus \sV_r$.
  In analogy to the above, we define the \emph{graded submodule zeta
    function} $\zeta_{\Omega \acts \sV}^{\gr}(s)$ of $\Omega$ by enumerating
  homogeneous $\Omega$-invariant $R$-submodules of $\sV$.
\end{rem}

\subsection{Reminder: graded Lie algebras}
\label{s:graded_Lie}

Let $R$ be a commutative Noetherian ring.
All Lie $R$-algebras in the following are assumed to be finitely generated as
$R$-modules.
Recall that an \emph{($\NN$-)graded Lie algebra} over $R$ is a Lie
$R$-algebra $\fg$ together with a decomposition $\fg = \bigoplus_{i=1}^\infty
\fg_i$ into $R$-submodules $\fg_i \le \fg$ such that $[\fg_i,\fg_j] \le
\fg_{i+j}$ for all $i,j \ge 1$.
Since $\fg$ is Noetherian as an $R$-module, $\fg_i = 0$ for sufficiently large $i$
whence such an algebra $\fg$ is nilpotent.
Following \cite[\S 2, Def.\ 1]{Kuz99}, we say that $\fg$ is \emph{fundamental}
if $[\fg_1,\fg_i] = \fg_{i+1}$ for all $i \ge 1$.
If $R = \RR$ or $R = \CC$, then the fundamental graded Lie $R$-algebras of
dimension at most~$7$ have been classified in \cite{Kuz99}.
In the case of dimension at most $5$, the classification in \cite{Kuz99} 
is in fact valid over any field of characteristic zero; see \cite[\S 2.2, Rem.\ 1]{Kuz99}.

Let $\fg$ be a finite-dimensional Lie algebra over a field.
Let $\fg = \fg^1 \supset \fg^2 \supset \dotsb$ be the lower central series
of~$\fg$. 
As is well-known, commutation in $\fg$ endows $\gr(\fg) := \bigoplus_{i=1}^\infty
\fg^i/\fg^{i+1}$ with the structure of a graded Lie algebra;
note that $\gr(\fg)$ is fundamental by construction. 
We call $\gr(\fg)$ the \emph{graded Lie algebra associated with $\fg$}.

The study of graded zeta functions seems quite natural in the context of
nilpotent Lie algebras.
It would be interesting to find group-theoretic interpretations,
in the spirit of~\S\ref{ss:groups}, of such zeta functions associated with
graded nilpotent Lie algebras.

\subsection{Graded subobject zeta functions as $p$-adic integrals}

In order to carry out explicit computations of local graded subobject zeta
functions, we will use the following straightforward variation of \cite[\S
5]{dSG00}; we only spell out the enumeration of graded subalgebras, the case of
ideals being analogous.

\begin{thm}
  \label{thm:graded_int}
  Let $\fO$ be the valuation ring of a non-Archimedean local field.
  Let $\sA$ be a (possibly non-associative) $\fO$-algebra whose underlying
  $\fO$-module is free with basis $\bm a = (a_1,\dotsc, a_d)$.
  Let $0 = \beta_1 < \dotsb < \beta_{r+1} = d$ and
  decompose $\sA = \sA_1 \oplus \dotsb \oplus \sA_r$
  by setting $\sA_i = \fO a_{1+\beta_{i}} \oplus \dotsb \oplus \fO
  a_{\beta_{i+1}}$.
  
  Let $T$ denote the $\fO$-module of block diagonal upper triangular $d \times
  d$-matrices over $\fO$ with block sizes $\beta_2-\beta_1, \dotsc, \beta_{r+1} -
  \beta_{r}$.
  Let $M(\XX)$  be the generic matrix of the same shape over $\fO$;
  in other words,
  \[
  M(\XX) = 
  \mathrm{diag}\left(
  \begin{bmatrix}
    X_{1,1} & \hdots & X_{1,\beta_2}\\
    & \ddots & \vdots \\
    & & X_{\beta_2,\beta_2}
  \end{bmatrix},
  \dotsc,
  \begin{bmatrix}
    X_{1+\beta_r,1+\beta_r} & \hdots & X_{1+\beta_r,d}\\
    & \ddots & \vdots\\
    & & X_{d,d}
  \end{bmatrix}
  \right).
  \]

  Let $R = \fO[\XX]$ 
  and let $\star\colon R^d \times R^d \to R^d$ be
  induced via base extension by multiplication in $\sA$ with respect to $\bm a$.
  Let $F \subset R$ consist of all entries of all $d$-tuples
  $(M_i(\XX) \star M_j(\XX))
  \adj(M(\XX))$ for $1\le i,j\le d$, where $\adj(M(\XX))$ denotes the adjugate
  matrix of $M(\XX)$
  and $M_i(\XX)$  the $i$th row of $M(\XX)$.
  Define $V = \{ \xx \in T : \divides{ \det(M(\xx)) }{ f(\xx) } \text{ for all }
  f \in F\}$.
  Let $q$ denote the residue field size of $\fO$, let $\mu$ denote the
  normalised Haar measure on $T \approx \fO^{\sum_{i=1}^r
    \binom{\beta_{i+1}-\beta_i + 1}2}$, and let $\abs{\,\cdotp}$ denote the absolute
  value on $K$ such that $\abs{\pi} = q^{-1}$ for any uniformiser $\pi$.
  Then
  \begin{equation}
    \label{eq:graded_int}
  \zeta_{\sA}^{\gr\le}(s) = (1-q^{-1})^{-d} \int_V 
  \prod_{i=1}^r \prod_{j=1}^{\beta_{i+1}-\beta_i}
  \abs{ x_{j+\beta_i,j+\beta_i}}^{s-j} \dd\mu(\xx). \qed
  \end{equation}
\end{thm}

\begin{rem}
  \label{rem:cone_conditions}
  As in \cite[\S 5]{dSG00},
  a matrix $\xx \in T$ belongs to
  the set $V$ in Theorem~\ref{thm:graded_int}
  if and only if its row span is a subalgebra of~$\fO^d$,
  regarded as an algebra via the given identification $\sA = \fO^d$.
\end{rem}

The following illustrates Theorem~\ref{thm:graded_int} for an infinite family of
graded algebras.

\begin{prop}
  \label{prop:maximal_class}
  Let $n \ge 1$ and
  let $\fm(n) = \fm_1(n) \oplus \dotsb \oplus \fm_n(n)$ be the graded Lie
  $\ZZ$-algebra of additive rank $n+1$ and nilpotency class $n$ with $\fm_1(n) = \ZZ e_0 \oplus \ZZ e_1$, $\fm_i(n) = \ZZ
  e_i$ for $i = 2,\dotsc,n$, and non-trivial commutators $[e_0,e_i] =
  e_{i+1}$ for $1 \le i \le n-1$.
  Let $k$ be a number field with ring of integers $\fo$.
  Then for each $v\in \Places_k$, 
  \[
  \zeta_{\fm(n)\otimes \fo_v}^{\gr\normal}(s) =
  1/\bigl((1-q_v^{-s})(1-q_v^{1-s}) (1-q_v^{-3s}) (1-q_v^{-4s}) \dotsb
  (1-q_v^{-(n+1)s})\bigr),
  \]
  where $\fm(n)\otimes\fo_v$ is regarded as an $\fo_v$-algebra.
  Denoting the Dedekind zeta function of $k$ by $\zeta_k(s)$, we thus have
  $$\zeta_{\fm(n)\otimes \fo}^{\gr\normal}(s) =
  \zeta_k(s)\zeta_k(s-1)\zeta_k(3s)\zeta_k(4s)\dotsb\zeta_k\bigl((n+1)s\bigr).$$
\end{prop}
\begin{proof}
  It is an elementary consequence of Theorem~\ref{thm:graded_int} and
  Remark~\ref{rem:cone_conditions} (both applied to the enumeration of ideals
  instead of subalgebras) 
  that for any $v\in \Places_k$,
  \begin{align*}
  \zeta_{\fm(n)\otimes\fo_v}^{\gr\normal}(s) =\, & (1-q_v^{-1})^{-n-1} 
  \\ & \times \int_V
  \abs{x_{1}}_v^{s-1} \abs{x_{3}}_v^{s-2} \abs{y_1}_v^{s-1} \dotsb
  \abs{y_{n-1}}_v^{s-1} \dd\mu(x_1,x_2,x_3,y_1,\dotsc,y_{n-1}),
  \end{align*}
  where
  $
  V = \bigl\{ (x_1,x_2,x_3,y_1,\dotsc,y_{n-1}) \in \fo_v^{n+2} : 
  y_{n-1} \mid y_{n} \mid \dotsb \mid y_{1} \mid x_1,x_2,x_3 \bigr\}$;
  indeed, $(x_1,\dotsc,y_{n-1}) \in V$ if and only if the row span
  of $\mathrm{diag}( \bigl[\begin{smallmatrix} x_1 & x_2 \\ &  x_3\end{smallmatrix}\bigr],y_1,\dotsc,y_{n-1})$
  is an ideal of $\fm(n)\otimes \fo_v$ (identified with $\fo_v^{n+1}$ via $(e_0,\dotsc,e_n)$).
  Define a bianalytic bijection
  \begin{align*}
  \varphi\colon (k_v^\times)^{n+2} \to (k_v^\times)^{n+2},
  \quad
  (x_1,x_2,x_3,y_1,\dotsc,y_{n-1}) \mapsto
  (&x_{1} y_1\dotsb y_{n-1}, \,\,
  x_2 y_1\dotsb y_{n-1}, \\&
  x_3 y_1 \dotsb y_{n-1}, \\& 
  y_1\dotsb y_{n-1}, \\&
  y_2\dotsb y_{n-1}, \dotsc, y_{n-1});
  \end{align*}
  note that the Jacobian determinant of $\varphi$ is
  $\det \varphi'(x_1,x_2,x_3,y_1,\dotsc,y_{n-1}) = y_1^3 y_2^4 \dotsb 
  y_{n-1}^{n+1}$.
  Since $V \cap (k_v^\times)^{n+2} = \varphi(\fo_v^{n+2} \cap
  (k_v^\times)^{n+2})$
  and $\mu(k_v^n\setminus (k_v^\times)^n) = 0$,
  by performing a change of variables using $\varphi$
  and using the well-known fact $\int_{\fo_v} \abs{z}_v^s \dd\mu(z) = (1-q_v^{-1})/(1-q_v^{-1-s})$,
  \begin{align*}
  \zeta_{\fm(n)\otimes\fo_v}^{\gr\normal}(s)
  & = (1-q_v^{-1})^{-n-1} \int_{\fo_v^{n+2}}
  \abs{x_1}^{s-1}_v \abs{x_3}^{s-2}_v  \abs{y_1}^{3s-1}_v
  \dotsb \abs{y_{n-1}}^{(n+1)s-1}_v \dd\mu(x_1,\dotsc,y_{n-1}) \\
  & = 1/\bigl((1-q_v^{-s})(1-q_v^{1-s}) (1-q_v^{-3s}) (1-q_v^{-4s}) \dotsb (1-q_v^{-(n+1)s})\bigr).
\end{align*}
The final claim follows by taking the product over all $v \in \Places_k$.
\end{proof}

\begin{rem*}
  To the author's knowledge, not a single example of a non-graded
  subobject zeta function of a nilpotent Lie  algebra of nilpotency class $\ge 5$ is
  known explicitly.
\end{rem*}

Integrals such as those in \eqref{eq:graded_int} are  special cases of those
associated with ``toric data'' in~\cite[\S 3]{topzeta2}.
Hence, the author's methods for manipulating such integrals as developed
in~\cite{topzeta2} apply directly without modification, as do the techniques  
explained below.

\section{Explicit formulae}
\label{s:explicit_formulae}

\subsection{Theory: local zeta functions of Denef type}

The following is a variation of the terminology employed in \cite[\S
5.2]{topzeta}.
As before, we assume that $k$ is a fixed number field.
Suppose that we are given a collection $\Zeta = (\Zeta_K(s))_K$ of analytic functions
of a complex variable $s$ (each defined in some right half-plane)
indexed by $p$-adic fields $K \supset k$ (up to
$k$-isomorphism). 
We say that $\Zeta$ is of \emph{Denef type} if there exist a finite set
$S\subset \Places_k$, $k$-varieties $V_1,\dotsc,V_r$, 
and rational functions $W_1,\dotsc,W_r \in \QQ(X,Y)$ such
that for all $v\in \Places_k\setminus S$ and all finite extensions $K/k_v$,
\begin{equation}
  \label{eq:main_denef}
  \Zeta_K(s) = \sum_{i=1}^r \# \bar V_i(\fO_K/\fP_K) \dtimes W_i(q_K^{\phantom {-s}}\!\!,q_K^{-s})
\end{equation}
is an identity of analytic functions;
here, we wrote $\bar V_i = \sV_i \otimes_{\fo} \fo/\fp_v$ for a fixed but
arbitrary $\fo$-model $\sV_i$ of $V_i$.

The following result formalises our discussion surrounding
\eqref{eq:intro_denef} from the introduction;
it summarises \cite[\S\S 2--3]{dSG00} (cf.~\cite[Thm~5.16]{topzeta})
and \cite[Thm~A]{SV14}.

\begin{thm}
  \label{thm:denef_formulae}
  Let $\bigl(\Zeta_K(s)\bigr)_K$ be one of the following collections of local
  zeta functions indexed by $p$-adic fields $K\supset k$ (up to $k$-isomorphism).
  \begin{enumerate}
  \item
    $\Zeta_K(s) = \zeta_{\sA\otimes_{\fo}\fO_K}^\le(s)$ or $\Zeta_K(s) =
    \zeta_{\sA\otimes_{\fo}\fO_K}^\normal(s)$
    (resp.\ $\Zeta_K(s) = \zeta_{\sA\otimes_{\fo}\fO_K}^{\gr\le}(s)$
    or $\Zeta_K(s) = \zeta_{\sA\otimes_{\fo}\fO_K}^{\gr\normal}(s)$),
    where $\sA$ is an $\fo$-form of a finite-dimensional (possibly
    non-associative) $k$-algebra as in \S\ref{ss:subalgebras} or
    \S\ref{ss:graded}, respectively.
  \item
    $\Zeta_K(s) = \zeta_{\mathsf \Omega \acts (\sV \otimes_{\fo}{\fO_K})}(s)$, where $\mathsf\Omega$
    and $\sV$ are as in \S\ref{ss:submodules}.
  \item
    $\Zeta_K(s) = \zeta_{\G(\fO_K)}^{\wirr}(s)$, where $\G$ is an $\fo$-form of a
    unipotent algebraic group over $k$ as in \S\ref{ss:reps}.
  \end{enumerate}
  Then $\bigl(\Zeta_K(s)\bigr)_K$ is of Denef type.
\end{thm}

The known proofs of Theorem~\ref{thm:denef_formulae} are constructive
but impractical due to their reliance on resolution of singularities.
We note that the exclusion of finitely many primes implicit in
Theorem~\ref{thm:denef_formulae} is one of the main reasons for our focus on
\itemph{generic} local zeta functions.

\subsection{By-products of the computation of topological zeta functions}

The computation of topological zeta functions is often considerably easier than
that of local ones.
In~\cite{topzeta,topzeta2,unipotent}, the author developed
practical methods for computing topological zeta functions associated with the
local zeta functions in Theorem~\ref{thm:denef_formulae};
these methods are not algorithms because they may fail if certain
non-degeneracy conditions are violated.
From now on, we will assume the validity of the following.

\begin{assumption}
  \label{A1}
  In the setting of Theorem~\ref{thm:denef_formulae},
  the method from \cite[\S 4]{topzeta2} (resp.~\cite[\S 5.4]{unipotent})
  for computing topological subalgebra and submodule zeta
  functions (resp.\ topological representation zeta functions) 
  succeeds.
\end{assumption}

\begin{rem}
  The author is unaware of a useful intrinsic characterisation of those groups,
  algebras, and modules such that Assumption~\ref{A1} is satisfied.
  The local zeta functions in Theorem~\ref{thm:denef_formulae} can be described
  in terms of $p$-adic integrals associated with a collection of polynomials.
  A sufficient condition for the validity of Assumption~\ref{A1} is
  ``non-degeneracy'' of said collection of polynomials in the sense of \cite[\S
  4.2]{topzeta}; cf.\ \cite[Lem.~5.7]{topzeta2} and \cite[\S 5.4.1]{unipotent}.
\end{rem}

The first stages of the methods for computing topological zeta functions
associated with the local zeta functions in Theorem~\ref{thm:denef_formulae},
as described in \cite{topzeta2,unipotent}, come close to constructing 
an explicit formula~\eqref{eq:main_denef}.
In detail, using \cite[Thm~4.10]{topzeta}
(see \cite[Thm~5.8]{topzeta2} and \cite[Thm~5.9]{unipotent}),
whenever they succeed,
these methods derive a formula~\eqref{eq:main_denef} such that the following two
assumptions are satisfied.

\begin{assumption}
  \label{A2}
  The $V_i$ in \eqref{eq:main_denef} are given as explicit subvarieties of
  algebraic tori over~$k$, defined by the vanishing of a finite number of Laurent
  polynomials and the non-vanishing of a single Laurent polynomial.
\end{assumption}

\begin{assumption}
  \label{A3}
  Up to multiplication by explicitly given rational functions of the form $(X-1)^aX^b$
  (for suitable $a,b\in \ZZ$),
  each $W_i$ in \eqref{eq:main_denef} is described explicitly 
  in terms of generating functions associated with half-open cones and convex polytopes.
\end{assumption}

We will clarify the deliberately vague formulation of Assumption~\ref{A3} in
\S\ref{s:ratfun}.\\

In summary, whenever they apply, the methods for computing topological zeta
functions in \cite{topzeta2,unipotent} fall short of ``constructing'' an
explicit formula \eqref{eq:main_denef} only in the sense that the $W_i$ are
characterised combinatorially instead of being explicitly given, say as
fractions of polynomials.

In the following sections,
assuming the validity of Assumptions~\ref{A1}--\ref{A3}, we will develop techniques
for performing further computations with a formula of the form
\eqref{eq:main_denef} with a view towards solving the Uniformity Problem from
the introduction in fortunate cases.

\section{Counting rational points on subvarieties of tori}
\label{s:count}

Assuming the validity of Assumption~\ref{A2},
this section is devoted to ``computing'' the numbers $\# \bar V_i(\fO_K/\fP_K)$
in~\eqref{eq:main_denef}. 
Using the inclusion-exclusion principle, we may reduce to the case that
the $V_i$ are all \itemph{closed} subvarieties of algebraic tori over $k$.
Note that the non-constructive version of the Uniformity Problem from
\S\ref{s:intro} has a positive solution whenever each $\# \bar
V_i(\fO_K/\fP_K)$ is a polynomial in $q_K$ (after excluding finitely
many places of~$k$).
The following method is based on the heuristic observation that the latter
condition is often satisfied for examples of interest.

\paragraph{Setup.}
Let $\Torus^n := \Spec(\ZZ[X_1^{\pm 1},\dotsc,X_n^{\pm 1}])$ and, for a
commutative ring $R$, write $\Torus^n_R := \Torus^n \otimes R$.
For a finite set $S\subset \Places_k$,
let $\fo_S = \{ x\in k : x \in \fo_v \text{ for all } v  \in
\Places_k\setminus S \}$
denote the usual ring of $S$-integers of $k$.
For $f_1,\dotsc,f_r \in \fo_S[X_1^{\pm 1},\dotsc,X_n^{\pm 1}]$,
define
\[
(f_1,\dotsc,f_r)^n_S := \Spec(\fo_S[X_1^{\pm 1},\dotsc,X_n^{\pm 1}]/\langle
f_1,\dotsc,f_r\rangle)
\subset \Torus^n_{\fo_S}.
\]
For $v \in \Places_k\setminus S$ and a finite extension $\fK$ of $\fo/\fp_v$,
let
$\abs{f_1,\dotsc,f_r}^n_{\fK}$ denote the number of $\fK$-rational points of
$(f_1,\dotsc,f_r)^n_S$.

\paragraph{Objective: symbolic enumeration.}
From now on, let $f_1,\dotsc,f_r \in \fo_S[X_1^{\pm 1},\dotsc,X_n^{\pm 1}]$ 
be given as above.
Our goal in the following is to symbolically ``compute'' the numbers
$\abs{f_1,\dotsc,f_r}^n_{\fK}$ as a function of $\fK$.
More precisely, 
the procedure described below constructs a polynomial, $H(X,c_1,\dotsc,c_\ell)$
say, over $\ZZ$ such that, after possibly enlarging~$S$, for all $v\in
\Places_k\setminus S$ and all finite extensions $\fK$ of $\fo/\fp_v$,
\[
(f_1,\dotsc,f_r)^n_{\fK} = H(\card \fK, \# \sU_1(\fK),\dotsc, \# \sU_\ell(\fK)),
\]
where each $\sU_i$ is an explicitly given closed subscheme of some $\Torus^{n_i}_{\fo_S}$.
We could of course simply take $H = c_1$ and $\sU_1 = (f_1,\dotsc,f_r)^n_S$ but
we seek to do better.
Indeed, in many cases of interest,
$H$ can be taken to be a polynomial in $X$ only.
In the following, we describe a method which has proven to be quite useful for
handling such cases.

\paragraph{Dimension $\le 1$.}
We first describe two base cases of our method.
Namely, if $n = 0$, then, after possibly enlarging $S$,
$(f_1,\dotsc,f_r)^n_S$ is either $\emptyset$ or $\Torus^0_{\fo_S} = \Spec(\fo_S)$ depending on
whether some $f_i \not= 0$ or not; thus, $\abs{f_1,\dotsc,f_r}^n_{\fK}
\in \{ 0,1\}$ for $\fK$ as above.

Secondly, if $n = 1$, then we use the Euclidean algorithm over $k$ (thus
possibly enlarging~$S$) to compute a single
square-free polynomial $f \in \fo_S[X_1]$
such that $(f_1,\dotsc,f_r)^1_S = (f)^1_S$.
If $f$ splits completely over $k$, then, after possibly enlarging~$S$ once again,
$\abs{f}^1_{\fK} = \deg(f)$ for all $\fK$ as above.
If $f$ does not split completely over $k$, then we introduce a new variable, $c_f$ say,
corresponding to the number of solutions of $f = 0$ in~$\fK^\times$.

\paragraph{Simplification.}
It is often useful to ``simplify'' the given Laurent polynomials $f_1,\dotsc,f_r$;
while this step was sketched in \cite[\S 6.6]{topzeta2}, here we provide some
further details.
As before, the set $S$ may need to be enlarged at various points in the
following.
First, we discard any zero polynomials among the $f_i$.
We then clear denominators so that each $f_i \in k[X_1,\dotsc,X_n]$
is an actual (not just Laurent) polynomial.
Next, we replace each $f_i$ by its square-free part in $k[X_1,\dotsc,X_n]$. 
For each pair $(i,j)$ of distinct indices, we then compute the (square-free part of
the) remainder,  $r$ say, of $f_i$ after multivariate polynomial division by
$f_j$ with respect to some term order (see e.g.\ \cite[\S 1.5]{AL94}).
If $r$ consists of fewer terms than $f_i$, we replace $f_i$ by $r$.
Next, for each pair $(i,j)$ as above and each term $t_i$ of $f_i$ and $t_j$ of
$f_j$, we are free to replace $f_i$ by (the square-free part of)
$\frac{t_j}g f_i - \frac{t_i}g f_j$, where $g = \gcd(t_i,t_j)$ (computed over $k$),
which we again do whenever it reduces the total number of terms.
After finitely many iterations of the above steps, $f_1,\dotsc,f_r$ will
stabilise at which point we conclude the simplification step.

\paragraph{}
We next describe two procedures which,
if applicable, allow us to express $\abs{f_1,\dotsc,f_r}^n_{\fK}$ in terms of
the numbers of rational points of subschemes of lower-dimensional tori. 
We then recursively attempt to solve the symbolic enumeration problem
from above for~these.

\paragraph{Reduction of dimension I: torus factors.}
As explained in \cite[\S 6.3]{topzeta2},
using the natural action of $\GL_n(\ZZ)$ on $\Torus^n$,
a Smith normal form computation allows us to effectively construct
$g_1,\dotsc,g_r\in \fo_S[X_1^{\pm 1},\dotsc,X_d^{\pm 1}]$ and an explicit isomorphism
$(f_1,\dotsc,f_r)^n_S \approx (g_1,\dotsc,g_r)^d_S\times_{\fo_S} \Torus^{n-d}_{\fo_S}$,
where $d$ is the dimension of the Newton polytope of $f_1\dotsb f_r$.
It follows that for all~$\fK$ as above,
$\abs{f_1,\dotsc,f_r}^n_{\fK} = \abs{g_1,\dotsc,g_r}^d_{\fK} \dtimes
(\card \fK-1)^{n-d}$.
In the following, we may thus assume that $n = d$.

\paragraph{Reduction of dimension II: solving for variables.}
Whenever it is applicable, the following lemma allows us to replace the problem
of symbolically computing $\abs{f_1,\dotsc,f_r}^n_{\fK}$ by four instances of
the same problem in dimension $n-1$.

\begin{lemma}
  \label{lem:cell}
  Let $F \subset \fo_S[X_1^{\pm 1},\dotsc,X_{n-1}^{\pm 1}]$.
  Further let $f = u - w X_n$ for non-zero $u, w \in \fo_S[X_1^{\pm
    1},\dotsc,X_{n-1}^{\pm 1}]$.
  Then for all $v\in \Places_k\setminus S$ and all finite extensions $\fK$ of $\fo/\fp_v$,
  \begin{align*}
    \abs{F,f}^n_{\fK} = \abs{F}^{n-1}_{\fK} - \abs{F,u}^{n-1}_{\fK}
    -\abs{F,w}^{n-1}_{\fK} + \card \fK \dtimes \abs{F,u,w}^{n-1}_{\fK}.
  \end{align*}
\end{lemma}
\begin{proof}
  Projection onto the first $n-1$ coordinates induces an
  isomorphism of $\fo_S$-schemes
  $(F,f)^n_S \setminus (F,f,w)^n_S \approx (F)^{n-1}_S \setminus (F,uw)^{n-1}_S$.
  As
  $(F, f, w)^n_S = (F,u,w)^n_S  \approx (F,u,w)^{n-1}_S \times_{\fo_S} \Torus^1_{\fo_S}$,
  the claim follows since for all $v\in \Places_k\setminus S$ and all finite
  extensions  $\fK$ of $\fo/\fp_v$,
  \begin{align*}
    \card{((F)^{n-1}_S \setminus (F,uw)^{n-1}_S)(\fK)}
    & =
    \abs{F}^{n-1}_{\fK} - \abs{F,u}^{n-1}_{\fK} - \abs{F,w}^{n-1}_{\fK}
    + \abs{F,u,w}^{n-1}_{\fK}.
    \qedhere
  \end{align*}
\end{proof}

\begin{rem}
  The evident analogue of Lemma~\ref{lem:cell} for Euler characteristics 
  of closed subvarieties of algebraic tori over $k$ has already been used in
  the author's software package \textsf{Zeta}~\cite{Zeta} for computing
  topological zeta functions.
  However, only the special case that $w \in \fo_S^\times$
  (so that $(F,w)^{n-1}_S = (F,u,w)^{n-1}_S = \emptyset$) was
  spelled out explicitly in \cite[\S 6.6]{topzeta2}.
\end{rem}

\paragraph{Final case.}
Finally, if none of the above techniques
for computing or decomposing $(f_1,\dotsc,f_r)^n_S$ applies,
then we introduce a new variable corresponding to $\abs{f_1,\dotsc,f_r}^n_{\fK}$.
In order to avoid this step whenever possibly, 
we first attempt to apply the above steps (including all possible applications
of Lemma~\ref{lem:cell}) without ever invoking this final case. 

\section{Local zeta functions as sums of rational functions}
\label{s:ratfun}

Suppose that Assumptions~\ref{A1}--\ref{A3} are satisfied.
Our first task in this section is to rewrite~\eqref{eq:main_denef} as a sum of
explicitly given rational functions.
With the method from \S\ref{s:count} at our disposal, 
this problem reduces to finding such an expression for each $W_i$.
We will see that Barvinok's algorithm from convex geometry 
solves this problem.
Our second task then concludes the computation
of the generic local zeta functions in Theorem~\ref{thm:denef_formulae};
it is concerned with adding a potentially large number of multivariate rational
functions.
We describe a method aimed towards improving the practicality of this step which,
while mathematically trivial, often vastly dominates the run-time of our
computations.

\subsection{Barvinok's algorithm: generating functions and substitutions}

Let $\cP \subset \Orth^n$ be a rational polyhedron and
let $\bm \lambda = (\lambda_1,\dotsc,\lambda_n)$ be algebraically independent
over $\QQ$.
It is well-known that the generating function
$\genfun \cP := \sum_{\alpha \in \cP \cap \ZZ^n} \bm\lambda^\alpha$
is rational in the sense that within the field of fractions of $\QQ\llb
\lambda_1,\dotsc,\lambda_n\rrb$, it belongs to $\QQ(\lambda_1,\dotsc,\lambda_n)$.
The standard proof of this fact (see e.g.\ \cite[Ch.\
13]{Bar08}) proceeds by reducing to the case that $\cP$ is a cone, in which case
an explicit formula for $\genfun \cP$ can be derived from a triangulation of
$\cP$ via the inclusion-exclusion principle.
This strategy for computing~$\genfun \cP$ is, however, of rather limited
practical use.

A far more sophisticated approach is given by ``Barvinok's algorithm''; see
\cite{Bar94,BP99}.
Barvinok's algorithm computes
$\genfun{\cP}$ for each (suitably encoded) rational polyhedron $\cP
\subset \Orth^n$ as a sum of rational 
functions of the form $c \bm\lambda^{\alpha_0} /( (1-\bm\lambda^{\alpha_1})\dotsb
(1-\bm\lambda^{\alpha_n}))$ for suitable $\alpha_0,\dotsc,\alpha_n \in \ZZ^n$
and $c \in \QQ$.
For a fixed ambient dimension $n \ge 1$, his algorithm runs in polynomial time
so that $\genfun \cP$ is computed as a short sum of short rational functions in a
precise technical sense.
Beyond its theoretical strength, Barvinok's algorithm is also powerful in practice
as demonstrated by the software implementation \LattE{}~\cite{LattE}. 

In the setting of Assumption~\ref{A3}, we are not primarily interested in
generating functions associated with polyhedra themselves but in rational
functions derived from such generating functions via monomial substitutions.
In detail, let $\bm\xi = (\xi_1,\dotsc,\xi_m)$ be algebraically independent over $\QQ$
and let $\sigma_1,\dotsc,\sigma_n \in \ZZ^m$.
Suppose that $\cP \subset \Orth^n$ is a rational polyhedron such that
$W := \genfun{\cP}(\bm\xi^{\sigma_1},\dotsc,\bm\xi^{\sigma_n})$ is
well-defined on the level of rational functions.
In principle, we could compute $W$ by first using the output of Barvinok's
algorithm in order to write $\genfun{\cP}$ in lowest terms, followed by an
application of the given substitution.
This method is, however, often impractical due to the computational cost of
(multivariate) rational function arithmetic.

A theoretically favourable and also practical alternative
is developed in \cite[\S 2]{BW03} (cf.~\cite[\S 5]{BP99})
There, a polynomial time algorithm is described which takes as input a short
representation of $\genfun\cP$ (as, in particular, provided by 
Barvinok's algorithm) and constructs a similar short representation for $W$.
The important point to note here is that while we assumed the substitution $\lambda_i \mapsto
\bm\xi^{\sigma_i}$ to be valid for $\genfun\cP$ itself, it may be undefined
for some of the summands in the expression provided by Barvinok's algorithm.

\subsection{Computing the $W_i$ in \eqref{eq:main_denef}}

We may now clarify the vague formulation of Assumption~\ref{A3}.
Namely, up to a factor $(X-1)^aX^b$,
the $W_i$ in \eqref{eq:main_denef} are obtained by applying suitable monomial
substitutions (see \cite[Rem.\ 4.12]{topzeta} and \cite[Thm\ 5.5]{unipotent})
to rational functions of the form
$\cZ^{\cC_0,\cP_1,\dotsc,\cP_m}(\xi_0,\dotsc,\xi_m)$
from \cite[Def.\ 3.6]{topzeta}.
The latter functions can, by their definitions, be written as
sums of rational functions obtained by applying suitable monomial substitutions
to generating functions enumerating lattice points inside rational half-open
cones;
as explained in \cite[\S 8.4]{topzeta2}, we may replace these half-open cones by
rational polyhedra. 
We may thus use Barvinok's algorithm as well as the techniques for
efficient monomial substitutions from \cite[\S 2]{BW03} in order to 
write each $W_i$ as a sum of bivariate rational functions of the form
\begin{equation}
  \label{eq:cycrat}
  f(X,Y) / \bigl((1-X^{a_1}Y^{b_1}) \dotsb (1-X^{a_m}Y^{b_m})\bigr)
\end{equation}
for suitable integers $a_i,b_i \in \ZZ$, $m \ge 0$, and $f(X,Y) \in \QQ[X,Y]$.

\subsection{Final summation}

In the following, we allow $f(X,Y)$ in \eqref{eq:cycrat} to
be an element of $\QQ[X,Y,c_1,c_2,\dotsc]$.
By taking into account the polynomials obtained using \S\ref{s:count},
at this point, we may thus assume that we constructed a finite sum of 
expressions \eqref{eq:cycrat} 
such that, after excluding finitely many places of $k$,
the local zeta functions in Theorem~\ref{thm:denef_formulae} are obtained by
specialising $X \mapsto q_K$,
$Y \mapsto q_K^{-s}$, and $c_i \mapsto \# \sU_i(\fO_K/\fP_K)$ for certain
explicit subschemes $\sU_i$ of tori over~$\fo$ (or over $\fo_S$).
All that remains to be done in order to recover the local zeta functions of
interest is to write the given sum of expressions \eqref{eq:cycrat} in lowest
terms.

While our intended applications of Barvinok's algorithm lie well within the
practical scope of \LattE{}~\cite{LattE}, it will often be infeasible to 
pass the rational functions \eqref{eq:cycrat} to a computer algebra system in
order carry out the final summation.
In addition to the sheer number of rational functions to be considered, a key
problem is due to the fact that the number of distinct pairs $(a_i,b_i)$ arising
from summands \eqref{eq:cycrat} often obscures the relatively simple shape of the
final sum (i.e.~the local zeta function to be computed).
This is consistent with the well-known observation (see e.g.~\cite[\S 2.3]{Den91a})
that few candidate poles of local zeta functions as provided by explicit
formulae \eqref{eq:main_denef} survive cancellation.

In order to carry out the final summation, we proceed in two stages.
First, we use an idea due to Woodward~\cite[\S 2.5]{Woo05} and add
and simplify those summands \eqref{eq:cycrat} such that distinguished pairs $1-X^cY^d$
occur in their written denominators;
our hope here is that some rays $(a_i,b_i)$ will be removed via cancellations.
While this step is not essential, it might improve the performance and memory
requirements of the final stage. 
Here, we first construct a common denominator of all the remaining rational
functions~\eqref{eq:cycrat}.
We then compute the final result by summing the \eqref{eq:cycrat} rewritten
over our common denominator, followed by one final division.
In addition to being trivially parallelisable, by only adding numerators,
we largely avoid costly rational function arithmetic.

\subsection{Implementation issues}

The method for computing generic local subobject or representation zeta
functions described above has been implemented (for $k = \QQ$)
by the author as part of his package \textsf{Zeta}~\cite{Zeta} for
Sage~\cite{Sage}.
The program \LattE{}~\cite{LattE} (which implements
Barvinok's algorithm) plays an indispensable role.
Moreover, the computer algebra system Singular~\cite{Singular} features
essentially in the initial stages of our method (as described in~\cite{topzeta2,unipotent}).

The author's implementation is primarily designed to find instances of positive
solutions to the Uniformity Problem; 
its functionality and practicality are both quite restricted in non-uniform
cases.
Furthermore, the author's method supplements Woodward's
approach \cite{Woo05} for computing local (subalgebra and ideal) zeta functions
as well as various ad hoc computations carried out by others without replacing
them. 
In particular, various examples of local zeta functions computed by 
Woodward cannot be reproduced using the present method.
In addition to the theoretical limitations of the techniques from
\cite{topzeta,topzeta2,unipotent}, this is also partially due to practical
obstructions:
while some computations of topological zeta functions in 
\cite{topzeta,topzeta2,unipotent} were already fairly involved, the present
method is orders of magnitude more demanding.

\section{Interlude: reduced representation zeta functions}
\label{s:reduced}

Reduced zeta functions arising from the enumeration of subalgebras and ideals
were introduced by Evseev~\cite{Evs09}.
They constitute a limit ``$p \to 1$'' of suitable local zeta functions
distinct from but related to the topological zeta functions of Denef and
Loeser~\cite{DL92} (which were later adapted to the case of subobject zeta
functions by du~Sautoy and Loeser~\cite{dSL04}).
Informally, Evseev's definition can be summarised as follows in our setting.
Let $\sA$ be an $\fo$-form of a $k$-algebra as in \S\ref{ss:subalgebras}.
For each $v\in \Places_k$, we may regard $\zeta_{\sA\otimes_{\fo}\fo_v}^\le(s)$
as a (rational) formal power series in $Y = q_v^{-s}$.
The reduced subalgebra zeta function of $\sA$ (an invariant of $\sA \otimes_{\fo}
\CC$, in fact)
is obtained by taking a limit ``$q_v \to 1$'' applied to the
coefficients of $\zeta_{\sA\otimes_{\fo}\fo_v}^\le(s)$ as a series in $Y$.
The rigorous definition of reduced zeta function in \cite{Evs09} involves the
motivic subobject zeta functions introduced by du~Sautoy and Loeser~\cite{dSL04}.

In this section, we show that ``reduced representation zeta functions'' 
associated with unipotent groups are always identically $1$.
In addition to imposing restrictions on the shapes of generic local
representation zeta functions of such groups, this fact provides a simple
necessary condition for the correctness of explicit calculations of local zeta
functions such as those documented below.

We begin with a variation of a result from \cite{stability}.
Let $V$ be a separated $k$-scheme of finite type.
For any embedding $k\subset \CC$, the topological Euler
characteristic $\Euler(V(\CC))$ is defined and well-known to be independent
of the embedding; cf.~\cite{Kat94}.

\begin{lemma}
  \label{lem:stab_red}
  Let $\sV_1,\dotsc,\sV_r$ be separated $\fo$-schemes of finite type and 
  $W_1,\dotsc,W_r \in \QQ(X,Y_1.\dotsc,Y_m)$.
  Suppose that for almost all $v\in \Places_k$ and all integers $f \ge 0$,
  each $W_i$ is regular at  $(q_v^f,Y_1,\dotsc,Y_m)$.
  Let $P\subset \Places_k$ have natural density $1$ and suppose that
  $$\sum\limits_{i=1}^r \# \sV_i(\fo/\fp_v) \dtimes W_i(q_v,Y_1,\dotsc,Y_m) = 0$$ for 
  all $v\in P$.
  Then
  $\sum\limits_{i=1}^r \Euler(\sV_i(\CC)) \dtimes W_i(1,Y_1,\dotsc,Y_m) = 0$.
\end{lemma}
\begin{proof}
  Using~\cite[Ch.\ 4]{Ser12}, in the setting of \cite[Thm~3.7]{stability},
  we may assume that $\alpha(1_{\Gamma_S}) = \Euler(V(\CC))$.
  The claim is now an immediate consequence of \cite[Thm~3.2]{stability} and its proof.
\end{proof}

\begin{rem}
  Given a formula \eqref{eq:main_denef} for local subalgebra or ideal zeta
  functions such that the regularity conditions in Lemma~\ref{lem:stab_red} are satisfied,
  we may read off the associated reduced zeta function as $\sum_{i=1}^r
  \Euler(V_i(\CC)) \dtimes W_i(1,Y)$ without using motivic zeta functions.
\end{rem}

The following is a consequence of the explicit formulae in \cite{DV15}.
\begin{thm}
  \label{thm:one_red_trivial}
  Let $\GG$ be a unipotent algebraic group over $k$.
  Let $\G$ be an $\fo$-form of~$\GG$ as an affine group scheme of finite type.
  There are separated $\fo$-schemes $\sU_1,\dotsc,\sU_\ell$ of finite type and
  rational functions $W_1,\dotsc,W_\ell\in \QQ(X,Y)$ such that
  \begin{enumerate}
  \item for almost all $v \in \Places_k$,
    $\zeta_{\G(\fo_v)}^{\wirr}(s) - 1 = \sum\limits_{i=1}^\ell \# \sU_i(\fo/\fp_v) \dtimes W_i(q_v,q_v^{-s})$,
  \item
    each $W_i$ is regular at each point $(q,Y)$ for $q \ge 1$, and
  \item
    $W_i(1,Y) = 0$ for $i = 1,\dotsc,\ell$.
    \end{enumerate}
\end{thm}
\begin{proof}
  In the setting of \cite[Prop.\ 3.4]{DV15}, the rational numbers $A_j$
  and $B_j$ can actually be assumed to be integers;
  this follows e.g.\ by taking square roots of principal minors and rewriting
  \cite[(2.3)]{DV15} as in \cite[(4.3)]{unipotent}.
  Next, using the same notation as in \cite[Prop.\ 3.4]{DV15}, 
  $\card{M_i} \le \card{U_i} + 1$
  whence the claim follows easily from \cite[Rem.\ 3.6]{DV15}.
\end{proof}

\begin{rem}
  Theorem~\ref{thm:one_red_trivial} refines
  the simple observation that for almost all $v\in \Places_k$,
  the coefficients of $\zeta_{\G(\fo_v)}^{\wirr}(s) - 1$ as a series in
  $q_v^{-s}$ are non-negative integers divisible by $q_v-1$, a simple
  consequence of the Kirillov orbit method. (Indeed, $(\fo/\fp_v)^\times$ acts
  freely on non-trivial characters while preserving the two types of radicals in \cite[Thm~2.6]{SV14}.)
\end{rem}

By combining Lemma~\ref{lem:stab_red} and
Theorem~\ref{thm:one_red_trivial}, we obtain the following.
\begin{cor}
  \label{cor:all_red_trivial}
  Let $\G$ be as in Theorem~\ref{thm:one_red_trivial}.
  Let $\sV_1,\dotsc,\sV_r$ be separated $\fo$-schemes of finite type and
  let $W_1,\dotsc,W_r\in \QQ(X,Y)$ such that for almost all $v \in \Places_k$,
  \[
  \zeta_{\G(\fo_v)}^{\wirr}(s) = \sum_{i=1}^r \# \sV_i(\fo/\fp_v) \dtimes W_i(q_v,q_v^{-s}).
  \]
  If each $W_i$ is regular at $(q,Y)$ for each $q \ge 1$,
  then $\sum\limits_{i=1}^r \Euler(\sV_i(\CC)) \dtimes W_i(1,Y) = 1$.
  \qed
\end{cor}

\begin{cor}
  \label{cor:uniform_red_trivial}
  Let $\G$ be as in Theorem~\ref{thm:one_red_trivial}.
  Let $W(X,Y) \in \QQ(X,Y)$ 
  such that
  \begin{enumerate}
  \item
    $W(X,Y)$ can be written over a denominator which is a product of non-zero factors of
    the form $1-X^aY^b$ for integers $a \ge 0$ and $b \ge 1$ and
  \item $\zeta_{\G(\fo_v)}^{\wirr}(s) = W(q_v^{\phantom{-s}}\!\!\!,q_v^{-s})$ for almost
    all $v\in \Places_k$.
  \end{enumerate}
  Then $W(1,Y) = 1$. \qed
\end{cor}

The assumptions in Corollary~\ref{cor:uniform_red_trivial} are
satisfied for many examples of interest; see Table~\ref{tab:six}.
In fact, even the following much stronger assumptions are often satisfied.

\begin{cor}
  \label{cor:prod_riemann}
  Let $\G$ be as in Theorem~\ref{thm:one_red_trivial}.
  Suppose that there are integers $a_i \ge 0$, $b_i \ge 1$, and $\varepsilon_i \in \{
  \pm 1\}$ for $i = 1,\dotsc,m$ such that for almost all $v\in \Places_k$,
  \[
  \zeta_{\G(\fo_v)}^{\wirr}(s) = \prod_{i=1}^m (1-q_v^{a_i - b_i s})^{\varepsilon_i}.
  \]
  Then $\sum\limits_{i=1}^m \varepsilon_i = 0$ and the \textnormal{multisets}
  $\{\!\!\{ b_i : \varepsilon_i = 1\}\!\!\}$ and
  $\{\!\!\{ b_i : \varepsilon_i = -1\}\!\!\}$
  coincide.
\end{cor}
\begin{proof}
  Corollary~\ref{cor:uniform_red_trivial} shows that
  $1 = \prod_{i=1}^m (1-Y^{b_i})^{\varepsilon_i}$.
  By considering the vanishing order of this function in $Y$
  at $1$, we see that $\sum_{i=1}^m \varepsilon_i = 0$.
  Let $c = \max( b_i : \varepsilon_i = 1)$ and $d = \max(b_i : \varepsilon_i = -1)$.
  If $\xi\in \CC$ is a primitive $c$th root of unity,
  then $1-\xi^{b_i} = 0$ for some $i$ with $\varepsilon_i = -1$ whence $c \le
  b_i \le d$; dually, $d \le c$ and the final claim follows by~induction.
\end{proof}

\begin{rem}
  The above results carry over verbatim to the case of
  representation zeta functions of ``principal congruence subgroups''
  $\G^m(\fo_v) := \exp(\fp_v^m \fg \otimes_{\fo} \fo_v)$ attached to an
  $\fo$-form of a perfect Lie $k$-algebra in \cite{AKOV13}.
  For example, by \cite[Thm~E]{AKOV13}, the ordinary representation zeta
  function of $\mathrm{SL}_3^1(\ZZ_p)$ ($p \not= 3$) is $W(p,p^{-s})$ for

  {
    \small
  \[
  W(X,Y) = 
  \frac{(X^2 Y^2 + X Y^2 + Y^3 + X^2 + X Y + Y) \times (X^2 - Y)(X - Y)X^3 }{(1 - X^2
    Y^3) (1 - X Y^2)}
  \]
  }

  \noindent and indeed $W(1,Y) = 1$.
\end{rem}

\section{Applications I: representation zeta functions of unipotent groups}
\label{app:reps}

In this and the following two sections, we record explicit examples of generic
local zeta functions of groups, algebras, and modules of interest which were
computed using the method developed in the present article and its implementation
\textsf{Zeta}~\cite{Zeta}.
The explicit formulae given below, as well as others, are also included with~\textsf{Zeta}.\\

It is well-known that, up to isomorphism, unipotent algebraic groups over $k$
correspond 1--1 to finite-dimensional nilpotent Lie $k$-algebras,
see \cite[Ch.\ IV]{DG70}.
Nilpotent Lie algebras of dimension at most $6$ over any field of characteristic
zero were first classified by Morozov~\cite{Mor58};
various alternative versions of this classification have been obtained.

As one of the main applications of the techniques developed in the present
article, for an arbitrary number field $k$,
we can compute the generic local (twist) representation zeta functions 
associated with all unipotent algebraic groups of dimension at most $6$
over~$k$.
The results of these computations are documented in Table~\ref{t:reps6} (p.\ \pageref{t:reps6}).
The structure of Table~\ref{t:reps6} mimics the list of associated topological
representation zeta functions in~\cite[Table~1]{unipotent}.
In detail, the first column lists the relevant Lie algebras using de~Graaf's
notation~\cite{dG07}; an algebra $L_{d,i}$ has dimension $d$.
For each Lie algebra $\fg$, we choose an $\fo$-form $\G$ of the unipotent
algebraic group over $k$ associated with $\fg$.
The second column in Table~\ref{t:reps6} contains formulae for the
representation zeta functions of the groups $\G(\fo_v)$ which are valid for
almost all $v\in \Places_k$ (depending on $\G$).
Note that Corollary~\ref{cor:prod_riemann} applies to the majority of
examples in Table~\ref{t:reps6}.
As we previously documented in \cite[\S 6]{unipotent}, generic local
representation zeta functions associated with various Lie algebras in
Table~\ref{tab:six} were previously known (but sometimes only recorded for $k =
\QQ$), as indicated in the third column.
For the convenience of the reader, the more detailed references to the
literature from \cite[Tab.\ 1]{unipotent} are reproduced in
Remark~\ref{rem:known_reps}.

\begin{rem}[From $\QQ$ to $k$]
  Apart from the four infinite families (see the following remark), all Lie
  algebras in Table~\ref{t:reps6} are defined over $\QQ$.
  By the invariance of \eqref{eq:main_denef} under local base extensions
  (Theorem~\ref{thm:denef_formulae}), it thus suffices to compute
  associated generic local representation zeta functions
  for $k = \QQ$.
\end{rem}

\begin{rem}[Computations for infinite families]
The method for computing generic local zeta functions developed in this article
takes as input a global object such as a nilpotent Lie $k$-algebra.
In order to carry out computations for the four infinite families 
$L_{6,19}(a)$, $L_{6,21}(a)$, $L_{6,22}(a)$, and $L_{6,24}(a)$ in
Table~\ref{tab:six}, additional arguments are required.

First, as explained in \cite{dG07}, we are free to multiply the parameters $a$
from above by elements of $(k^\times)^2 \le k^\times$ without changing the $k$-isomorphism
type of the Lie algebra, $\bm\fg(a)$ say, in question.
We may thus assume that $0 \not= a \in \fo$ in the following.
The definition of $\bm\fg(a)$ in \cite{dG07} then provides us with
a canonical $\fo$-form, $\fg(a)$ say, of $\bm\fg(a)$ which is in fact defined over $\ZZ[a]$.
Let $\G_a$ be an $\fo$-form of the unipotent algebraic group over $k$ associated
with $\bm\fg(a)$.
As explained in \cite[\S 2]{SV14}, the structure constants of $\fg(a)$
(with respect to its defining basis from \cite{dG07})
give rise to a formula for $\zeta_{\G_a(\fo_v)}^{\wirr}(s)$ in terms of certain
explicit $\fo$-defined $p$-adic integrals
(see \cite[Cor.\ 2.11]{SV14});
this formula is valid for almost all $v\in \Places_k$.

It is an elementary exercise to verify that
if $\bm\fg = L_{6,19}(a)$ or $\bm\fg = L_{6,21}(a)$, then
the polynomials featuring in the aforementioned integral formulae for
$\zeta_{\G_a(\fo_v)}^{\wirr}(s)$ are all monomials in $a$ and the variables
$Y_1,\dotsc,Y_d$ (in the notation of \cite[\S 2.2]{SV14} and up to signs).
It follows that up to excluding finitely many $v\in \Places_k$,
$\zeta_{\G_a(\fo_v)}^{\wirr}(s)$ does not depend on $a$.
We may therefore simply carry out our calculation for $k = \QQ$ and $a = 1$, say.

Let $\bm\fg(a)$ be $L_{6,22}(a)$ or $L_{6,24}(a)$.
Another simple calculation reveals that (again up to signs) a single
non-monomial polynomial occurs in the associated integral formulae from above,
namely $Y_1^2 - a Y_2^2$.
For any fixed $a$, 
by applying the procedure from \cite[\S 5.4]{unipotent} as well as the steps
described in the present article,
we produce a rational function $W_a(X,Y,Z)$ such that
$\zeta_{\G_a(\fo_v)}^{\wirr}(s) = W_a(q_v,q_v^{-s}, c_a(v))$ for almost all $v\in
\Places_k$,
where $c_a(v)$ denotes the number of roots of $X^2-a$ in $\fo/\fp_v$;
it is well-known
that if $a \not\in (k^\times)^2$, then
for almost all $v\in \Places_k$, $c_a(v) = 0$ or $c_a(v) = 2$ according to whether $\fp_v$ remains inert
or splits in $k(\sqrt a)$, respectively.
The critical observation (which follows easily from \cite[\S 5.4]{unipotent})
is that $W := W_a$ is independent of $a$ and also of $k$.
We may thus compute $W$ explicitly by e.g.\ taking $k = \QQ$ and $a = 2$.
\end{rem}

\begin{rem}
  \label{rem:known_reps}
  Explicit references for the known instances of generic local representation
  zeta functions in Table~\ref{tab:six} are as follows (cf.~\cite[Tab.\ 1]{unipotent}):
  \begin{center}
  \begin{tabular}{ll|ll}
    algebra & reference & algebra & reference \\\hline
  $L_{3,2}$ & \cite[Thm~5]{NM89}) &
  $L_{4,3}$ & $M_3$ \cite[(4.2.24)]{Ezzat} \\
  $L_{5,4}$ & $B_4$ \cite[Ex.~6.3]{Snocken}) &
  $L_{5,5}$ & $G_{5,3}$ \cite[Tab.~5.2]{Ezzat}) \\
  $L_{5,7}$ & $M_4$ \cite[(4.2.24)]{Ezzat} &
  $L_{5,8}$ & $M_{3,3}$ \cite[(5.3.7)]{Ezzat} \\
  $L_{5,9}$ & $F_{3,2}$ \cite[Tab.~5.2]{Ezzat}) & &  $=G_3$
    \cite[Ex.~6.2]{Snocken})\\
  $L_{6,10}$ & $G_{6,12}$ \cite[Tab.~5.2]{Ezzat}) &
  $L_{6,18}$ & $M_5$ \cite[(4.2.24)]{Ezzat} \\
  $L_{6,19}(0)$ & $G_{6,7}$ \cite[Tab.~5.2]{Ezzat}) &
  $L_{6,19}(1)$ & $G_{6,14}$ \cite[Tab.~5.2]{Ezzat} 
  \\
  $L_{6,22}(0)$ & \cite[Ex.~6.5]{Snocken} &
  $L_{6,22}(a)$ ($a \in k^\times \!\setminus\! (k^\times)^2$)\!\!\!\!\!\! &
  \cite{Ezz14}
  \\
  $L_{6,25}$ & $M_{4,3}$ \cite[(5.3.7)]{Ezzat} & 
  $L_{6,26}$ & $F_{1,1}$ \cite[Thm~B]{SV14}.
  \end{tabular}
  \end{center}

  The author would like to emphasise that all the formulae in
  Table~\ref{tab:six} were obtained using the method developed here.
  In particular, our computations  provide independent confirmation of the
  aforementioned (sometimes computer-assisted but predominantly manual and ad
  hoc) calculations found in the literature.
\end{rem}

For an example in dimension $> 6$,
recall from \S\ref{ss:reps} that $\Uni_n$ denotes the group scheme of upper unitriangular
$n\times n$-matrices.
Using the notation from \cite{dG07} as in Table~\ref{t:reps6},
$\Uni_3 \otimes \QQ$ (the Heisenberg group) and $\Uni_4 \otimes \QQ$
are the unipotent algebraic groups 
over $\QQ$ associated with the Lie algebras $L_{3,2}$ and $L_{6,19}(1)$, respectively.
The following result obtained using the method from the present article
illustrates that the simple shapes of the corresponding local representation
zeta functions in Table~\ref{t:reps6} may mislead. 

\begin{thm}
  For almost all primes $p$ and all finite extensions $K/\QQ_p$,
  $$\zeta_{\Uni_5(\fO_K)}^{\wirr}(s) = W(q_K,q_K^{-s}),$$ where
  \begin{align*}
    W =\,\, & \bigl(X^{10}Y^{10} - X^9 Y^9 - 2 X^9 Y^8 + X^9 Y^7 + X^8 Y^8 - X^7 Y^7 - 2
    X^7Y^6 + X^7Y^5 \\ & + 6 X^6 Y^6 - 4 X^5 Y^6 - 4 X^5 Y^4 + 6 X^4Y^4 + X^3
    Y^5 - 2 X^3 Y^4 - X^3 Y^3 \\ & + X^2 Y^2 + X Y^3 - 2 X Y^2 - XY + 1\bigr)
    \times \bigl(1 - Y^3\bigr) \times \bigl(1 - Y\bigr)
    \\ & / \bigl(
    (1 - X^6 Y^4)
    (1 - X^3Y^3)
    (1 - XY^3)
    (1 - X^2Y^2)
    (1 - X^2Y)^2
    \bigr).
  \end{align*}
\end{thm}

The topological representation zeta function of $\Uni_6$
cannot be computed using \cite{unipotent}.
Consequently, the corresponding local zeta functions cannot be computed using
the method developed here.

Observe that the numerator of each $W(X,Y)$ in Table~\ref{tab:six} is divisible
by a polynomial of the form $1-Y^e$.
Experimental evidence provided by these examples and those in \textsf{Zeta}
suggests that the following $p$-adic version of \cite[Qu.~7.4]{unipotent} might
have a positive answer.

\begin{question}
  Let $\G$ be an $\fo$-form of a non-abelian unipotent algebraic group over~$k$.
  Does the meromorphic continuation of $\zeta_{\G(\fo_v)}^{\wirr}(s)$ always vanish at zero
  for almost all $v \in \Places_k$?
\end{question}

\begin{rem*}
  By~\cite[Cor.\ 2]{GJK14}, if $p$ is odd, then the meromorphic continuation of
  the ordinary (=~non-twisted) representation zeta function of a
  compact FAb $p$-adic analytic group vanishes at $-2$.
\end{rem*}

\begin{savenotes}
\begin{table}[H]
  \small
  \centering
  \begin{tabular}{l|ll}
   \hline
   Lie algebra & $W(X,Y)$ s.t.\ $\zeta_{\G(\fo_v)}^{\wirr}(s) =
   W(q_v,q_v^{-s})$ for almost all $v\in \Places_k$ $\phantom{1^{1^{1^1}}}$\!\!\!\!\!\!\!\!  &
   known \\
    \hline
    abelian & $1$ & $\checkmark$\\
    $L_{3,2}$ & $(1-Y)/(1-XY)$ & $\checkmark$\\ 

    $L_{4,3}$ & $(1-Y)^2/(1-XY)^2$ & $\checkmark$\\ 

    $L_{5,4}$ & $(1-Y^2)/(1-XY^2)$ & $\checkmark$ \\ 

    $L_{5,5}$ &
    $(1-XY^2)(1-Y)/\bigl((1-X^2Y^2)(1-XY)\bigr)$ & $\checkmark$\\ 

    $L_{5,6}$ &
    $(1-X^2Y^2)(1-Y)^2/
    \bigl((1-X^3Y^2)(1-XY)^2 \bigr)$ \\

    $L_{5,7}$ & $(1-Y)^2/\bigl((1-X^2Y)(1-XY)\bigr)$ & $\checkmark$\\ 

    $L_{5,8}$ &  $(1-Y)/(1-X^2 Y)$ & $\checkmark$\\ 

    $L_{5,9}$ & $(1-Y)^2/\bigl( (1-X^2Y)(1-XY) \bigr)$ & $\checkmark$\\ 

    $L_{6,10}$& $(1-Y^2)(1-Y) /\bigl( (1-XY^2)(1-XY) \bigr)$  & $\checkmark$\\

    $L_{6,11}$&
    $\frac{(- X^3Y^4 + X^3Y^3 - 2 X^2Y^3 + 3 X^2Y^2 - 3 XY^2 + 2 XY - Y + 1)(1 -Y)}{(1-X^4Y^3)(1 - X^2Y^2)}$ \\

    $L_{6,12}$& $(1-X^2Y^2)(1-Y)^2 / \bigl( (1-X^3Y^2)(1-XY)^2\bigr)$ \\

    $L_{6,13}$& 
    $\frac{(X^4Y^6 + X^4Y^5 - X^3Y^4 - 2 X^2Y^3 - XY^2 + Y + 1)(1 - Y)^2}{(1-X^3Y^3)(1 - X^2Y^2)(1-XY^2)(1- XY)}$
      \\

    $L_{6,14}$&  
    $\frac{(X^4Y^6 - X^4Y^4 + X^3Y^5 - 2 X^2Y^3 + XY  - Y^2 + 1)(1- Y)^2}
    {(1-X^3Y^3)(1 - X^3Y^2)(1 - XY^2)(1 - XY)}$
    \\

    $L_{6,15}$ &
    $\frac{(-X^5Y^4 - X^4Y^3 + X^3Y^2 - X^2Y^2  + XY + 1)(1 - Y)^2}
    {(1 - X^5Y^3) (1 - X^3Y^2)(1 - XY)}$
    \\

    $L_{6,16}$ &
    $(1-Y^2)(1-Y)^2 / \bigl( (1-X^2Y)(1-XY^2)(1-XY) \bigr)$  \\

    $L_{6,17}$&  
    $ (1-X^3Y^2)(1-Y)^2/\bigl( (1-X^4Y^2)(1-X^2Y)(1-XY)\bigr)$
    \\
    
    $L_{6,18}$& $(1-Y)^2/\bigl( (1-X^3Y)(1-XY)\bigr)$ & $\checkmark$\\

    $L_{6,19}(0)$  & $(1-Y)^2/\bigl( (1-X^2Y)(1-XY)\bigr)$ & $\checkmark$\\

    $L_{6,19}(a)$ ($a\in k^\times$) \!\!\!\!&
    $(1-Y^2)(1-Y)/\bigl((1-X^2Y)(1-XY^2)\bigr)$ & $\checkmark (a=1)$\\

    $L_{6,20}$ &
    $ (1-XY^2)(1-Y) /\bigl((1-X^2Y)(1-X^2Y^2)\bigr)$
    \\

    $L_{6,21}(0)$ &
    $(1-Y)^2/(1-X^2Y)^2$  \\

    $L_{6,21}(a)$ ($a\in k^\times$)\!\!\!\! &
    $ (1-X^2Y^2)(1-Y)^2 /\bigl( (1-X^3Y^2)(1-X^2Y)(1-XY)\bigr)$ 
    \\

    $L_{6,22}(0)$ &
    $ (1-X^2Y^2)(1-Y)/\bigl((1-X^3Y^2)(1-XY)\bigr)$ & $\checkmark$\\

    $L_{6,22}(a)$ 
    
    & 
    if $\fp_v$ splits in $k(\sqrt a)$:\,\,\,\,\,
    $(1-Y)^2/(1-XY)^2$  & $\checkmark$ \\
    \,\,\,\,\,\,($a\in k^\times \!\setminus\! (k^\times)^2$)\footnote{For $a \in
      (k^\times)^2$, $L_{6,22}(a) \approx L_{3,2}^2$ decomposes.}
    & 
    if $\fp_v$ is inert in $k(\sqrt a)$: $(1-Y^2)/(1-X^2Y^2)$ & $\checkmark$
    \\

    $L_{6,23}$ & 
    $ (1-X^3Y^2)(1-Y)/\bigl((1-X^4Y^2)(1-X^2Y)\bigr)$
    \\

    $L_{6,24}(0)$ & 
    $\frac{(X^4Y^4  - X^4Y^3 + X^3Y^3  - 2 X^2Y^2 + XY - Y + 1)(1 - Y)}
    {(1 - X^3Y^2)^2 (1 - XY)}$
    \\

    $L_{6,24}(a)$ 
     &
    if $a \in (k^\times)^2$ or $\fp_v$ splits in $k(\sqrt a)$: 
    $\frac{(- XY^2 + 2 XY  - 2 Y + 1) (1 - Y)}{(1 - X^3Y^2)(1 - XY)}$
    \\
    \,\,\,\,\,\,($a\in k^\times$)
    &
    if $a \not\in (k^\times)^2$ and $\fp_v$ is inert in $k(\sqrt a)$:
    $\frac{(1-XY^2)(1-Y)}{(1-X^3Y^2)(1-XY)}$
    \\

    $L_{6,25}$&  $(1-XY)(1-Y)/(1-X^2Y)^2$ & $\checkmark$\\
    
    $L_{6,26}$& $(1-Y)/(1-X^3Y)$ & $\checkmark$\\ 
    \hline
  \end{tabular}
  \caption{Generic local representation zeta functions
    associated with all indecomposable unipotent algebraic groups of dimension 
    at most six over a number field}
  \label{t:reps6}
  \label{tab:six}
\end{table}
\end{savenotes}

\section{Applications II: classical subobject zeta functions}

\subsection{Subalgebras: $\Gl_2(\QQ)$}
\label{ss:gl2}

The first computations of the subalgebra zeta functions of $\Sl_2(\ZZ_p)$
are due, independently, to du~Sautoy~\cite{dS00} (for $p \not= 2$, relying
heavily on~\cite{Il99}) and White~\cite{Whi00}.
These zeta functions have later been confirmed by different means in
\cite{dST02}, \cite[\S 4.2]{KV09}, and \cite[\S 7.1]{topzeta} (for $p \not= 2$).
Up until now, $\Sl_2(\QQ)$ has remained the sole example of an insoluble Lie
$\QQ$-algebra whose generic local subalgebra zeta functions have been computed.
Using the method developed in the present article, we obtain the following.

\begin{thm}
  \label{thm:gl2}
  For almost all primes~$p$ and all finite extensions $K/\QQ_p$,
\begin{align*}
  \zeta_{\Gl_2(\fO_K)}^\le(s) & = W(q_K^{\phantom{-s}}\!\!,q_K^{-s}),
\end{align*}
where
\begin{align*}
  W(X,Y) =\, & \bigl(- X^{8} Y^{10} -  X^{8} Y^{9} -  X^{7} Y^{9} - 2 X^{7} Y^{8} + X^{7} Y^{7} -  X^{6} Y^{8} \\
& - X^{6} Y^{7} + 2 X^{6} Y^{6} - 2 X^{5} Y^{7} + 2 X^{5} Y^{5} - 3 X^{4} Y^{6} + 3 X^{4} Y^{4} \\
& -2 X^{3} Y^{5} + 2 X^{3} Y^{3} - 2 X^{2} Y^{4} + X^{2} Y^{3} + X^{2} Y^{2} -  X Y^{3} \\
& + 2 X Y^{2} + X Y + Y + 1 \bigr)/ \Bigl((1- X^7 Y^6) (1-X^3Y^3)(1-X^2Y^2)^2(1-Y)\Bigr).
\end{align*}
\end{thm}

The topological subalgebra zeta function $\zeta_{\Gl_2(\QQ),\topo}^\le(s) = (27s-14)/(6(6s-7)(s-1)^3s)$
of $\Gl_2(\QQ)$ was first recorded in \cite[\S 7.3]{topzeta} (relying on
techniques from~\cite{topzeta2}); the result given there is consistent with Theorem~\ref{thm:gl2}.
Theorem~\ref{thm:gl2} is particularly interesting since
the simple shape of $\zeta_{\Gl_2(\QQ),\topo}^\le(s)$ might seem indicative of a local
zeta function which is a product of ``cyclotomic factors'' $1 - q_K^{a - bs}$ or
their inverses, which is in fact not the case.

We note that the computations underpinning Theorem~\ref{thm:gl2}
used that $\Gl_2(R) \approx \Sl_2(R) \oplus R$ for any commutative
ring $R$ in which $2$ is invertible; here we regarded $R$ as an abelian Lie
$R$-algebra.
Theorem~\ref{thm:gl2} therefore also illustrates the potentially wild effect of
direct sums on subalgebra zeta functions;
in contrast, \cite{dSW08} contains examples of subalgebra and ideal zeta
functions associated with nilpotent Lie algebras which are very well-behaved
under this operation.

The rational function $W(X,Y)$ in Theorem~\ref{thm:gl2} satisfies the functional
equation $$W(X^{-1},Y^{-1}) = X^6 Y^4 W(X,Y)$$ predicted by \cite[Thm~A]{Vol10}
(cf.\ \cite[\S 5]{stability}).
Moreover, the reduced subalgebra zeta function of $\Gl_2(\QQ)$ is $W(1,Y) =
(1-Y^3)/\bigl((1-Y)^3(1-Y^2)^2\bigr)$, as predicted 
by \cite[Thm~3.3]{Evs09} (using the fact that the reduced subalgebra zeta function
of $\Sl_2(\ZZ)$ is $(1-Y^3)/\bigl((1-Y)^2(1-Y^2)^2\bigr)$ (by \cite[Prop.~4.1]{Evs09})).

\subsection{Subalgebras: $k[T]/T^n$ for $n \le 4$}

Most examples of local subalgebra zeta functions in the literature 
are concerned with (often nilpotent) Lie algebras.
An important exception is given by the subalgebra zeta functions of
$\ZZ_p^n$ endowed with component-wise multiplication;
explicit formulae for these zeta functions are known for $n \le 3$ (see~\cite{Nak96}).
In the following, we consider another natural family of associative, commutative
algebras, $k[T]/T^n$, for $n \le 4$.

Due to the simplicity of the associated ``cone integrals'' as in \cite{dSG00},
the formulae for $n = 2, 3$ recorded in the following can be obtained by hand with little
difficulty.
Using a substantially more involved computation,
the techniques developed in the present article also allow us to consider the
case $n = 4$.
For $n = 5$, the author's techniques for computing topological subalgebra zeta
functions do not apply, i.e.\ Assumption~\ref{A1} is violated.

\begin{thm}
  \label{thm:X4}
  For almost all primes~$p$ and all finite extensions $K/\QQ_p$,
  writing $q = q_K$,
  \begin{align*}
  \zeta_{\fO_K[T]/T^2}^\le(s) & = 
  \frac{1-q^{-2s}}{(1-q^{-s})^2(1-q^{1-2s})},  \\
  \zeta_{\fO_K[T]/T^3}^\le(s) & = F_{\QQ[T]/T^3}(q,q^{-s})
  \times \frac{1-q^{2-4s}}
  {(1-q^{4-5s})(1-q^{2-3s})^2(1-q^{1-2s})(1-q^{-s})}, \text{ and } \\
  \zeta_{\fO_K[T]/T^4}^\le(s) & = F_{\QQ[T]/T^4}(q,q^{-s})
  /\bigl( 
  (1-q^{13-13s})(1-q^{9-9s})(1-q^{8-8s})(1-q^{6-6s})^2 \\& \quad\quad \quad\quad\quad\quad\quad\quad\quad\times(1-q^{5-6s})
  (1-q^{5-5s})(1-q^{3-4s}) (1-q^{-s})
  \bigr),
  \end{align*}
  where
  $F_{\QQ[T]/T^3} = -X^{4} Y^{7} - X^{4} Y^{6} - X^{3} Y^{5} + X^{3} Y^{4} - X^{2} Y^{4} +
  X^{2} Y^{3} - X Y^{3} + X Y^{2} + Y + 1$
  and $F_{\QQ[T]/T^4} = 1 + \dotsb - X^{49}Y^{54} \in \QQ[X,Y]$ is given in Appendix~\ref{s:numerators}.
\end{thm}

The topological subalgebra zeta function of $\QQ[T]/T^4$ can be found in~\cite[\S 9.2]{topzeta2}.
As in \S\ref{ss:gl2}, the zeta functions in Theorem~\ref{thm:X4} satisfy the
functional equations predicted by \cite[Thm~A]{Vol10} and the associated reduced
subalgebra zeta functions coincide with those computed using~\cite{Evs09};
while Evseev only considered reduced zeta functions of \itemph{Lie}
algebras, his reasoning also applies to more general, possibly non-associative,
algebras.
For example, using Theorem~\ref{thm:X4}, after considerable cancellation,
we find the reduced subalgebra zeta function 
of $\QQ[T]/T^4$ to be $(Y^6 + Y^4 + 2Y^3 + Y^2 + 1)/\bigl( (1-Y^6) (1-Y^2)
(1-Y)^2 \bigr)$,
as predicted by Evseev's results.

\subsection{Subalgebras: soluble, non-nilpotent Lie algebras}

Taylor~\cite[Ch.\ 6]{Tay01} computed local subalgebra zeta functions associated
with soluble, non-nilpotent Lie algebras of the form 
$k^d \rtimes k$ (semidirect sum) for $d = 2,3$, where $k^d$ and $k$ are regarded
as abelian Lie algebras.
In particular, he (implicitly) computed the subalgebra zeta function of the Lie
algebra $\mathfrak{tr}_2(\ZZ_p)$ of upper triangular $2\times 2$-matrices
over~$\ZZ_p$ (see \cite[\S 3.4.2]{dSW08}).
Klopsch and Voll~\cite{KV09} computed subalgebra zeta functions of
arbitrary $3$-dimensional Lie $\ZZ_p$-algebras in terms of Igusa's local zeta
functions attached to associated quadratic forms.
Regarding the enumeration of ideals of soluble, non-nilpotent Lie algebras,
Woodward~\cite{Woo08} computed local ideal zeta functions of
$\mathfrak{tr}_d(\ZZ_p)$ and certain combinatorially defined quotients of these
algebras.

Since, to the author's knowledge, no examples of generic local subalgebra zeta 
functions associated with soluble, non-nilpotent Lie algebras of dimension
$4$ have been record\-ed in the literature, we now include some examples.

\begin{thm}
  \label{thm:M}
  Let $\sM^i$ denote an arbitrary but fixed $\ZZ$-form of the soluble Lie $\QQ$-algebra $M^i$
  of dimension~$4$ from \cite{dG05}.
  Then for almost all primes~$p$ and all finite extensions $K/\QQ_p$, writing $q
  = q_K$,
  \begin{align*}
    \zeta_{\sM^6_{0,0} \otimes\fO_K}^{\le} & =
    \bigl( q^{8-7s} - q^{7-5s} + q^{6-5s} - 2 q^{5-4s} + q^{4-4s}
      + q^{4-3s} - 2 q^{3-3s} + q^{2-2s} \\&\quad\quad - q^{1-2s} + 1\bigr)
      /\bigl(
      (1 - q^{6-4s})
      (1 - q^{3-2s})^2
      (1 - q^{1-s})^2
      (1 - q^{-s})
      \bigr),
    \\
    \zeta_{\sM^8 \otimes\fO_K}^{\le} & =
    \bigl(
      q^{5-7s} - 3 q^{4-5s} + q^{4-4s} + 2 q^{3-5s} - 2 q^{3-4s} + q^{3-3s} +
      q^{2-4s}\\& \quad\quad  - 2 q^{2-3s} + 2 q^{2-2s} + q^{1-3s} - 3 q^{1-2s} + 1
      \bigr)\\& \quad\quad/\bigl(
      (1 - q^{6-5s})
      (1 - q^{2-2s})
      (1 - q^{1-s})^3
      (1 - q^{-s})\bigr),
    \\
    \zeta_{\sM^{12} \otimes\fO_K}^{\le} & =
    \frac{1 - q^{2-3s}}{(1 - q^{3-2s}) (1 - q^{2-2s}) (1 - q^{2-s}) (1 - q^{1-s}) (1 - q^{-s})},
    \\
    \zeta_{\sM^{13}_0 \otimes\fO_K}^{\le} & =
    \frac{ -q^{4-5s} - q^{3-4s} + q^{3-3s} - 2 q^{2-3s} + 2 q^{2-2s} - q^{1-2s} + q^{1-s} + 1}
    {
      (1 - q^{4-3s})
      (1 - q^{3-2s})
      (1 - q^{2-2s})
      (1 - q^{1-s})
      (1 - q^{-s})
    }.
  \end{align*}
\end{thm}

\begin{rem}
  Let $\bm\fg$ be the non-abelian Lie $\QQ$-algebra of dimension~$2$.
  Define a $\ZZ$-form $\fg$ of $\bm\fg$ by $\fg = \ZZ x \oplus \ZZ y$ and
  $[x,y] = y$.
  Then it is easy to see that for all $p$-adic fields~$K$,
  $\zeta_{\fg\otimes\fO_K}^{\le}(s) = 1/\bigl((1-q_K^{-s})(1-q_K^{1-s})\bigr)$.
  Using the notation from \cite{dG05} as in Theorem~\ref{thm:M}, $M^8 \approx
  \bm\fg \oplus \bm\fg$ and $M^{13}_0 \approx \bm\fg \otimes_{\QQ} \QQ[X]/X^2$.
\end{rem}

\subsection{Submodules: $\Uni_n$ for $n \le 5$ and relatives}
\label{ss:Un}
For any commutative ring $R$, we consider
\[
\mathrm U_n(R) =
\begin{bmatrix} 
  1 & R & \dotsb & R \\
  0 & \ddots & \ddots & \vdots \\
  \vdots & \ddots & \ddots  & R \\
  0 & \dotsb & 0 & 1
\end{bmatrix}
\]
together with its natural action on $R^n$ by right-multiplication.
For $n \le 4$, the determination of submodule zeta functions associated with
$\Uni_n$ in the following is quite straightforward, even without the
techniques developed here; the case $n = 5$, however, is rather more
complicated, as is the resulting formula.

\begin{thm}
  \label{thm:Un}
  For almost all primes~$p$ and all finite extensions $K/\QQ_p$, writing $q = q_K$,
  \begin{align*}
    \zeta_{\Uni_2(\fO_K) \acts \fO_K^2}(s)
  & =
  \frac{1}{ (1-q^{1-2s})(1-q^{-s})}, \\
  \zeta_{\Uni_3(\fO_K) \acts \fO_K^4}(s) & = 
  \frac{1 - q^{1-4s}}
  {
    {\left(1 - q^{2-4s} \right)}
    {\left(1 - q^{1-3s} \right)}
    {\left(1 - q^{1-2s}\right)}
    {\left(1 - q^{-s}\right)}}, \\
  \zeta_{\Uni_4(\fO_K) \acts \fO_K^4}(s) & = 
  F_{\Uni_4}(q,q^{-s})
  /\bigl( (1-q^{4-8s}) 
  (1-q^{3-7s})
  (1-q^{2-6s})
  (1-q^{2-5s})\\& \quad\quad \times
  (1-q^{2-4s}) 
  (1-q^{1-4s})
  (1-q^{1-2s})
  (1-q^{1-3s})
  (1-q^{-s})
  \bigr), \\
  \zeta_{\Uni_5(\fO_K) \acts \fO_K^4}(s) & = 
  F_{\Uni_5}(q,q^{-s})
  /\bigl( 
  {(1 - q^{6-13s} )}
  {(1 - q^{6-12s} )}
  {(1 - q^{4-11s} )}\\ &\quad\quad \times
  {(1 - q^{4-10s})} 
  {(1 - q^{3-10s})} 
  {(1 - q^{4-9s} )} 
  {(1 - q^{3-9s} )}
  {(1 - q^{4-8s} )} \\& \quad\quad \times
  {(1 - q^{3-8s} )} 
  {(1 - q^{2-8s} )} 
  {(1 - q^{3-7s} )}
  {(1 - q^{2-7s} )} 
  {(1 - q^{2-6s} )} \\& \quad\quad \times
  {(1 - q^{2-5s} )}
  {(1 - q^{1-5s} )} 
  {(1 - q^{2-4s} )}
  {(1 - q^{1-4s})}
  {(1 - q^{1-2s} )} \\& \quad\quad \times
  {(1 - q^{-s} )} \bigr),
\end{align*}
where
\begin{align*}
  F_{\Uni_4} =\, &
- X^{10} Y^{30} + X^{9} Y^{26} + X^{9} Y^{25} + X^{9} Y^{24} -  X^{9} Y^{23} + 2 X^{8} Y^{23} -  X^{8} Y^{22} + 2 X^{7} Y^{22} \\
& -2 X^{7} Y^{21} - 2 X^{7} Y^{20} + X^{6} Y^{21} - 2 X^{7} Y^{19} + X^{6} Y^{20} -  X^{6} Y^{18} -  X^{6} Y^{17} -  X^{5} Y^{18} \\
& - X^{5} Y^{17} + 2 X^{6} Y^{15} -  X^{5} Y^{16} + X^{5} Y^{14} - 2 X^{4} Y^{15} + X^{5} Y^{13} + X^{5} Y^{12} + X^{4} Y^{13} \\
& + X^{4} Y^{12} -  X^{4} Y^{10} + 2 X^{3} Y^{11} -  X^{4} Y^{9} + 2 X^{3} Y^{10} + 2 X^{3} Y^{9} - 2 X^{3} Y^{8} + X^{2} Y^{8} \\
& -2 X^{2} Y^{7} + X Y^{7} -  X Y^{6} -  X Y^{5} -  X Y^{4} + 1 
\end{align*}
and $F_{\Uni_5} = 1 + \dotsb + X^{43}Y^{124}$ is given in Appendix~\ref{s:numerators}.
These formulae for $n \le 5$ satisfy the functional equation
\[
\zeta_{\Uni_n(\fO_K) \acts \fO_K^n}(s) \Big\vert_{q\to q^{-1}} =
(-1)^n q^{\binom n 2 - \binom{n+1} 2 s} \dtimes
\zeta_{\Uni_n(\fO_K) \acts \fO_K^n}(s).
\]
\end{thm}

Despite the increasing complexity of the formulae in Theorem~\ref{thm:Un},
we note that the ``reduced submodule zeta function'' of $\Uni_n(\ZZ)$ acting on 
$\ZZ^n$ (defined and computed using a simple variation of \cite{Evs09}) 
is given by the simple formula $1/((1-Y)(1-Y^2)\dotsb(1-Y^n))$ for all $n \ge 1$.

\begin{rem}
  \label{rem:Lie_and_Un}
  Let $\bm\fg$ be an $n$-dimensional nilpotent Lie $k$-algebra.

  \begin{enumerate}
    \item
      \label{rem:Lie_and_Un1}
      By Engel's theorem, after choosing a suitable basis, we may regard $\ad(\bm\fg)$ as
      a subset of the enveloping associative algebra $k[\Uni_n(k)]$ of
      $\Uni_n(k)$ within $\Mat_n(k)$.
      In particular, the submodule growth of $\Uni_n(\fo_v)$ acting on $\fo_v^n$
      provides a lower bound for the ideal growth of nilpotent Lie
      $\fo_v$-algebras of additive rank $n$ (and without $\fo_v$-torsion).
    \item
      \label{rem:Lie_and_Un2}
      Suppose that $n > 1$.
      It is easy to see that
      the minimal number of generators of $k[\Uni_n(k)]$ as a unital,
      associative $k$-algebra is $n - 1$
      (use, for instance,  \cite[p.\ 263]{GS64}).
      Let $\bm{\mathfrak z}$ denote the centre of $\bm\fg$.
      Then, as a Lie algebra $\ad(\bm\fg) \approx \bm\fg/\bm{\mathfrak z}$ is generated by
      $\dim_k(\bm\fg / ([\bm\fg,\bm\fg]+\bm{\mathfrak z}))$ many elements.
      Hence, if $\bm\fg$ has class $\ge 3$, then $\ad(\bm\fg)$ is generated by fewer than
      $n-1$ elements.
      If, on the other hand, $\bm\fg$ has class~$2$, then $n \ge 3$ and $\ad(\bm\fg)$
      is an abelian Lie algebra while $k[\Uni_n(k)]$ is non-commutative.
      We conclude that $\ad(\bm\fg)$ never generates all of $k[\Uni_n(k)]$ for $n
      > 1$.
    \end{enumerate}
\end{rem}

\begin{question}
  \label{qu:alpha_Un}
  Is the abscissa of convergence of $\zeta_{\Uni_n(\fo) \acts \fo^n}(s)$ always
  $1$ for $n \ge 1$?
\end{question}

In view of Remark~\ref{rem:Lie_and_Un}(\ref{rem:Lie_and_Un1}),  Question~\ref{qu:alpha_Un} 
is particularly interesting since the abscissa of convergence of a subalgebra
zeta function derived from a $k$-algebra of dimension $n$, say,
is bounded from below by a linear function of $n$ (cf.\ \cite[Thm~5.1]{Bra09}).

Let $n \ge 2$. If Question~\ref{qu:alpha_Un} has a positive answer, then
there does \itemph{not} exist a nilpotent Lie $\fo$-algebra $\fg$ which is
finitely generated as an $\fo$-module such that
$\zeta_{\Uni_n(\fo_v)\acts\fo_v^n}(s) = \zeta_{\fg\otimes_{\fo}\fo_v}^\normal(s)$
for almost all $v\in \Places_k$.  
Indeed, it is easy to see that for every finite $S \subset \Places_k$,
the abscissa of convergence of
$\prod_{v\in \Places_k\setminus S} \zeta_{\fg \otimes_{\fo}\fo_v}^{\normal}(s)$ is at least
$d := \dim_k(\fg/[\fg,\fg]\otimes_{\fo} k)$ (cf.~\cite[Prop.\ 1]{GSS88}) and
we may clearly assume $d > 1$.
A positive answer to Question~\ref{qu:alpha_Un} would thus refine
Remark~\ref{rem:Lie_and_Un}(\ref{rem:Lie_and_Un2}). 

For another illustration of the generally wild effect of direct products of
algebraic structures on
associated zeta functions, we now consider generic local submodule zeta
functions associated with products $\Uni_{n_1}\times \dotsb \times \Uni_{n_r}$,
diagonally embedded into $\Uni_{n_1 + \dotsb + n_r}$.

\begin{thm}
  For almost all primes~$p$ and all finite extensions $K/\QQ_p$,
  writing $q = q_K$,
  \begin{align*}
  \zeta_{\Uni_2^2(\fO_K) \acts \fO_K^4}(s) & =
  (1 - q^{2-3s})/\bigl((1-q^{3-3s})( 1-q^{2-2s})^2(1-q^{1-s})(1-q^{-s})\bigr),
  \\
  \zeta_{\Uni_2^3(\fO_K) \acts \fO_K^6}(s) & =
  F_{\Uni_2^3}(q,q^{-s})/\bigl((1 - q^{8-5s})(1 - q^{5-4s})(1 - q^{4-3s})(1 -
  q^{3-2s})^3\\
  &
  \quad\quad\times(1 - q^{2-s})(1 - q^{1-s})(1 - q^{-s})\bigr),
  \\
  \zeta_{(\Uni_3 \times \Uni_2)(\fO_K) \acts \fO_K^5}(s) & =
  F_{\Uni_3\times\Uni_2}(q,q^{-s})/\bigl( 
  (1 - q^{6-6s})
  (1 - q^{4-5s})
  (1 - q^{3-4s})
  (1 - q^{3-3s})
  \\&\quad\quad\times
  (1 - q^{2-3s})
  (1 - q^{2-2s})^2
  (1 - q^{1-s})
  (1 - q^{-s})
  \bigr),
  \\
  \zeta_{\Uni_3^2(\fO_K) \acts \fO_K^6}(s) & = 
  F_{\Uni_3^2}(q,q^{-s})/\bigl(
  (1 - q^{9-9s})(1 - q^{8-8s})(1 - q^{6-7s})(1 - q^{5-7s}) \\
  & \quad\quad \times
  (1 - q^{6-6s})(1 - q^{4-6s})(1 - q^{4-5s})(1 - q^{3-5s})
  (1 - q^{3-4s})\\
  & \quad\quad \times (1 - q^{3-3s})(1 - q^{2-3s})(1 - q^{2-2s})^2
  (1 - q^{1-s})(1 - q^{-s})
  \bigr),
  \end{align*}
  where 
  \begin{align*}
    F_{\Uni_2^3} & =
    -X^{14} Y^{12} + 3 X^{11} Y^9 - X^{11} Y^8 - 2 X^{10} Y^9 + 2 X^{10} Y^8 - X^8
    Y^7 + 2 X^7 Y^7 \\& \quad\quad - 2 X^7 Y^5 + X^6 Y^5 - 2 X^4 Y^4 + 2 X^4Y^3 + X^3 Y^4 - 3
    X^3Y^3 + 1, \\
    F_{\Uni_3 \times \Uni_2} & = 
    X^{13} Y^{18} - X^{11} Y^{15} - 2 X^{11} Y^{14} + X^{11} Y^{13} + X^{10}
    Y^{14} - 2 X^{10} Y^{13} + X^9 Y^{12} \\& \quad\quad- 2 X^8 Y^{12} + 3 X^8 Y^{11} - 2 X^7
    Y^{11} + X^8 Y^9 + X^7 Y^{10} + X^6 Y^8 + X^5 Y^9 \\& \quad\quad - 2 X^6 Y^7 + 3 X^5 Y^7 -
    2 X^5 Y^6 + X^4 Y^6 - 2 X^3 Y^5 + X^3 Y^4 + X^2 Y^5 \\&\quad\quad- 2 X^2 Y^4 - X^2 Y^3 + 1,
\end{align*}
  and
  $F_{\Uni_3^2} = -X^{43}Y^{57} + \dotsb + 1$ is given in
  Appendix~\ref{s:numerators}.

  These generic local zeta functions satisfy the following functional equations:
  \begin{align*}
    \zeta_{\Uni_2^2(\fO_K) \acts \fO_K^4}(s) \Big\vert_{q\to q^{-1}} & =
    q^{6-6s} \dtimes \zeta_{\Uni_2^2(\fO_K) \acts \fO_K^4}(s),\\
    \zeta_{\Uni_2^3(\fO_K) \acts \fO_K^6}(s) \Big\vert_{q\to q^{-1}} & =
    q^{15-9s} \dtimes \zeta_{\Uni_2^3(\fO_K) \acts \fO_K^6}(s),\\
    \zeta_{(\Uni_3\times\Uni_2)(\fO_K) \acts \fO_K^5}(s) \Big\vert_{q\to q^{-1}} & =
    -q^{10-9s} \dtimes \zeta_{(\Uni_3\times\Uni_2)(\fO_K) \acts \fO_K^5}(s),\\
    \zeta_{\Uni_3^2(\fO_K) \acts \fO_K^6}(s) \Big\vert_{q\to q^{-1}}
    & = q^{15-12s} \dtimes \zeta_{\Uni_3^2(\fO_K) \acts \fO_K^6}(s).
  \end{align*}
\end{thm}

Further examples of the above form are included with \textsf{Zeta}; here, we
only record the following functional equations.
\begin{thm}
  \label{thm:feqn_exs}
  For almost all primes $p$ and all finite extensions $K/\QQ_p$, writing $q =
  q_K$,
  \begin{align*}
    \zeta_{(\Uni_5\times \Uni_1)(\fO_K) \acts \fO_K^6}(s) \Big\vert_{q\to q^{-1}} & =
    q^{15-16s} \dtimes \zeta_{(\Uni_5\times \Uni_1)(\fO_K) \acts \fO_K^6}(s), \\
    \zeta_{(\Uni_3\times \Uni_2 \times \Uni_1)(\fO_K) \acts \fO_K^6}(s) \Big\vert_{q\to q^{-1}} & =
    q^{15-10s} \dtimes \zeta_{(\Uni_3\times \Uni_2 \times \Uni_1)(\fO_K) \acts \fO_K^6}(s), \\
    \zeta_{(\Uni_4\times \Uni_2)(\fO_K) \acts \fO_K^6}(s) \Big\vert_{q\to q^{-1}} & =
    q^{15-13s} \dtimes \zeta_{(\Uni_4\times \Uni_2)(\fO_K) \acts \fO_K^6}(s),  \\
    \zeta_{(\Uni_3\times \Uni_2^2)(\fO_K) \acts \fO_K^7}(s) \Big\vert_{q\to q^{-1}} & =
    -q^{21-12s} \dtimes \zeta_{(\Uni_3\times \Uni_2^2)(\fO_K) \acts \fO_K^7}(s).
  \end{align*}
\end{thm}

We note that the Uniformity Problem has a positive solution for each of the four
families of local zeta functions in Theorem~\ref{thm:feqn_exs}.

\section{Applications III: graded subobject zeta functions}
\label{app:graded}

By \cite[\S 5.1]{Kuz99}, up to isomorphism, there are exactly $26$
non-abelian fundamental graded Lie $\CC$-algebras (see \S\ref{s:graded_Lie}) of
dimension at most six.
All of these algebras are defined in terms of integral structure constants
which thus provide us with ``natural'' $\QQ$-forms.
It turns out that for each of the resulting $26$ graded Lie $\QQ$-algebras, we
can use the techniques developed here to compute their associated generic local
graded subalgebra and graded ideal zeta functions.
We note that for various of these Lie algebras, the associated non-graded
subalgebra and ideal zeta functions are unknown.

\paragraph{Examples of graded ideal zeta functions.}
Table~\ref{t:grid} lists the generic local ideal zeta functions associated with
the aforementioned $26$ graded Lie $\QQ$-algebras.
The first column contains the names of the associated $\CC$-algebras as in
\cite{Kuz99};
here, an algebra called ``m$d$\_$c$\_$i$'' has dimension~$d$ and nilpotency class~$c$.

Given a $\ZZ$-form $\fg$ of a graded Lie algebra $\bm\fg$
as indicated by an entry in the first column,
the rational function $W(X,Y)$ in the corresponding entry of
the second column satisfies the following property:
for almost all rational primes~$p$ and all finite extensions $K/\QQ_p$, 
$\zeta_{\fg\otimes \fO_K}^{\gr\normal}(s) = W(q_K,q_K^{-s})$.
An entry $\pm X^a Y^b$ in the third column of Table~\ref{t:grid} indicates
that the corresponding $W(X,Y)$ satisfies $W(X^{-1},Y^{-1}) = \pm X^aY^b
\dtimes W(X,Y)$; an entry ``\ding{55}'' signifies the absence of such a functional equation.

The algebras m6\_3\_2 and m6\_3\_3 are precisely the graded Lie algebras
associated with $L_{(3,2)}$ in \cite[Thm~2.32]{dSW08}
(also called $L_W$~\cite[Thm~3.4]{Woo05} and $L_{6,25}$~\cite{dG07})
and $\fg_{6,7}$ in \cite[Thm~2.45]{dSW08} (called $L_{6,19}(0)$ in \cite{dG07}),
respectively.
The non-graded local ideal zeta functions of these algebras do not satisfy
functional equations of the above form either.
The algebra m6\_4\_1 is the graded Lie algebra associated with $L_{6,21}(0)$
from \cite{dG07}; to the author's knowledge, the non-graded local (and
topological) subalgebra and ideal zeta functions of this algebra are unknown.

We note that the formulae for m3\_2, m4\_3, m5\_4\_1, and m6\_5\_1 in
Table~\ref{t:grid} are consistent with and explained by
Proposition~\ref{prop:maximal_class}.

\paragraph{Examples of graded subalgebra zeta functions.}
While the methods developed here can be used to compute the generic  local
graded subalgebra zeta functions of all $26$ algebras in Table~\ref{t:grid},
we chose to only include the smaller ones of these examples 
in Table~\ref{t:grsub} (and Appendix~\ref{app:grsub}); for a complete list, we
refer to \textsf{Zeta}~\cite{Zeta}.

\paragraph{Open questions.}
Voll~\cite[Thm~A]{Vol10} established local functional equations under
``inversion of $p$'' for generic local subalgebra zeta functions without any
further assumptions on the algebra in question.
It is reasonable to expect the following question to have a positive answer;
the precise form of \eqref{eq:grsub_feqn} below was suggested to the author by Voll.

\begin{question}
  Let $\sA = \sA_1 \oplus \dotsb \oplus \sA_r$ be an $\fo$-form of a possibly
  non-associative finite-dimensional $k$-algebra together with a direct sum
  decomposition into free $\fo$-submodules.
  Let $n = \rank_{\fo}(\sA)$ and $m = \sum\limits_{i=1}^{r}
  \binom{\rank_{\fo}(\sA_i)} 2$.
  Is it always the case that
  \begin{equation}
    \label{eq:grsub_feqn}
  \zeta_{\sA\otimes_{\fo}\fo_v}^{\gr\le}(s) \Big\vert_{q_v^{\phantom 1}\to
    q_v^{-1}} = (-1)^n q_v^{m - ns} \dtimes \zeta_{\sA\otimes_{\fo}\fo_v}(s)
  \end{equation}
  for almost all $v\in \Places_k$?
\end{question}

The following three questions are graded analogues of conjectures due to
Voll~\cite{Vol16}.

\begin{question}
  Let $\bm\fg = \bm\fg_1 \oplus \dotsb \oplus \bm\fg_c$ be a finite-dimensional
  graded Lie $k$-algebra of class~$c$.
  Let $d_i = \dim(\bm\fg_i)$ and $d = \dim(\bm\fg)$.
  Let $0 = \bm{\mathfrak z}_0 \subset \dotsb \subset \bm{\mathfrak z}_c =
  \bm\fg$ be the upper central series
  of~$\bm\fg$ and write $e_i = \dim(\bm\fg/\bm{\mathfrak z}_i)$.
  Let $\fg$ be an $\fo$-form of $\bm\fg$ as a graded Lie algebra.

  \begin{enumerate}
  \item 
    Does $\zeta_{\fg\otimes_{\fo}\fo_v}^{\gr\normal}(s)$ have degree
    $e_1 + \dotsb + e_c$ in $q_v^{-s}$ for almost all $v \in \Places_k$?
  \item
    Suppose that there exists $W \in \QQ(X,Y)$ such that
    $\zeta_{\fg\otimes_{\fo}\fo_v}^{\gr\normal}(s) = W(q_v,q_v^{-s})$ for almost
    all~$v\in \Places_k$.
    Does $W$ have degree $\binom{d_1} 2 + \dotsb + \binom{d_c} 2$
    in $X$?
  \item
    Suppose that for almost all $v\in \Places_k$,
    \[
    \zeta_{\fg\otimes_{\fo}\fo_v}^{\gr\normal}(s) \Big\vert_{q_v^{\phantom 1}
      \to q_v^{-1}} = \varepsilon q_v^{a-bs} \dtimes \zeta_{\fg\otimes_{\fo}\fo_v}^{\gr\normal}(s),
    \]
    where $\varepsilon = \pm 1$ and $a,b\in \ZZ$.
    Do we have $\varepsilon = (-1)^d$, $a = \binom{d_1} 2 + \dotsb +
    \binom{d_c}2$, and $b = e_1 + \dotsb + e_c$?
  \end{enumerate}
\end{question}

Finally, the following is closely related to the questions raised in \cite[\S 8.2]{topzeta}.

\begin{question}
  \label{qu:pole_order}
  Let $\fg = \fg_1 \oplus \dotsb \oplus \fg_c$ be a graded nilpotent Lie
  $\fo$-algebra of class $c$,
  where each $\fg_i$ is free and of finite rank as an $\fo$-module.
  Do $\zeta_{\fg\otimes_{\fo}\fo_v}^{\gr\le}(s)$ 
  and $\zeta_{\fg\otimes_{\fo}\fo_v}^{\gr\normal}(s)$ always have a pole of
  order $c$ at zero for $v \in \Places_k$?
\end{question}

As in \cite[\S 8.2]{topzeta}, a natural follow-up question would be to interpret
or predict the leading coefficients of the zeta functions in
Question~\ref{qu:pole_order} expanded as Laurent series in~$s$; 
however, perhaps unexpectedly, the examples in Tables~\ref{t:grid}--\ref{t:grsub} show
that these leading coefficients are not functions of $v$ and the numbers
$(\rank_{\fo}(\fg_1),\dotsc,\rank_{\fo}(\fg_c))$ alone.

{\small
\begin{table}[H]
  \centering
  \begin{tabular}{rlc}
    $\bm\fg$
    & $W(X,Y)$ s.t.\ $\zeta_{\fg\otimes \fO_K}^{\gr\normal}(s) =  W(q_K,q_K^{-s})$ 
    & FEqn \\
    \hline
    m3\_2 &
    $1/ \bigl({\left(1 - X Y \right)} {\left(1 - Y^3\right)} {\left(1 - Y\right)}\bigr)$   &
    $-X Y^{5}$
    \\
    m4\_2 &
    $1 /\bigl( {\left(1 - X^{2} Y\right)} {\left(1 - X Y\right)}
    {\left(1 -Y^3\right)}
    {\left(1 - Y\right)} \bigr)$ &
    $X^{3} Y^{6}$
    \\
    m4\_3 &
    $1 /\bigl( {\left(1 - X Y\right)} {\left(1 - Y^4\right)} {\left(1 -
        Y^3\right)} {\left(1 - Y\right)}\bigr)$
    & $X Y^{9}$
    \\
    m5\_2\_1 &
    $\frac{1 -Y^6}
    {{\left(1 - X Y^{3} \right)}
      {\left(1 - X^{2} Y \right)}
      {\left(1 - X Y \right)}
      {\left(1 - Y^5 \right)}
      {\left(1 - Y^3 \right)}
      {\left(1 - Y \right)}}$ &
    $-X^{4} Y^{8}$
    \\
    m5\_2\_2 &
    $1/\bigl(
    {\left(1 - X^{3} Y\right)}
    {\left(1 - X^{2} Y\right)}
    {\left(1 - X Y\right)}
    {\left(1 - Y^3\right)}
    {\left(1 - Y\right)}
  \bigr)$ &
    $-X^{6} Y^{7}$
    \\
    m5\_2\_3 &
    $1 /\bigl(
    {\left(1 - X^{3} Y\right)}
    {\left(1 - X^{2} Y\right)}
    {\left(1 - X Y \right)}
    {\left(1 - Y^5\right)}
    {\left(1 - Y\right)}\bigr)$
    & $-X^{6} Y^{9}$
    \\
    m5\_3\_1 &
    $\frac{1 - Y^{8} }{
    {\left(1 - X Y^{4} \right)}
    {\left(1 - X Y \right)}
    {\left(1 - Y^5 \right)}
    {\left(1 - Y^4 \right)}
    {\left(1 - Y^3 \right)}
    {\left(1 - Y\right)}}$ &
    $-X^{2} Y^{10}$
    \\
    m5\_3\_2 &
    $1 /\bigl(
    {\left(1 - X^{2} Y\right)}
    {\left(1 - X Y\right)}
    {\left(1 - Y^4\right)}
    {\left(1 - Y^3\right)}
    {\left(1 - Y\right)}
    \bigr)$&
    $-X^{3} Y^{10}$
    \\
    m5\_4\_1 &
    $1 /\bigl(
    {\left(1 - X Y\right)}
    {\left(1 - Y^5\right)}
    {\left(1 - Y^4\right)}
    {\left(1 - Y^3\right)}
    {\left(1 - Y\right)}
    \bigr)$&
    $-X Y^{14}$
    \\
    m6\_2\_1 &
    $\frac{X Y^{8} + X Y^{5} + Y^{5} + X Y^{3} + Y^{3} + 1}{
      {\left(1 - X^{2} Y^{5} \right)}
      {\left(1 - X^{2} Y^{3} \right)}
      {\left(1 - X^{2} Y \right)}
      {\left(1 - X Y \right)}
      {\left(1 - Y^6\right)}
      {\left(1 - Y \right)}
    }$ &
    $X^{6} Y^{9}$
    \\
    m6\_2\_2 &
    $\frac{1 - Y^{6}}
    {{\left(1 - X Y^{3} \right)}
      {\left(1 - X^{3} Y \right)}
      {\left(1 - X^{2} Y  \right)}
      {\left(1 - X Y \right)}
      {\left(1 - Y^5 \right)}
      {\left(1 - Y^3\right)}
      {\left(1 - Y\right)}}
    $ &
    $X^{7} Y^{9}$
    \\
    m6\_2\_3 &
    $\frac{Y^{4} + Y^{3} + Y^{2} + Y + 1}
    {{\left(1 - X Y^{5} \right)}
      {\left(1 - X^{3} Y \right)}
      {\left(1 - X^{2} Y \right)}
      {\left(1 - X Y \right)}
      {\left(1 - Y^3 \right)}^{2}}$
    &
    $X^{7} Y^{10}$
    \\
    m6\_2\_4 &
    $\frac{1 - X Y^{8}}{
      {\left(1 - X Y^{6} \right)}
      {\left(1 - X Y^{5} \right)}
      {\left(1 - X^{3} Y \right)}
      {\left(1 - X^{2} Y \right)}
      {\left(1 - X Y \right)}
      {\left(1 - Y^3 \right)}
      {\left(1 - Y \right)}
    }$ &
    $X^{7} Y^{10}$
    \\
    m6\_2\_5 &
    $\frac 1{
    {\left(1 - X^{4} Y\right)}
    {\left(1 - X^{3} Y\right)}
    {\left(1 - X^{2} Y\right)}
    {\left(1 - X Y\right)}
    {\left(1 - Y^3\right)}
    {\left(1 - Y\right)}}$
      &
    $X^{10} Y^{8}$
    \\
    m6\_2\_6 &
    $\frac 1{
    {\left(1 - X^{4} Y\right)}
    {\left(1 - X^{3} Y\right)}
    {\left(1 - X^{2} Y\right)}
    {\left(1 - X Y\right)}
    {\left(1 - Y^5 \right)}
    {\left(1 - Y\right)}}$
    & $X^{10} Y^{10}$
    \\
    m6\_3\_1 &
    $\frac{1 - Y^{8}} {
      {\left(1 - X Y^{4} \right)}
      {\left(1 - X^{2} Y \right)}
      {\left(1 - X Y \right)}
      {\left(1 - Y^5\right)}
      {\left(1 - Y^4\right)}
      {\left(1 - Y^3\right)}
      {\left(1 - Y\right)}
      }$
    & $X^{4} Y^{11}$
    \\
    m6\_3\_2&
    $\frac{Y^{8} + Y^{7} + 2 \, Y^{6} + 2 \, Y^{5} + 2 \, Y^{4} + 2 \, Y^{3} +
      Y^{2} + Y + 1}{
      {\left(1 - X Y^{3} \right)}
      {\left(1 - X^{2} Y \right)} 
      {\left(1 - X Y \right)}
      {\left(1 - Y^6\right)}
      {\left(1 - Y^5 \right)}
      {\left(1 - Y^{4} \right)}
    }$ &
    \ding{55}
    \\
    m6\_3\_3  &
    same as for m6\_3\_2  &
    \ding{55}\\
    m6\_3\_4 &
    $1/\bigl(
    {\left(1 - X Y^{3} \right)}
    {\left(1 - X^{2} Y \right)}
    {\left(1 - X Y \right)}
    {\left(1 - Y^5\right)}
    {\left(1 - Y^3\right)}
    {\left(1 - Y\right)}
    \bigr)$ &
    $X^{4} Y^{14}$
    \\
    m6\_3\_5 & same as for m6\_3\_4 & $X^{4} Y^{14}$ \\
    m6\_3\_6 &
    $1 /\bigl(
    {\left(1 - X^{3} Y\right)}
    {\left(1 - X^{2} Y\right)}
    {\left(1 - X Y\right)}
    {\left(1 - Y^4\right)}
    {\left(1 - Y^3\right)}
    {\left(1 - Y\right)}
    \bigr)
      $
      & $X^{6} Y^{11}$
    \\
    m6\_4\_1 &
    $\frac{Y^{3} - Y + 1}{
      {\left(1 - X Y^{4} \right)}
      {\left(1 - X Y \right)}
      {\left(1 - Y^6 \right)}
      {\left(1 - Y^5 \right)}
      {\left(1 - Y \right)}^{2}}$ &
    \ding{55}
    \\
    m6\_4\_2 &
    $\frac{1 - Y^{8}}
    {{\left(1 - X Y^{4} \right)}
      {\left(1 - X Y \right)}
      {\left(1 - Y^6\right)}
      {\left(1 - Y^5 \right)}
      {\left(1 - Y^4\right)}
      {\left(1 - Y^3 \right)}
      {\left(1 - Y\right)}}$
    &
    $X^{2} Y^{16}$
    \\
    m6\_4\_3 & 
    $1 /\bigl(
    {\left(1 - X^{2} Y\right)}
    {\left(1 - X Y\right)}
    {\left(1 - Y^5\right)}
    {\left(1 - Y^4\right)}
    {\left(1 - Y^3\right)}
    {\left(1 - Y\right)}
      \bigr)$ & 
    $X^{3} Y^{15}$
    \\
    m6\_5\_1 & 
    $1 / \bigl(
    {\left(1 - X Y \right)}
    {\left(1 - Y^6 \right)}
    {\left(1 - Y^5 \right)}
    {\left(1 - Y^4 \right)}
    {\left(1 - Y^3 \right)}
    {\left(1 - Y \right)}
    \bigr)$ &
    $X Y^{20}$
    \\
    m6\_5\_2 &
    same as for m6\_5\_1 & $X Y^{20}$
  \end{tabular}
  \caption{Examples of generic local graded ideal zeta functions}
  \label{t:grid}
\end{table}}

{\small
\begin{table}[h]
  \centering
  \begin{tabular}{rll}
    $\bm\fg$ & $W(X,Y)$ s.t.\ $\zeta_{\fg\otimes \fO_K}^{\gr\le}(s) = W(q_K,q_K^{-s})$ & FEqn \\
      \hline
    m3\_2 & 
    $\frac{1 - XY^3}{(1 - XY^2)(1 - XY)(1-Y^2)(1-Y)}$  &
    $-XY^3$ \\
    m4\_2 &
    $\frac{1 - XY^3}{(1 - X^2Y)(1 - XY^2)(1 - XY)(1-Y^2)(1-Y)}$ &
    $X^3Y^4$ \\
    m4\_3 &
    $\frac{X^{2} Y^{9} + X^{2} Y^{7} + X^{2} Y^{6} - X Y^{6} - 2 \, X Y^{5} - 2
      \, X Y^{4} - X Y^{3} + Y^{3} + Y^{2} + 1}
    {{\left(1 - X Y^{3}\right)} {\left(1 - X Y^{2}\right)} {\left(1 - X
          Y\right)} {\left(1- Y^4\right)} {\left(1-Y^2\right)}
      {\left(1-Y\right)}}$ &
    $XY^4$ \\
    m5\_2\_1 &
    $\frac{-X^{2} Y^{5} -X^{2} Y^{3} - X Y^{3} + X Y^{2} + Y^{2} + 1}
    {{\left(1 - X^{2} Y\right)} {\left(1 - X Y^{2}\right)} {\left(1 - X^2
          Y^2\right)} {\left(1 - X Y\right)} {\left(1 - Y^3\right)}
      {\left(1 - Y\right)}}$
    &
    $-X^{4} Y^{5}$
    \\
    m5\_2\_2 &
    $\frac{1 - X Y^{3}}{
      {\left(1 - X Y^{2} \right)}
      {\left(1 - X^{3} Y\right)}
      {\left(1 - X^{2} Y \right)}
      {\left(1 - X Y \right)}
      {\left(1 - Y^2 \right)}
      {\left(1 - Y \right)}
    }$ &
    $-X^{6} Y^{5}$
    \\
    m5\_2\_3 &
    $ \frac{-X^{4} Y^{7} - X^{3} Y^{6} - X^{3} Y^{4} - X^{2} Y^{5} + X^{3} Y^{3}
      - X Y^{4} + X^{2} Y^{2} + X Y^{3} + X Y + 1}{
      {\left(1 - X^3 Y^3 \right)}
      {\left(1 - X^{2} Y^{3} \right)}
      {\left(1 - X^{3} Y \right)}
      {\left(1 - X^{2} Y \right)}
      {\left(1 - Y^3 \right)}
      {\left(1 - Y \right)}
    }$ 
    & $-X^{6} Y^{5}$
    \\
    m5\_3\_1 &
    $W_{531}$ \eqref{subalgebras:m5_3_1}
    & $-X^2Y^5$
    \\
    m5\_3\_2 &
    $\frac{X^{2} Y^{9} + X^{2} Y^{7} + X^{2} Y^{6} - X Y^{6} - 2 \, X Y^{5} - 2
      \, X Y^{4} - X Y^{3} + Y^{3} + Y^{2} + 1}
    {
      {\left(1 - X Y^{3} \right)}
      {\left(1 - X Y^{2} \right)}
      {\left(1 - X^{2} Y \right)}
      {\left(1 - X Y \right)}
      {\left(1 - Y^{4}\right)}
      {\left(1 - Y^2\right)}
      {\left(1 - Y\right)}
    }$ &
    $-X^{3} Y^{5}$
    \\
    m5\_4\_1 &
    $W_{541}$ \eqref{subalgebras:m5_4_1}
    & $-XY^5$
    \\
    m6\_2\_1 &
    $W_{621}$ \eqref{subalgebras:m6_2_1}
    & $X^{6} Y^{6}$
    \\
    m6\_2\_2 &
    $\frac{-X^{2} Y^{5} - X^{2} Y^{3} - X Y^{3} + X Y^{2} + Y^{2} + 1}{
      {\left(1 - X^2 Y^2 \right)}
      {\left(1 - X Y^{2} \right)}
      {\left(1 - X^{3} Y \right)}
      {\left(1 - X^{2} Y \right)}
      {\left(1 - X Y \right)}
      {\left(1 - Y^3\right)}
      {\left(1 - Y\right)}
    }$ &
    $X^{7} Y^{6}$
    \\
    m6\_2\_3 & 
    $W_{623}$ \eqref{subalgebras:m6_2_3} &
    $X^7Y^6$
    \\
    m6\_2\_5 &
    $\frac{1 - X Y^{3}}{
      {\left(1 - X Y^{2} \right)}
      {\left(1 - X^{4} Y \right)}
      {\left(1 - X^{3} Y \right)}
      {\left(1 - X^{2} Y \right)}
      {\left(1 - X Y \right)}
      {\left(1 - Y^2 \right)}
      {\left(1 - Y \right)}}
    $ &
    $X^{10} Y^{6}$
    \\
    m6\_2\_6 &
    $\frac{
      -X^{4} Y^{7} - X^{3} Y^{6} - X^{3} Y^{4} - X^{2} Y^{5} + X^{3} Y^{3} -
      X Y^{4} + X^{2} Y^{2} + X Y^{3} + X Y + 1}{
      {\left(1 - X^{3} Y^{3}\right)}
      {\left(1 - X^{2} Y^{3}\right)}
      {\left(1 - X^{4} Y\right)}
      {\left(1 - X^{3} Y \right)}
      {\left(1 - X^{2} Y\right)}
      {\left(1 - Y^3\right)}
      {\left(1 - Y\right)}
    }$ &
    $X^{10} Y^{6}$
    \\
    m6\_3\_1 & 
    $W_{631}$ \eqref{subalgebras:m6_3_1}
    & $X^4 Y^6$
    \\
    m6\_3\_2 &
    $W_{632}$ \eqref{subalgebras:m6_3_2}
    & $X^4 Y^6$
    \\
    m6\_3\_3 & 
    $W_{633}$ \eqref{subalgebras:m6_3_3}
    & $X^4 Y^6$
    \\
    m6\_3\_6 &
    $\frac{X^{2} Y^{9} + X^{2} Y^{7} + X^{2} Y^{6} - X Y^{6} - 2 \, X Y^{5} - 2
      \, X Y^{4} - X Y^{3} + Y^{3} + Y^{2} + 1}{
      {\left(1 - X Y^{3} \right)} 
      {\left(1 - X Y^{2} \right)}
      {\left(1 - X^{3} Y \right)}
      {\left(1 - X^{2} Y \right)}
      {\left(1 - X Y \right)}
      {\left(1 - Y^{4} \right)}
      {\left(1 - Y^2\right)}
      {\left(1 - Y\right)}
    }$ & $X^{6} Y^{6}$
    \\
    m6\_4\_3 &
    $W_{643}$ \eqref{subalgebras:m6_4_3}
    & $X^3 Y^6$
  \end{tabular}
  \caption{Examples of generic local graded subalgebra zeta functions}
  \label{t:grsub}
\end{table}}

\appendix

\section{Large numerators of  local subobject zeta functions}
\label{s:numerators}

{\scriptsize\begin{align*}
  F_{\Uni_5}  = \,\, &
X^{43} Y^{124} + X^{42} Y^{121} -  X^{42} Y^{120} -  X^{42} Y^{119} - 2 X^{42} Y^{118} + 2 X^{41} Y^{118} - 3 X^{41} Y^{117} \\
& + X^{42} Y^{115} - 2 X^{41} Y^{116} + X^{42} Y^{114} - 3 X^{41} Y^{115} - 2 X^{40} Y^{116} -  X^{42} Y^{113} -  X^{41} Y^{114} \\
& + 2 X^{40} Y^{115} + 4 X^{41} Y^{113} - 2 X^{40} Y^{114} -  X^{39} Y^{115} - 2 X^{40} Y^{113} - 2 X^{39} Y^{114} + X^{41} Y^{111} \\
& + 6 X^{40} Y^{112} - 3 X^{39} Y^{113} -  X^{41} Y^{110} + X^{40} Y^{111} - 4 X^{39} Y^{112} + 5 X^{40} Y^{110} + 6 X^{39} Y^{111} \\
& + X^{38} Y^{112} + 3 X^{39} Y^{110} - 6 X^{38} Y^{111} -  X^{40} Y^{108} + 8 X^{39} Y^{109} + 2 X^{38} Y^{110} - 2 X^{40} Y^{107} \\
& + 4 X^{39} Y^{108} + 5 X^{38} Y^{109} - 3 X^{37} Y^{110} + X^{39} Y^{107} + 9 X^{38} Y^{108} - 4 X^{39} Y^{106} + 8 X^{38} Y^{107} \\
& + X^{37} Y^{108} - 4 X^{39} Y^{105} + 3 X^{38} Y^{106} + 6 X^{37} Y^{107} - 2 X^{36} Y^{108} -  X^{39} Y^{104} - 5 X^{38} Y^{105} \\
& + 13 X^{37} Y^{106} + 2 X^{36} Y^{107} - 8 X^{38} Y^{104} + 9 X^{37} Y^{105} + 3 X^{36} Y^{106} -  X^{35} Y^{107} + 2 X^{39} Y^{102} \\
& -7 X^{38} Y^{103} - 3 X^{37} Y^{104} + 8 X^{36} Y^{105} + X^{35} Y^{106} - 6 X^{38} Y^{102} - 15 X^{37} Y^{103} + 5 X^{36} Y^{104} \\
& - X^{35} Y^{105} + 4 X^{38} Y^{101} - 15 X^{37} Y^{102} + 6 X^{36} Y^{103} + 12 X^{35} Y^{104} -  X^{38} Y^{100} - 16 X^{37} Y^{101} \\
& -10 X^{36} Y^{102} + 7 X^{35} Y^{103} - 2 X^{34} Y^{104} + 3 X^{38} Y^{99} + 4 X^{37} Y^{100} - 22 X^{36} Y^{101} + 8 X^{35} Y^{102} \\
& + 5 X^{34} Y^{103} + X^{38} Y^{98} - 28 X^{36} Y^{100} - 8 X^{35} Y^{101} + 8 X^{37} Y^{98} - 8 X^{36} Y^{99} - 19 X^{35} Y^{100} \\
& + 13 X^{34} Y^{101} + 2 X^{33} Y^{102} + X^{37} Y^{97} - 2 X^{36} Y^{98} - 30 X^{35} Y^{99} + X^{34} Y^{100} + X^{33} Y^{101} \\
& - X^{37} Y^{96} + 17 X^{36} Y^{97} - 19 X^{35} Y^{98} - 16 X^{34} Y^{99} + 6 X^{33} Y^{100} + 17 X^{36} Y^{96} - 6 X^{35} Y^{97} \\
& -34 X^{34} Y^{98} + 2 X^{33} Y^{99} + X^{32} Y^{100} -  X^{37} Y^{94} + 7 X^{36} Y^{95} + 16 X^{35} Y^{96} - 32 X^{34} Y^{97} \\
& + 2 X^{32} Y^{99} -  X^{36} Y^{94} + 32 X^{35} Y^{95} - 15 X^{34} Y^{96} - 21 X^{33} Y^{97} + 2 X^{32} Y^{98} - 5 X^{36} Y^{93} \\
& + 16 X^{35} Y^{94} + 4 X^{34} Y^{95} - 32 X^{33} Y^{96} - 2 X^{32} Y^{97} -  X^{36} Y^{92} + 17 X^{35} Y^{93} + 56 X^{34} Y^{94} \\
& -22 X^{33} Y^{95} - 13 X^{32} Y^{96} + 3 X^{31} Y^{97} -  X^{36} Y^{91} - 5 X^{35} Y^{92} + 30 X^{34} Y^{93} - 13 X^{33} Y^{94} \\
& -22 X^{32} Y^{95} - 7 X^{35} Y^{91} + 37 X^{34} Y^{92} + 59 X^{33} Y^{93} - 25 X^{32} Y^{94} - 5 X^{31} Y^{95} - 3 X^{35} Y^{90} \\
& -5 X^{34} Y^{91} + 41 X^{33} Y^{92} - 29 X^{32} Y^{93} - 17 X^{31} Y^{94} -  X^{35} Y^{89} - 17 X^{34} Y^{90} + 69 X^{33} Y^{91} \\
& + 41 X^{32} Y^{92} - 22 X^{31} Y^{93} + X^{30} Y^{94} - 11 X^{34} Y^{89} + 18 X^{33} Y^{90} + 48 X^{32} Y^{91} - 26 X^{31} Y^{92} \\
& -5 X^{30} Y^{93} - 12 X^{34} Y^{88} - 26 X^{33} Y^{89} + 89 X^{32} Y^{90} + 13 X^{31} Y^{91} - 10 X^{30} Y^{92} -  X^{29} Y^{93} \\
& + 2 X^{34} Y^{87} - 26 X^{33} Y^{88} + 52 X^{32} Y^{89} + 33 X^{31} Y^{90} - 23 X^{30} Y^{91} - 3 X^{29} Y^{92} + 3 X^{34} Y^{86} \\
& -34 X^{33} Y^{87} - 24 X^{32} Y^{88} + 82 X^{31} Y^{89} - 4 X^{29} Y^{91} + X^{34} Y^{85} - 10 X^{33} Y^{86} - 38 X^{32} Y^{87} \\
& + 88 X^{31} Y^{88} + 24 X^{30} Y^{89} - 11 X^{29} Y^{90} - 3 X^{33} Y^{85} - 69 X^{32} Y^{86} + 4 X^{31} Y^{87} + 66 X^{30} Y^{88} \\
& -6 X^{29} Y^{89} - 2 X^{28} Y^{90} + 7 X^{33} Y^{84} - 30 X^{32} Y^{85} - 31 X^{31} Y^{86} + 101 X^{30} Y^{87} - 5 X^{28} Y^{89} \\
& + 4 X^{33} Y^{83} - 11 X^{32} Y^{84} - 103 X^{31} Y^{85} + 28 X^{30} Y^{86} + 37 X^{29} Y^{87} - 4 X^{28} Y^{88} + 10 X^{32} Y^{83} \\
& -77 X^{31} Y^{84} - 5 X^{30} Y^{85} + 91 X^{29} Y^{86} - 2 X^{28} Y^{87} + X^{33} Y^{81} + 9 X^{32} Y^{82} - 40 X^{31} Y^{83} \\
& -99 X^{30} Y^{84} + 53 X^{29} Y^{85} + 21 X^{28} Y^{86} - 2 X^{27} Y^{87} + 4 X^{32} Y^{81} + 9 X^{31} Y^{82} - 115 X^{30} Y^{83} \\
& + 20 X^{29} Y^{84} + 53 X^{28} Y^{85} - 6 X^{27} Y^{86} + 4 X^{32} Y^{80} + 32 X^{31} Y^{81} - 78 X^{30} Y^{82} - 80 X^{29} Y^{83} \\
& + 56 X^{28} Y^{84} + 5 X^{27} Y^{85} - 2 X^{32} Y^{79} + 22 X^{31} Y^{80} - 18 X^{30} Y^{81} - 148 X^{29} Y^{82} + 44 X^{28} Y^{83} \\
& + 30 X^{27} Y^{84} -  X^{26} Y^{85} + 11 X^{31} Y^{79} + 42 X^{30} Y^{80} - 114 X^{29} Y^{81} - 25 X^{28} Y^{82} + 46 X^{27} Y^{83} \\
& + 3 X^{31} Y^{78} + 56 X^{30} Y^{79} - 64 X^{29} Y^{80} - 138 X^{28} Y^{81} + 32 X^{27} Y^{82} + 6 X^{26} Y^{83} -  X^{31} Y^{77} \\
& + 36 X^{30} Y^{78} + 35 X^{29} Y^{79} - 143 X^{28} Y^{80} + 10 X^{27} Y^{81} + 22 X^{26} Y^{82} - 3 X^{31} Y^{76} + 14 X^{30} Y^{77} \\
& + 89 X^{29} Y^{78} - 118 X^{28} Y^{79} - 100 X^{27} Y^{80} + 29 X^{26} Y^{81} + 4 X^{25} Y^{82} + X^{31} Y^{75} - 3 X^{30} Y^{76} \\
& + 76 X^{29} Y^{77} - 130 X^{27} Y^{79} + 33 X^{26} Y^{80} + 10 X^{25} Y^{81} - 7 X^{30} Y^{75} + 50 X^{29} Y^{76} + 107 X^{28} Y^{77} \\
& -152 X^{27} Y^{78} - 62 X^{26} Y^{79} + 13 X^{25} Y^{80} + X^{30} Y^{74} + 6 X^{29} Y^{75} + 121 X^{28} Y^{76} - 71 X^{27} Y^{77} \\
& -99 X^{26} Y^{78} + 23 X^{25} Y^{79} + X^{24} Y^{80} - 3 X^{30} Y^{73} - 19 X^{29} Y^{74} + 100 X^{28} Y^{75} + 88 X^{27} Y^{76} \\
& -137 X^{26} Y^{77} - 16 X^{25} Y^{78} + 8 X^{24} Y^{79} - 12 X^{29} Y^{73} + 34 X^{28} Y^{74} + 145 X^{27} Y^{75} - 119 X^{26} Y^{76} \\
& -53 X^{25} Y^{77} + 16 X^{24} Y^{78} - 7 X^{29} Y^{72} - 16 X^{28} Y^{73} + 167 X^{27} Y^{74} + 47 X^{26} Y^{75} - 107 X^{25} Y^{76} \\
& -3 X^{24} Y^{77} + X^{23} Y^{78} -  X^{29} Y^{71} - 28 X^{28} Y^{72} + 84 X^{27} Y^{73} + 118 X^{26} Y^{74} - 132 X^{25} Y^{75} \\
& -29 X^{24} Y^{76} + 4 X^{23} Y^{77} - 2 X^{29} Y^{70} - 27 X^{28} Y^{71} - 12 X^{27} Y^{72} + 209 X^{26} Y^{73} + 6 X^{25} Y^{74} \\
& -55 X^{24} Y^{75} + 5 X^{23} Y^{76} - 13 X^{28} Y^{70} - 54 X^{27} Y^{71} + 156 X^{26} Y^{72} + 80 X^{25} Y^{73} - 117 X^{24} Y^{74} \\
& -6 X^{23} Y^{75} + X^{22} Y^{76} -  X^{28} Y^{69} - 57 X^{27} Y^{70} + 38 X^{26} Y^{71} + 213 X^{25} Y^{72} - 39 X^{24} Y^{73} \\
& -19 X^{23} Y^{74} + X^{22} Y^{75} + 2 X^{28} Y^{68} - 33 X^{27} Y^{69} - 51 X^{26} Y^{70} + 208 X^{25} Y^{71} + 28 X^{24} Y^{72} \\
& -77 X^{23} Y^{73} - 4 X^{22} Y^{74} - 13 X^{27} Y^{68} - 114 X^{26} Y^{69} + 107 X^{25} Y^{70} + 164 X^{24} Y^{71} - 41 X^{23} Y^{72} \\
& -2 X^{22} Y^{73} - 80 X^{26} Y^{68} - 11 X^{25} Y^{69} + 221 X^{24} Y^{70} - 5 X^{23} Y^{71} - 35 X^{22} Y^{72} -  X^{21} Y^{73} \\
& + 6 X^{27} Y^{66} - 35 X^{26} Y^{67} - 132 X^{25} Y^{68} + 154 X^{24} Y^{69} + 87 X^{23} Y^{70} - 36 X^{22} Y^{71} -  X^{21} Y^{72} \\
& + 3 X^{27} Y^{65} - 131 X^{25} Y^{67} + 44 X^{24} Y^{68} + 200 X^{23} Y^{69} - 22 X^{22} Y^{70} - 14 X^{21} Y^{71} + 10 X^{26} Y^{65} \\
& -97 X^{25} Y^{66} - 124 X^{24} Y^{67} + 176 X^{23} Y^{68} + 33 X^{22} Y^{69} - 21 X^{21} Y^{70} + 11 X^{26} Y^{64} - 18 X^{25} Y^{65} \\
& -177 X^{24} Y^{66} + 109 X^{23} Y^{67} + 135 X^{22} Y^{68} - 15 X^{21} Y^{69} -  X^{20} Y^{70} + 7 X^{26} Y^{63} + 21 X^{25} Y^{64} \\
& -162 X^{24} Y^{65} - 73 X^{23} Y^{66} + 144 X^{22} Y^{67} - 3 X^{21} Y^{68} - 8 X^{20} Y^{69} + 2 X^{26} Y^{62} + 35 X^{25} Y^{63} \\
& -62 X^{24} Y^{64} - 189 X^{23} Y^{65} + 145 X^{22} Y^{66} + 71 X^{21} Y^{67} - 8 X^{20} Y^{68} + X^{26} Y^{61} + 16 X^{25} Y^{62} \\
& - X^{24} Y^{63} - 216 X^{23} Y^{64} - 7 X^{22} Y^{65} + 101 X^{21} Y^{66} - 9 X^{20} Y^{67} - 2 X^{19} Y^{68} + 5 X^{25} Y^{61} \\
& + 63 X^{24} Y^{62} - 131 X^{23} Y^{63} - 155 X^{22} Y^{64} + 132 X^{21} Y^{65} + 36 X^{20} Y^{66} + 5 X^{25} Y^{60} + 51 X^{24} Y^{61} \\
& -32 X^{23} Y^{62} - 241 X^{22} Y^{63} + 48 X^{21} Y^{64} + 53 X^{20} Y^{65} - 8 X^{19} Y^{66} - 2 X^{25} Y^{59} + 21 X^{24} Y^{60} \\
& + 76 X^{23} Y^{61} - 205 X^{22} Y^{62} - 92 X^{21} Y^{63} + 86 X^{20} Y^{64} + 10 X^{19} Y^{65} + 10 X^{24} Y^{59} + 86 X^{23} Y^{60} \\
& -92 X^{22} Y^{61} - 205 X^{21} Y^{62} + 76 X^{20} Y^{63} + 21 X^{19} Y^{64} - 2 X^{18} Y^{65} - 8 X^{24} Y^{58} + 53 X^{23} Y^{59} \\
& + 48 X^{22} Y^{60} - 241 X^{21} Y^{61} - 32 X^{20} Y^{62} + 51 X^{19} Y^{63} + 5 X^{18} Y^{64} + 36 X^{23} Y^{58} + 132 X^{22} Y^{59} \\
& -155 X^{21} Y^{60} - 131 X^{20} Y^{61} + 63 X^{19} Y^{62} + 5 X^{18} Y^{63} - 2 X^{24} Y^{56} - 9 X^{23} Y^{57} + 101 X^{22} Y^{58} \\
& -7 X^{21} Y^{59} - 216 X^{20} Y^{60} -  X^{19} Y^{61} + 16 X^{18} Y^{62} + X^{17} Y^{63} - 8 X^{23} Y^{56} + 71 X^{22} Y^{57} \\
& + 145 X^{21} Y^{58} - 189 X^{20} Y^{59} - 62 X^{19} Y^{60} + 35 X^{18} Y^{61} + 2 X^{17} Y^{62} - 8 X^{23} Y^{55} - 3 X^{22} Y^{56} \\
& + 144 X^{21} Y^{57} - 73 X^{20} Y^{58} - 162 X^{19} Y^{59} + 21 X^{18} Y^{60} + 7 X^{17} Y^{61} -  X^{23} Y^{54} - 15 X^{22} Y^{55} \\
& + 135 X^{21} Y^{56} + 109 X^{20} Y^{57} - 177 X^{19} Y^{58} - 18 X^{18} Y^{59} + 11 X^{17} Y^{60} - 21 X^{22} Y^{54} + 33 X^{21} Y^{55} \\
& + 176 X^{20} Y^{56} - 124 X^{19} Y^{57} - 97 X^{18} Y^{58} + 10 X^{17} Y^{59} - 14 X^{22} Y^{53} - 22 X^{21} Y^{54} + 200 X^{20} Y^{55} \\
& + 44 X^{19} Y^{56} - 131 X^{18} Y^{57} + 3 X^{16} Y^{59} -  X^{22} Y^{52} - 36 X^{21} Y^{53} + 87 X^{20} Y^{54} + 154 X^{19} Y^{55} \\
& -132 X^{18} Y^{56} - 35 X^{17} Y^{57} + 6 X^{16} Y^{58} -  X^{22} Y^{51} - 35 X^{21} Y^{52} - 5 X^{20} Y^{53} + 221 X^{19} Y^{54} \\
& -11 X^{18} Y^{55} - 80 X^{17} Y^{56} - 2 X^{21} Y^{51} - 41 X^{20} Y^{52} + 164 X^{19} Y^{53} + 107 X^{18} Y^{54} - 114 X^{17} Y^{55} \\
& -13 X^{16} Y^{56} - 4 X^{21} Y^{50} - 77 X^{20} Y^{51} + 28 X^{19} Y^{52} + 208 X^{18} Y^{53} - 51 X^{17} Y^{54} - 33 X^{16} Y^{55} \\
& + 2 X^{15} Y^{56} + X^{21} Y^{49} - 19 X^{20} Y^{50} - 39 X^{19} Y^{51} + 213 X^{18} Y^{52} + 38 X^{17} Y^{53} - 57 X^{16} Y^{54} \\
& - X^{15} Y^{55} + X^{21} Y^{48} - 6 X^{20} Y^{49} - 117 X^{19} Y^{50} + 80 X^{18} Y^{51} + 156 X^{17} Y^{52} - 54 X^{16} Y^{53} \\
& -13 X^{15} Y^{54} + 5 X^{20} Y^{48} - 55 X^{19} Y^{49} + 6 X^{18} Y^{50} + 209 X^{17} Y^{51} - 12 X^{16} Y^{52} - 27 X^{15} Y^{53} \\
& -2 X^{14} Y^{54} + 4 X^{20} Y^{47} - 29 X^{19} Y^{48} - 132 X^{18} Y^{49} + 118 X^{17} Y^{50} + 84 X^{16} Y^{51} - 28 X^{15} Y^{52} \\
& - X^{14} Y^{53} + X^{20} Y^{46} - 3 X^{19} Y^{47} - 107 X^{18} Y^{48} + 47 X^{17} Y^{49} + 167 X^{16} Y^{50} - 16 X^{15} Y^{51} \\
& -7 X^{14} Y^{52} + 16 X^{19} Y^{46} - 53 X^{18} Y^{47} - 119 X^{17} Y^{48} + 145 X^{16} Y^{49} + 34 X^{15} Y^{50} - 12 X^{14} Y^{51} \\
& + 8 X^{19} Y^{45} - 16 X^{18} Y^{46} - 137 X^{17} Y^{47} + 88 X^{16} Y^{48} + 100 X^{15} Y^{49} - 19 X^{14} Y^{50} - 3 X^{13} Y^{51} \\
& + X^{19} Y^{44} + 23 X^{18} Y^{45} - 99 X^{17} Y^{46} - 71 X^{16} Y^{47} + 121 X^{15} Y^{48} + 6 X^{14} Y^{49} + X^{13} Y^{50} \\
& + 13 X^{18} Y^{44} - 62 X^{17} Y^{45} - 152 X^{16} Y^{46} + 107 X^{15} Y^{47} + 50 X^{14} Y^{48} - 7 X^{13} Y^{49} + 10 X^{18} Y^{43} \\
& + 33 X^{17} Y^{44} - 130 X^{16} Y^{45} + 76 X^{14} Y^{47} - 3 X^{13} Y^{48} + X^{12} Y^{49} + 4 X^{18} Y^{42} + 29 X^{17} Y^{43} \\
& -100 X^{16} Y^{44} - 118 X^{15} Y^{45} + 89 X^{14} Y^{46} + 14 X^{13} Y^{47} - 3 X^{12} Y^{48} + 22 X^{17} Y^{42} + 10 X^{16} Y^{43} \\
& -143 X^{15} Y^{44} + 35 X^{14} Y^{45} + 36 X^{13} Y^{46} -  X^{12} Y^{47} + 6 X^{17} Y^{41} + 32 X^{16} Y^{42} - 138 X^{15} Y^{43} \\
& -64 X^{14} Y^{44} + 56 X^{13} Y^{45} + 3 X^{12} Y^{46} + 46 X^{16} Y^{41} - 25 X^{15} Y^{42} - 114 X^{14} Y^{43} + 42 X^{13} Y^{44} \\
& + 11 X^{12} Y^{45} -  X^{17} Y^{39} + 30 X^{16} Y^{40} + 44 X^{15} Y^{41} - 148 X^{14} Y^{42} - 18 X^{13} Y^{43} + 22 X^{12} Y^{44} \\
& -2 X^{11} Y^{45} + 5 X^{16} Y^{39} + 56 X^{15} Y^{40} - 80 X^{14} Y^{41} - 78 X^{13} Y^{42} + 32 X^{12} Y^{43} + 4 X^{11} Y^{44} \\
& -6 X^{16} Y^{38} + 53 X^{15} Y^{39} + 20 X^{14} Y^{40} - 115 X^{13} Y^{41} + 9 X^{12} Y^{42} + 4 X^{11} Y^{43} - 2 X^{16} Y^{37} \\
& + 21 X^{15} Y^{38} + 53 X^{14} Y^{39} - 99 X^{13} Y^{40} - 40 X^{12} Y^{41} + 9 X^{11} Y^{42} + X^{10} Y^{43} - 2 X^{15} Y^{37} \\
& + 91 X^{14} Y^{38} - 5 X^{13} Y^{39} - 77 X^{12} Y^{40} + 10 X^{11} Y^{41} - 4 X^{15} Y^{36} + 37 X^{14} Y^{37} + 28 X^{13} Y^{38} \\
& -103 X^{12} Y^{39} - 11 X^{11} Y^{40} + 4 X^{10} Y^{41} - 5 X^{15} Y^{35} + 101 X^{13} Y^{37} - 31 X^{12} Y^{38} - 30 X^{11} Y^{39} \\
& + 7 X^{10} Y^{40} - 2 X^{15} Y^{34} - 6 X^{14} Y^{35} + 66 X^{13} Y^{36} + 4 X^{12} Y^{37} - 69 X^{11} Y^{38} - 3 X^{10} Y^{39} \\
& -11 X^{14} Y^{34} + 24 X^{13} Y^{35} + 88 X^{12} Y^{36} - 38 X^{11} Y^{37} - 10 X^{10} Y^{38} + X^{9} Y^{39} - 4 X^{14} Y^{33} \\
& + 82 X^{12} Y^{35} - 24 X^{11} Y^{36} - 34 X^{10} Y^{37} + 3 X^{9} Y^{38} - 3 X^{14} Y^{32} - 23 X^{13} Y^{33} + 33 X^{12} Y^{34} \\
& + 52 X^{11} Y^{35} - 26 X^{10} Y^{36} + 2 X^{9} Y^{37} -  X^{14} Y^{31} - 10 X^{13} Y^{32} + 13 X^{12} Y^{33} + 89 X^{11} Y^{34} \\
& -26 X^{10} Y^{35} - 12 X^{9} Y^{36} - 5 X^{13} Y^{31} - 26 X^{12} Y^{32} + 48 X^{11} Y^{33} + 18 X^{10} Y^{34} - 11 X^{9} Y^{35} \\
& + X^{13} Y^{30} - 22 X^{12} Y^{31} + 41 X^{11} Y^{32} + 69 X^{10} Y^{33} - 17 X^{9} Y^{34} -  X^{8} Y^{35} - 17 X^{12} Y^{30} \\
& -29 X^{11} Y^{31} + 41 X^{10} Y^{32} - 5 X^{9} Y^{33} - 3 X^{8} Y^{34} - 5 X^{12} Y^{29} - 25 X^{11} Y^{30} + 59 X^{10} Y^{31} \\
& + 37 X^{9} Y^{32} - 7 X^{8} Y^{33} - 22 X^{11} Y^{29} - 13 X^{10} Y^{30} + 30 X^{9} Y^{31} - 5 X^{8} Y^{32} -  X^{7} Y^{33} \\
& + 3 X^{12} Y^{27} - 13 X^{11} Y^{28} - 22 X^{10} Y^{29} + 56 X^{9} Y^{30} + 17 X^{8} Y^{31} -  X^{7} Y^{32} - 2 X^{11} Y^{27} \\
& -32 X^{10} Y^{28} + 4 X^{9} Y^{29} + 16 X^{8} Y^{30} - 5 X^{7} Y^{31} + 2 X^{11} Y^{26} - 21 X^{10} Y^{27} - 15 X^{9} Y^{28} \\
& + 32 X^{8} Y^{29} -  X^{7} Y^{30} + 2 X^{11} Y^{25} - 32 X^{9} Y^{27} + 16 X^{8} Y^{28} + 7 X^{7} Y^{29} -  X^{6} Y^{30} \\
& + X^{11} Y^{24} + 2 X^{10} Y^{25} - 34 X^{9} Y^{26} - 6 X^{8} Y^{27} + 17 X^{7} Y^{28} + 6 X^{10} Y^{24} - 16 X^{9} Y^{25} \\
& -19 X^{8} Y^{26} + 17 X^{7} Y^{27} -  X^{6} Y^{28} + X^{10} Y^{23} + X^{9} Y^{24} - 30 X^{8} Y^{25} - 2 X^{7} Y^{26} \\
& + X^{6} Y^{27} + 2 X^{10} Y^{22} + 13 X^{9} Y^{23} - 19 X^{8} Y^{24} - 8 X^{7} Y^{25} + 8 X^{6} Y^{26} - 8 X^{8} Y^{23} \\
& -28 X^{7} Y^{24} + X^{5} Y^{26} + 5 X^{9} Y^{21} + 8 X^{8} Y^{22} - 22 X^{7} Y^{23} + 4 X^{6} Y^{24} + 3 X^{5} Y^{25} \\
& -2 X^{9} Y^{20} + 7 X^{8} Y^{21} - 10 X^{7} Y^{22} - 16 X^{6} Y^{23} -  X^{5} Y^{24} + 12 X^{8} Y^{20} + 6 X^{7} Y^{21} \\
& -15 X^{6} Y^{22} + 4 X^{5} Y^{23} -  X^{8} Y^{19} + 5 X^{7} Y^{20} - 15 X^{6} Y^{21} - 6 X^{5} Y^{22} + X^{8} Y^{18} \\
& + 8 X^{7} Y^{19} - 3 X^{6} Y^{20} - 7 X^{5} Y^{21} + 2 X^{4} Y^{22} -  X^{8} Y^{17} + 3 X^{7} Y^{18} + 9 X^{6} Y^{19} \\
& -8 X^{5} Y^{20} + 2 X^{7} Y^{17} + 13 X^{6} Y^{18} - 5 X^{5} Y^{19} -  X^{4} Y^{20} - 2 X^{7} Y^{16} + 6 X^{6} Y^{17} \\
& + 3 X^{5} Y^{18} - 4 X^{4} Y^{19} + X^{6} Y^{16} + 8 X^{5} Y^{17} - 4 X^{4}
Y^{18} + 9 X^{5} Y^{16} + X^{4} Y^{17}  -3 X^{6} Y^{14} \\&   + 5 X^{5} Y^{15} + 4
X^{4} Y^{16} - 2 X^{3} Y^{17} + 2 X^{5} Y^{14} + 8 X^{4} Y^{15} -  X^{3} Y^{16}
 -6 X^{5} Y^{13} + 3 X^{4} Y^{14} \\& + X^{5} Y^{12} + 6 X^{4} Y^{13} + 5 X^{3}
 Y^{14} - 4 X^{4} Y^{12} + X^{3} Y^{13}  - X^{2} Y^{14} - 3 X^{4} Y^{11} + 6
 X^{3} Y^{12}  + X^{2} Y^{13} \\& - 2 X^{4} Y^{10} - 2 X^{3} Y^{11} -  X^{4}
 Y^{9}  -2 X^{3} Y^{10} + 4 X^{2} Y^{11} + 2 X^{3} Y^{9} -  X^{2} Y^{10} -  X
 Y^{11}  - 2 X^{3} Y^{8} \\& - 3 X^{2} Y^{9}  + X Y^{10} - 2 X^{2} Y^{8} + X Y^{9} -
 3 X^{2} Y^{7} + 2 X^{2} Y^{6} - 2 X Y^{6} -  X Y^{5} - X Y^{4} + X Y^{3} + 1 
\end{align*}}

{\scriptsize\begin{align*}
\label{FX4}
F_{\QQ[T]/T^4}  = \, &
- X^{49} Y^{54} -  X^{49} Y^{53} - 2 X^{48} Y^{52} -  X^{47} Y^{52} - 3 X^{47} Y^{51} - 2 X^{46} Y^{51} - 3 X^{46} Y^{50} \\
& -4 X^{45} Y^{50} + X^{46} Y^{48} - 4 X^{45} Y^{49} + X^{45} Y^{48} - 4 X^{44} Y^{49} + X^{45} Y^{47} - 2 X^{44} Y^{48} \\
& + 3 X^{44} Y^{47} - 8 X^{43} Y^{48} + X^{44} Y^{46} -  X^{43} Y^{47} + 8 X^{43} Y^{46} - 9 X^{42} Y^{47} + X^{43} Y^{45} \\
& + 3 X^{42} Y^{46} + 9 X^{42} Y^{45} - 12 X^{41} Y^{46} + X^{42} Y^{44} + 10 X^{41} Y^{45} + 10 X^{41} Y^{44} - 13 X^{40} Y^{45} \\
& + 23 X^{40} Y^{44} + 7 X^{40} Y^{43} - 19 X^{39} Y^{44} - 3 X^{40} Y^{42} + 35 X^{39} Y^{43} + 3 X^{39} Y^{42} - 19 X^{38} Y^{43} \\
& -3 X^{39} Y^{41} + 54 X^{38} Y^{42} - 15 X^{38} Y^{41} - 24 X^{37} Y^{42} - 6 X^{38} Y^{40} + 74 X^{37} Y^{41} - 31 X^{37} Y^{40} \\
& -25 X^{36} Y^{41} - 5 X^{37} Y^{39} + 95 X^{36} Y^{40} - 55 X^{36} Y^{39} - 30 X^{35} Y^{40} - 4 X^{36} Y^{38} + 110 X^{35} Y^{39} \\
& - X^{36} Y^{37} - 85 X^{35} Y^{38} - 28 X^{34} Y^{39} + 10 X^{35} Y^{37} + 131 X^{34} Y^{38} - 3 X^{35} Y^{36} - 127 X^{34} Y^{37} \\
& -31 X^{33} Y^{38} + 22 X^{34} Y^{36} + 143 X^{33} Y^{37} - 4 X^{34} Y^{35} - 160 X^{33} Y^{36} - 29 X^{32} Y^{37} + 46 X^{33} Y^{35} \\
& + 154 X^{32} Y^{36} - 8 X^{33} Y^{34} - 204 X^{32} Y^{35} - 30 X^{31} Y^{36} + 73 X^{32} Y^{34} + 159 X^{31} Y^{35} - 11 X^{32} Y^{33} \\
& -246 X^{31} Y^{34} - 26 X^{30} Y^{35} + X^{32} Y^{32} + 113 X^{31} Y^{33} + 169 X^{30} Y^{34} - 19 X^{31} Y^{32} - 290 X^{30} Y^{33} \\
& -27 X^{29} Y^{34} + X^{31} Y^{31} + 148 X^{30} Y^{32} + 166 X^{29} Y^{33} - 26 X^{30} Y^{31} - 314 X^{29} Y^{32} - 23 X^{28} Y^{33} \\
& + 3 X^{30} Y^{30} + 193 X^{29} Y^{31} + 162 X^{28} Y^{32} - 39 X^{29} Y^{30} - 344 X^{28} Y^{31} - 22 X^{27} Y^{32} + 3 X^{29} Y^{29} \\
& + 230 X^{28} Y^{30} + 153 X^{27} Y^{31} - 49 X^{28} Y^{29} - 354 X^{27} Y^{30} - 17 X^{26} Y^{31} + 6 X^{28} Y^{28} + 271 X^{27} Y^{29} \\
& + 142 X^{26} Y^{30} - 68 X^{27} Y^{28} - 359 X^{26} Y^{29} - 16 X^{25} Y^{30} + 6 X^{27} Y^{27} + 301 X^{26} Y^{28} + 121 X^{25} Y^{29} \\
& -85 X^{26} Y^{27} - 344 X^{25} Y^{28} - 11 X^{24} Y^{29} + 10 X^{26} Y^{26} + 332 X^{25} Y^{27} + 104 X^{24} Y^{28} - 104 X^{25} Y^{26} \\
& -332 X^{24} Y^{27} - 10 X^{23} Y^{28} + 11 X^{25} Y^{25} + 344 X^{24} Y^{26} + 85 X^{23} Y^{27} - 121 X^{24} Y^{25} - 301 X^{23} Y^{26} \\
& -6 X^{22} Y^{27} + 16 X^{24} Y^{24} + 359 X^{23} Y^{25} + 68 X^{22} Y^{26} - 142 X^{23} Y^{24} - 271 X^{22} Y^{25} - 6 X^{21} Y^{26} \\
& + 17 X^{23} Y^{23} + 354 X^{22} Y^{24} + 49 X^{21} Y^{25} - 153 X^{22} Y^{23} - 230 X^{21} Y^{24} - 3 X^{20} Y^{25} + 22 X^{22} Y^{22} \\
& + 344 X^{21} Y^{23} + 39 X^{20} Y^{24} - 162 X^{21} Y^{22} - 193 X^{20} Y^{23} - 3 X^{19} Y^{24} + 23 X^{21} Y^{21} + 314 X^{20} Y^{22} \\
& + 26 X^{19} Y^{23} - 166 X^{20} Y^{21} - 148 X^{19} Y^{22} -  X^{18} Y^{23} + 27 X^{20} Y^{20} + 290 X^{19} Y^{21} + 19 X^{18} Y^{22} \\
& -169 X^{19} Y^{20} - 113 X^{18} Y^{21} -  X^{17} Y^{22} + 26 X^{19} Y^{19} + 246 X^{18} Y^{20} + 11 X^{17} Y^{21} - 159 X^{18} Y^{19} \\
& -73 X^{17} Y^{20} + 30 X^{18} Y^{18} + 204 X^{17} Y^{19} + 8 X^{16} Y^{20} - 154 X^{17} Y^{18} - 46 X^{16} Y^{19} + 29 X^{17} Y^{17} \\
& + 160 X^{16} Y^{18} + 4 X^{15} Y^{19} - 143 X^{16} Y^{17} - 22 X^{15} Y^{18} + 31 X^{16} Y^{16} + 127 X^{15} Y^{17} + 3 X^{14} Y^{18} \\
& -131 X^{15} Y^{16} - 10 X^{14} Y^{17} + 28 X^{15} Y^{15} + 85 X^{14} Y^{16} + X^{13} Y^{17} - 110 X^{14} Y^{15} + 4 X^{13} Y^{16} \\
& + 30 X^{14} Y^{14} + 55 X^{13} Y^{15} - 95 X^{13} Y^{14} + 5 X^{12} Y^{15} + 25 X^{13} Y^{13} + 31 X^{12} Y^{14} - 74 X^{12} Y^{13} \\
& + 6 X^{11} Y^{14} + 24 X^{12} Y^{12} + 15 X^{11} Y^{13} - 54 X^{11} Y^{12} + 3 X^{10} Y^{13} + 19 X^{11} Y^{11} - 3 X^{10} Y^{12} \\
& -35 X^{10} Y^{11} + 3 X^{9} Y^{12} + 19 X^{10} Y^{10} - 7 X^{9} Y^{11} - 23 X^{9} Y^{10} + 13 X^{9} Y^{9} - 10 X^{8} Y^{10} \\
& -10 X^{8} Y^{9} -  X^{7} Y^{10} + 12 X^{8} Y^{8} - 9 X^{7} Y^{9} - 3 X^{7} Y^{8} -  X^{6} Y^{9} + 9 X^{7} Y^{7} \\
& -8 X^{6} Y^{8} + X^{6} Y^{7} -  X^{5} Y^{8} + 8 X^{6} Y^{6} - 3 X^{5} Y^{7} + 2 X^{5} Y^{6} -  X^{4} Y^{7} \\
& + 4 X^{5} Y^{5} -  X^{4} Y^{6} + 4 X^{4} Y^{5} -  X^{3} Y^{6} + 4 X^{4} Y^{4} + 3 X^{3} Y^{4} + 2 X^{3} Y^{3} \\
& + 3 X^{2} Y^{3} + X^{2} Y^{2} + 2 X Y^{2} + Y + 1 
\end{align*}}

{\scriptsize\begin{align*}
F_{\Uni_3^2} =\, &
- X^{43} Y^{57} + 4 X^{41} Y^{53} -  X^{41} Y^{52} - 3 X^{40} Y^{53} - 2 X^{41} Y^{51} + 4 X^{40} Y^{52} + X^{41} Y^{50} 
 + 4 X^{40} Y^{51} \\& -  X^{39} Y^{52} - 3 X^{40} Y^{50} - 3 X^{39} Y^{51} + 5 X^{39} Y^{50} + 4 X^{38} Y^{51} - 4 X^{38} Y^{50} 
 - X^{37} Y^{51} - 2 X^{39} Y^{48} \\& - 3 X^{38} Y^{49} + 3 X^{37} Y^{50} + X^{39} Y^{47} -  X^{38} Y^{48} + 5 X^{37} Y^{49} 
 -3 X^{38} Y^{47} - 3 X^{37} Y^{48} -  X^{36} Y^{49} \\& + 3 X^{38} Y^{46} + 4 X^{36} Y^{48} - 6 X^{37} Y^{46} - 2 X^{36} Y^{47} 
 + X^{35} Y^{48} -  X^{37} Y^{45} -  X^{36} Y^{46} - 2 X^{35} Y^{47} \\& + 2 X^{37} Y^{44} + 4 X^{36} Y^{45} - 4 X^{35} Y^{46} 
 + 3 X^{34} Y^{47} - 3 X^{36} Y^{44} - 6 X^{35} Y^{45} + X^{34} Y^{46}  -
 X^{36} Y^{43} \\& + 8 X^{35} Y^{44} - 6 X^{34} Y^{45} 
 + X^{33} Y^{46} + 6 X^{35} Y^{43} - 9 X^{34} Y^{44} + 4 X^{33} Y^{45} -  X^{35}
 Y^{42}  - 4 X^{34} Y^{43} \\& - 6 X^{33} Y^{44} 
 - X^{32} Y^{45} - 2 X^{35} Y^{41} + 13 X^{34} Y^{42} + 4 X^{33} Y^{43} - 2
 X^{32} Y^{44} + 7 X^{34} Y^{41} - 11 X^{33} Y^{42}  \\&
 -6 X^{32} Y^{43}  + 3 X^{31} Y^{44} - 4 X^{34} Y^{40} - 2 X^{33} Y^{41} + 10 X^{32} Y^{42} - 4 X^{31} Y^{43} + 10 X^{33} Y^{40} 
 + 19 X^{32} Y^{41} \\& - 8 X^{31} Y^{42} + X^{30} Y^{43} - 2 X^{33} Y^{39} - 16 X^{32} Y^{40} - 21 X^{31} Y^{41} - 2 X^{30} Y^{42} 
 -7 X^{32} Y^{39} + 40 X^{31} Y^{40} \\& + 6 X^{30} Y^{41} -  X^{29} Y^{42} + 8 X^{32} Y^{38} + 19 X^{31} Y^{39} - 14 X^{30} Y^{40} 
 -3 X^{32} Y^{37} - 30 X^{31} Y^{38} - 9 X^{30} Y^{39} \\& - 9 X^{29} Y^{40} - 2 X^{28} Y^{41} + 2 X^{31} Y^{37} + 32 X^{30} Y^{38} 
 + 32 X^{29} Y^{39} + 2 X^{28} Y^{40} + 2 X^{30} Y^{37} - 10 X^{29} Y^{38} \\& - 19 X^{28} Y^{39} - 13 X^{30} Y^{36} - 26 X^{29} Y^{37} 
 + 18 X^{28} Y^{38} -  X^{27} Y^{39} + 18 X^{29} Y^{36}  + 42 X^{28} Y^{37} +
 X^{26} Y^{39} \\& + X^{30} Y^{34} - 3 X^{29} Y^{35} 
 -31 X^{28} Y^{36} - 15 X^{27} Y^{37} - 5 X^{26} Y^{38}  - 23 X^{28} Y^{35} + 29
 X^{27} Y^{36} + 14 X^{26} Y^{37} \\& + X^{29} Y^{33} 
 + 2 X^{28} Y^{34} + 17 X^{27} Y^{35} - 5 X^{25} Y^{37}  + 3 X^{28} Y^{33} - 30 X^{27} Y^{34} - 20 X^{26} Y^{35} + 5 X^{25} Y^{36} 
 \\&+ 3 X^{28} Y^{32} -  X^{27} Y^{33} + 20 X^{26} Y^{34}  + 33 X^{25} Y^{35} - 2 X^{24} Y^{36} - 12 X^{26} Y^{33} - 20 X^{25} Y^{34} 
 -9 X^{24} Y^{35} \\& -  X^{27} Y^{31} - 4 X^{26} Y^{32}  - 34 X^{25} Y^{33} + 15 X^{24} Y^{34} + 2 X^{23} Y^{35} + 3 X^{27} Y^{30} 
 + 12 X^{26} Y^{31} + 9 X^{25} Y^{32} \\& + 34 X^{24} Y^{33}  + X^{23} Y^{34} - 7 X^{26} Y^{30} + X^{25} Y^{31} - 48 X^{24} Y^{32} 
 -17 X^{23} Y^{33} + X^{22} Y^{34} + X^{26} Y^{29} \\& + 13 X^{25} Y^{30} - 14 X^{24} Y^{31} + 18 X^{23} Y^{32} + 10 X^{22} Y^{33} 
 + 4 X^{25} Y^{29} + 13 X^{24} Y^{30} - 4 X^{23} Y^{31}  \\& + 5 X^{22} Y^{32} - 5 X^{25} Y^{28} - 4 X^{24} Y^{29} - 20 X^{23} Y^{30} 
 -36 X^{22} Y^{31} - 3 X^{21} Y^{32} + 15 X^{24} Y^{28}  \\& + 20 X^{23} Y^{29} + 3 X^{22} Y^{30} + 26 X^{21} Y^{31} + X^{20} Y^{32} 
 -4 X^{24} Y^{27} - 8 X^{23} Y^{28} - 2 X^{22} Y^{29} \\& - 21 X^{21} Y^{30} - 3 X^{20} Y^{31} + 22 X^{22} Y^{28} - 22 X^{21} Y^{29} 
 + 3 X^{23} Y^{26} + 21 X^{22} Y^{27} + 2 X^{21} Y^{28} \\& + 8 X^{20} Y^{29} + 4 X^{19} Y^{30} -  X^{23} Y^{25} - 26 X^{22} Y^{26} 
 -3 X^{21} Y^{27} - 20 X^{20} Y^{28} - 15 X^{19} Y^{29} \\& + 3 X^{22} Y^{25} + 36 X^{21} Y^{26} + 20 X^{20} Y^{27} + 4 X^{19} Y^{28} 
 + 5 X^{18} Y^{29} - 5 X^{21} Y^{25} + 4 X^{20} Y^{26} \\& - 13 X^{19} Y^{27} - 4 X^{18} Y^{28} - 10 X^{21} Y^{24} - 18 X^{20} Y^{25} 
 + 14 X^{19} Y^{26} - 13 X^{18} Y^{27} -  X^{17} Y^{28} \\& -  X^{21} Y^{23} + 17 X^{20} Y^{24} + 48 X^{19} Y^{25} -  X^{18} Y^{26} 
 + 7 X^{17} Y^{27} -  X^{20} Y^{23} - 34 X^{19} Y^{24} \\& - 9 X^{18} Y^{25} - 12 X^{17} Y^{26} - 3 X^{16} Y^{27} - 2 X^{20} Y^{22} 
 -15 X^{19} Y^{23} + 34 X^{18} Y^{24} + 4 X^{17} Y^{25} \\& + X^{16} Y^{26} + 9 X^{19} Y^{22} + 20 X^{18} Y^{23} + 12 X^{17} Y^{24} 
 + 2 X^{19} Y^{21} - 33 X^{18} Y^{22} - 20 X^{17} Y^{23} \\& + X^{16} Y^{24} - 3 X^{15} Y^{25} - 5 X^{18} Y^{21} + 20 X^{17} Y^{22} 
 + 30 X^{16} Y^{23} - 3 X^{15} Y^{24} + 5 X^{18} Y^{20} \\& - 17 X^{16} Y^{22} - 2 X^{15} Y^{23} -  X^{14} Y^{24} - 14 X^{17} Y^{20} 
 -29 X^{16} Y^{21} + 23 X^{15} Y^{22} + 5 X^{17} Y^{19} \\& + 15 X^{16} Y^{20} + 31 X^{15} Y^{21} + 3 X^{14} Y^{22} -  X^{13} Y^{23} 
 - X^{17} Y^{18} - 42 X^{15} Y^{20} - 18 X^{14} Y^{21} \\& + X^{16} Y^{18} - 18 X^{15} Y^{19} + 26 X^{14} Y^{20} + 13 X^{13} Y^{21} 
 + 19 X^{15} Y^{18} + 10 X^{14} Y^{19} - 2 X^{13} Y^{20} \\& - 2 X^{15} Y^{17} - 32 X^{14} Y^{18} - 32 X^{13} Y^{19} - 2 X^{12} Y^{20} 
 + 2 X^{15} Y^{16} + 9 X^{14} Y^{17} + 9 X^{13} Y^{18} \\& + 30 X^{12} Y^{19} + 3 X^{11} Y^{20} + 14 X^{13} Y^{17} - 19 X^{12} Y^{18} 
 -8 X^{11} Y^{19} + X^{14} Y^{15} - 6 X^{13} Y^{16} \\&  - 40 X^{12} Y^{17} + 7 X^{11} Y^{18} + 2 X^{13} Y^{15} + 21 X^{12} Y^{16} 
 + 16 X^{11} Y^{17} + 2 X^{10} Y^{18} -  X^{13} Y^{14}  + 8 X^{12} Y^{15} \\& - 19 X^{11} Y^{16} - 10 X^{10} Y^{17} + 4 X^{12} Y^{14} 
 -10 X^{11} Y^{15} + 2 X^{10} Y^{16} + 4 X^{9} Y^{17}  - 3 X^{12} Y^{13} \\& + 6 X^{11} Y^{14} + 11 X^{10} Y^{15} - 7 X^{9} Y^{16} 
 + 2 X^{11} Y^{13} - 4 X^{10} Y^{14} - 13 X^{9} Y^{15}  + 2 X^{8} Y^{16}  +
 X^{11} Y^{12} \\& + 6 X^{10} Y^{13} + 4 X^{9} Y^{14} 
 + X^{8} Y^{15} - 4 X^{10} Y^{12} + 9 X^{9} Y^{13} - 6 X^{8} Y^{14}  -  X^{10}
 Y^{11}  + 6 X^{9} Y^{12} - 8 X^{8} Y^{13} \\&
 + X^{7} Y^{14} -  X^{9} Y^{11} + 6 X^{8} Y^{12} + 3 X^{7} Y^{13} - 3 X^{9}
 Y^{10}  + 4 X^{8} Y^{11}  - 4 X^{7} Y^{12} 
 -2 X^{6} Y^{13} + 2 X^{8} Y^{10} \\& + X^{7} Y^{11} + X^{6} Y^{12} -  X^{8} Y^{9} +
 2 X^{7} Y^{10}   + 6 X^{6} Y^{11} 
 -4 X^{7} Y^{9} - 3 X^{5} Y^{11} + X^{7} Y^{8} + 3 X^{6} Y^{9} + 3 X^{5} Y^{10}
 \\& - 5 X^{6} Y^{8} + X^{5} Y^{9} 
 - X^{4} Y^{10}  - 3 X^{6} Y^{7} + 3 X^{5} Y^{8} + 2 X^{4} Y^{9} + X^{6} Y^{6} + 4 X^{5} Y^{7} - 4 X^{5} Y^{6} 
 -5 X^{4} Y^{7} \\& + 3 X^{4} Y^{6} + 3 X^{3} Y^{7}  + X^{4} Y^{5} - 4 X^{3} Y^{6} -  X^{2} Y^{7} - 4 X^{3} Y^{5} 
 + 2 X^{2} Y^{6} + 3 X^{3} Y^{4} + X^{2} Y^{5} - 4 X^{2} Y^{4} + 1 
\end{align*}}

\section{Formulae for local graded subalgebra zeta functions}
\label{app:grsub}

{\small
\begin{equation}
  \label{subalgebras:m5_3_1}
  \begin{aligned}
    W_{531} = \,& \bigl(- X^{5} Y^{18} -  X^{5} Y^{16} -  X^{5} Y^{15} -  X^{4} Y^{16} -  X^{5} Y^{14} -  X^{4} Y^{15} + 2 X^{4} Y^{13} 
    + X^{3} Y^{14} \\& + X^{4} Y^{12} + 2 X^{3} Y^{13} + X^{4} Y^{11} + X^{3} Y^{12} + X^{2} Y^{13} + X^{4} Y^{10} 
    + X^{3} Y^{11} + X^{4} Y^{9} \\&+ 3 X^{3} Y^{10} + 2 X^{3} Y^{9} + X^{2} Y^{10} -
    X^{3} Y^{8} - 2 X^{2} Y^{9} 
     -3 X^{2} Y^{8} -  X Y^{9} -  X^{2} Y^{7} \\& -  X Y^{8} -  X^{3} Y^{5} -  X^{2} Y^{6} -  X Y^{7} 
    -2 X^{2} Y^{5} -  X Y^{6} -  X^{2} Y^{4} - 2 X Y^{5} + X Y^{3} \\& + Y^{4} + X Y^{2} 
    + Y^{3} + Y^{2} + 1 \bigr)\\&/
    \bigl(
    {\left(1 - X Y^{5}\right)}
    {\left(1 - X^{2} Y^{3}\right)}
    {\left(1 - X^2 Y^{4}\right)}
    {\left(1 - X Y^{2}\right)}
    {\left(1 - X Y\right)}
    {\left(1-Y^5\right)}
    {\left(1 - Y^2\right)}
    {\left(1 - Y\right)}\bigr)
\end{aligned}
\end{equation}
}

{\small
  \begin{equation}
  \label{subalgebras:m5_4_1}
  \begin{aligned}
    W_{541} =\, &
    \bigl(- X^{3} Y^{21} -  X^{3} Y^{20} - 3 X^{3} Y^{19} - 5 X^{3} Y^{18} - 7 X^{3} Y^{17} + X^{2} Y^{18} - 8 X^{3} Y^{16}
    + 4 X^{2} Y^{17} \\& - 7 X^{3} Y^{15} + 9 X^{2} Y^{16} - 6 X^{3} Y^{14} + 16 X^{2} Y^{15} - 6 X^{3} Y^{13} + 19 X^{2} Y^{14} 
    - X Y^{15} - 4 X^{3} Y^{12} \\& + 21 X^{2} Y^{13} - 4 X Y^{14} - 3 X^{3} Y^{11} + 21 X^{2} Y^{12} - 8 X Y^{13} 
    - X^{3} Y^{10} + 20 X^{2} Y^{11} - 14 X Y^{12} \\&  + 18 X^{2} Y^{10} - 18 X Y^{11} + 14 X^{2} Y^{9} - 20 X Y^{10} 
    + Y^{11} + 8 X^{2} Y^{8} - 21 X Y^{9}  + 3 Y^{10} \\& + 4 X^{2} Y^{7} - 21 X Y^{8} + 4 Y^{9} 
    + X^{2} Y^{6} - 19 X Y^{7} + 6 Y^{8} - 16 X Y^{6} + 6 Y^{7} - 9 X Y^{5} \\& + 7 Y^{6} 
    -4 X Y^{4} + 8 Y^{5} -  X Y^{3} + 7 Y^{4} + 5 Y^{3} + 3 Y^{2} + Y + 1 \bigr) \\&
    /\bigl(
    {\left(1 - X Y^{4} \right)} {\left(1 - X Y^{3} \right)}
    {\left(1 - X Y^{2}\right)} {\left(1 - X Y\right)} 
    {\left(1 - Y^7\right)}
    {\left(1 - Y^{4}\right)} {\left(1 - Y^3\right)} {\left(1-Y^2\right)}
    \bigr)
  \end{aligned}
\end{equation}}

{\small
  \begin{equation}
    \label{subalgebras:m6_2_1}
    \begin{aligned}
      W_{621} =\, &
      \bigl(X^{4} Y^{8} + X^{4} Y^{6} + X^{3} Y^{6} - X^{3} Y^{5} + X^{2}
      Y^{6} - X^{3} Y^{4} - X^{2} Y^{5} - X^{3} Y^{3} - X Y^{5} \\&- X^{2}
      Y^{3} - X Y^{4} + X^{2} Y^{2} - X Y^{3} + X Y^{2} + Y^{2} + 1\bigr)
      \left(X Y^{2} + 1\right)  \\&/ \bigl(
      {\left(1 - X^{2} Y^{3}\right)}
      {\left(1 - X Y^{3} \right)}
      {\left(1 - X^{3} Y^{2}\right)}
      {\left(1 - X^2 Y^2\right)}
      {\left(1 - X^{2} Y\right)}
      {\left(1 - X Y\right)}
      {\left(1 - Y^3\right)}
      {\left(1 - Y \right)}
      \bigr)
  \end{aligned}
\end{equation}}

{\small
  \begin{equation}
    \label{subalgebras:m6_2_3}
    \begin{aligned}
      W_{623} =\, &
      \bigl(X^{9} Y^{14} + X^{9} Y^{13} + X^{9} Y^{12} - 3 X^{8} Y^{11} - 2
      X^{8} Y^{10} -  X^{6} Y^{11} -  X^{6} Y^{10} + X^{7} Y^{8} - 3 X^{6} Y^{9}
      \\ & -  X^{5} Y^{9} -  X^{4} Y^{10} + X^{6} Y^{7} + 2 X^{5} Y^{8} + X^{6}
      Y^{6}
      + 2 X^{5} Y^{7} + 2 X^{4} Y^{8} + 2 X^{5} Y^{6} + 2 X^{4} Y^{7} 
      \\ &
      + X^{3} Y^{8} + 2 X^{4} Y^{6} + X^{3} Y^{7} - X^{5} Y^{4} -  X^{4} Y^{5}
        - 3 X^{3} Y^{5} + X^{2} Y^{6} -  X^{3} Y^{4} -  X^{3} Y^{3} \\& - 2 X Y^{4}
        -3 X Y^{3} + Y^{2} + Y + 1 \bigr) / \bigl(
    {\left(1 - X^{4} Y^{3}\right)}
    {\left(1 - X^3Y^3 \right)}
    {\left(1 - X Y^{3}\right)}
    {\left(1 - X Y^{2}\right)}^{2}  \\& \quad\quad\quad\quad\quad\quad\quad\quad\quad\quad\quad\quad\quad\quad\quad \times
    {\left(1 - X^{3} Y\right)} 
    {\left(1 - X^{2} Y\right)}
    {\left(1 - X Y\right)}
    {\left(1 - Y^2\right)}^{2}
    \bigr)
  \end{aligned}
\end{equation}}

{\small \begin{equation}
  \label{subalgebras:m6_3_1}
  \begin{aligned}
    W_{631} = \,& \bigl(
    - X^{5} Y^{18} -  X^{5} Y^{16} -  X^{5} Y^{15} -  X^{4} Y^{16} -  X^{5}
    Y^{14} -  X^{4} Y^{15} + 2 X^{4} Y^{13}  + X^{3} Y^{14} \\& + X^{4} Y^{12} 
    + 2 X^{3} Y^{13} + X^{4} Y^{11} + X^{3} Y^{12} + X^{2} Y^{13} + X^{4} Y^{10}
    + X^{3} Y^{11} + X^{4} Y^{9} \\& + 3 X^{3} Y^{10}  + 2 X^{3} Y^{9}   + X^{2}
    Y^{10} -  X^{3} Y^{8} - 2 X^{2} Y^{9} -3 X^{2} Y^{8} -  X Y^{9} -  X^{2}
    Y^{7} -  X Y^{8} \\&  -  X^{3} Y^{5} -  X^{2} Y^{6}
    - X Y^{7}  -2 X^{2} Y^{5} -  X Y^{6} -  X^{2} Y^{4} - 2 X Y^{5} + X Y^{3} + Y^{4} + X
    Y^{2} \\& + Y^{3} + Y^{2} + 1 \bigr) / \bigl(
    {\left(1 - X Y^{5} \right)} 
    {\left(1 - X^2 Y^{4}\right)} 
    {\left(1 - X^{2} Y^{3}\right)}
    {\left(1 - X Y^{2}\right)}
    {\left(1 - X^{2} Y\right)} \\& \quad\quad\quad\quad\quad\quad\quad\quad \times
    {\left(1 - X Y\right)} 
    {\left(1 - Y^5\right)}
    {\left(1 - Y^2\right)}
    {\left(1 - Y\right)}
    \bigr)
  \end{aligned}
\end{equation}}

{\small \begin{equation}
  \label{subalgebras:m6_3_2}
  \begin{aligned}
    W_{632} = \,& \bigl(
    X^{3} Y^{16} + 2 X^{3} Y^{15} + 4 X^{3} Y^{14} + 7 X^{3} Y^{13} + X^{2}
    Y^{14} + 10 X^{3} Y^{12} + 2 X^{2} Y^{13} + 11 X^{3} Y^{11} \\&
    - X Y^{13}   + 10 X^{3} Y^{10} - 3 X^{2} Y^{11} - 3 X Y^{12} + 7 X^{3} Y^{9}
    - 8 X^{2} Y^{10} -6 X Y^{11} + 4 X^{3} Y^{8} \\&
    - 11 X^{2} Y^{9} - 9 X Y^{10} + 2 X^{3} Y^{7} - 11 X^{2} Y^{8} - 10 X Y^{9}
    + Y^{10} + X^{3} Y^{6} - 10 X^{2} Y^{7} \\& - 11 X Y^{8} + 2 Y^{9} - 9 X^{2}
    Y^{6} - 11 X Y^{7} 
    + 4 Y^{8} - 6 X^{2} Y^{5} - 8 X Y^{6} + 7 Y^{7} - 3 X^{2} Y^{4} \\& - 3 X
    Y^{5} + 10 Y^{6} 
    - X^{2} Y^{3} + 11 Y^{5} + 2 X Y^{3} + 10 Y^{4} + X Y^{2} + 7 Y^{3} + 4
    Y^{2}  + 2 Y + 1 \bigr) \\& \times (1-Y)/\bigl(
    {\left(1 - X Y^{3}\right)}
    {\left(1 - X^2 Y^2\right)}
    {\left(1 - X Y^{2}\right)}
    {\left(1 - X^{2} Y\right)}
    {\left(1 - X Y\right)} \\& \quad\quad\quad\quad\quad\times
    {\left(1 - Y^5\right)}
    {\left(1 - Y^{4}\right)}
    {\left(1 - Y^3\right)}
    {\left(1 - Y^2\right)}
    \bigr)
  \end{aligned}
\end{equation}}

{\small \begin{equation}
  \label{subalgebras:m6_3_3}
  \begin{aligned}
    W_{633} = & \bigl(
    X^{3} Y^{14} + X^{3} Y^{13} + 3 X^{3} Y^{12} + 4 X^{3} Y^{11} + X^{2} Y^{12}
    + 5 X^{3} Y^{10} + 4 X^{3} Y^{9}  - X^{2} Y^{10}  \\& -  X Y^{11} 
        + 3 X^{3}  Y^{8} 
    - 3 X^{2} Y^{9} -  X Y^{10} + X^{3} Y^{7} - 5
    X^{2} Y^{8} -4 X Y^{9} + X^{3} Y^{6} - 5 X^{2} Y^{7} 
     \\& - 4 X Y^{8} - 4 X^{2} Y^{6} - 5 X Y^{7} + Y^{8}  -4 X^{2} Y^{5} - 5 X Y^{6}
    + Y^{7} -  X^{2} Y^{4} - 3 X Y^{5} + 3 Y^{6}   \\& -  X^{2} Y^{3} - X Y^{4} 
    + 4 Y^{5} + 5 Y^{4} + X Y^{2} + 4 Y^{3} + 3 Y^{2} +
    Y + 1 \bigr) \\& / \bigl(
    {\left(1 - X Y^{3}\right)}
    {\left(1 - X^2 Y^2\right)}
    {\left(1 - X Y^{2}\right)}
    {\left(1 - X^{2} Y\right)}
    {\left(1 - X Y\right)}
    {\left(1 - Y^3\right)}
    {\left(1 - Y^4\right)}^{2}
    \bigr)
  \end{aligned}
\end{equation}}

{\small \begin{equation}
  \label{subalgebras:m6_4_3}
  \begin{aligned}
W_{643} = & \bigl(
- X^{3} Y^{21} -  X^{3} Y^{20} - 3 X^{3} Y^{19} - 5 X^{3} Y^{18} - 7 X^{3}
Y^{17} + X^{2} Y^{18} - 8 X^{3} Y^{16} + 4 X^{2} Y^{17} 
\\& - 7 X^{3} Y^{15} + 9 X^{2} Y^{16} - 6 X^{3} Y^{14} + 16 X^{2} Y^{15} - 6
X^{3} Y^{13} + 19 X^{2} Y^{14} - X Y^{15} - 4 X^{3} Y^{12} \\& + 21 X^{2} Y^{13} - 4 X Y^{14} - 3 X^{3} Y^{11} + 21 X^{2}
Y^{12} - 8 X Y^{13} - X^{3} Y^{10} + 20 X^{2} Y^{11} - 14 X Y^{12}    \\&  + 18 X^{2} 
Y^{10}  - 18 X Y^{11}+ 14 X^{2} Y^{9} - 20 X Y^{10}  + Y^{11} + 8 X^{2} Y^{8}
 - 21 X Y^{9} + 3 Y^{10} \\& + 4 X^{2} Y^{7} - 21 X Y^{8} + 4 Y^{9}  + X^{2}
Y^{6} - 19 X Y^{7} + 6 Y^{8} - 16 X Y^{6}  + 6 Y^{7} - 9 X Y^{5} + 7 Y^{6} 
\\&  -4 X Y^{4} + 8 Y^{5} -  X Y^{3} + 7 Y^{4} + 5 Y^{3} + 3 Y^{2} + Y 
 + 1\bigr) \\& / \bigl(
 {\left(1 - X Y^{4} \right)}
 {\left(1 - X Y^{3}\right)}
 {\left(1 - X Y^{2}\right)}
 {\left(1 - X^{2} Y\right)}
 {\left(1 - X Y\right)} \\& \quad \times
 {\left(1 - Y^{7}\right)}
 {\left(1 - Y^{4}\right)}
 {\left(1 - Y^3\right)}
 {\left(1 - Y^2\right)}
\bigr)
\end{aligned}
\end{equation}}

{
  \bibliographystyle{abbrv}
  \tiny
  \bibliography{padzeta}
}

\end{document}